\documentclass[11 pt]{amsart}

\usepackage{hyperref}

\usepackage{changebar}
\usepackage{latexsym}
\usepackage{amssymb, amsmath}
\usepackage{amsthm}

\usepackage{geometry}
\usepackage{comment}
\usepackage[T1]{fontenc}
\usepackage{lmodern}
\usepackage{amssymb}

\usepackage[stretch=10]{microtype}

\usepackage{tikz}
\usetikzlibrary{calc}
\usetikzlibrary{decorations.pathreplacing,angles,quotes}

\usepackage{fixltx2e}

\newtheorem{theorem}{Theorem}[section]
\newtheorem{lemma}[theorem]{Lemma}

\newtheorem{cor}[theorem]{Corollary}
\newtheorem{sublem}[theorem]{Sublemma}
\newtheorem{proposition}[theorem]{Proposition}

\theoremstyle{definition}
\newtheorem{definition}[theorem]{Definition}

\theoremstyle{remark}
\newtheorem{remark}[theorem]{Remark}

\numberwithin{equation}{section}

\newcommand{\cA}{\mathcal{A}}
\newcommand\cB{{\mathcal B}}
\newcommand{\BB}{\mathcal{B}}
\newcommand{\cC}{\mathcal{C}}
\newcommand{\cE}{\mathcal{E}}
\newcommand{\cF}{\mathcal{F}}
\newcommand{\cG}{\mathcal{G}}
\newcommand{\cI}{\mathcal{I}}
\newcommand{\cK}{\mathcal{K}}
\newcommand{\cL}{\mathcal{L}}
\newcommand{\LL}{\mathcal{L}}
\newcommand{\cM}{\mathcal{M}}
\newcommand{\cN}{\mathcal{N}}
\newcommand{\cP}{\mathcal{P}}
\newcommand{\hP}{\mathring{\cP}}
\newcommand{\cQ}{\mathcal{Q}}
\newcommand\cR{{\mathcal R}}
\newcommand{\cS}{\mathcal{S}}
\newcommand{\cW}{\mathcal{W}}
\newcommand{\UU}{\mathcal{U}}

\newcommand{\bH}{\mathbb{H}}
\newcommand\bN{{\mathbb N}}
\newcommand\integer{{\mathbb Z}}

\newcommand{\bx}{\bar{x}}
\newcommand{\bpsi}{\overline{\psi}}
\newcommand{\bvf}{\overline{\vf}}
\newcommand{\bTheta}{\overline{\Theta}}
\newcommand{\bGamma}{\overline{\Gamma}}
\newcommand{\bd}{\bar{d}}

\newcommand{\tpsi}{\widetilde{\psi}}

\newcommand{\hW}{\widehat{\cW}}
\newcommand{\hLambda}{\hat{\Lambda}}

\newcommand{\Sh}{{I_u}}
\newcommand{\Lo}{{L_u}}
\newcommand{\Los}{{L_s}}
\newcommand{\tmax}{\tau_{\max}}
\newcommand{\tmin}{\tau_{\min}}
\newcommand{\Fix}{{\mbox{Fix}\, }}
\newcommand{\hspan}{h_{\scriptsize{\mbox{span}}}}
\newcommand{\hsep}{h_{\scriptsize{\mbox{sep}}}}
\newcommand{\htop}{h_{\scriptsize{\mbox{top}}}}
\newcommand{\musrb}{\mu_{\tiny{\mbox{SRB}}}}
\newcommand{\hatmusrb}{\hat \mu_{\tiny{\mbox{SRB}}}}
\newcommand{\Det}{{\mbox{Det}\, }}
\newcommand{\vf}{\varphi}

\newcommand{\ve}{\varepsilon}

\newcommand{\inter}{\mbox{int}}
\newcommand{\diam}{\mbox{diam}}

\begin{document}

\title[Measure of Maximal Entropy for Sinai Billiard Maps]{On the Measure of Maximal Entropy for Finite Horizon Sinai Billiard Maps}

\author{Viviane Baladi}
\address{CNRS, Institut de Math\'ematiques de Jussieu (IMJ-PRG), 
Sorbonne Universit\'e,
4, Place Jussieu, 75005 Paris, France}
\curraddr
{Laboratoire de Probabilit\'es, Statistique et Mod\'elisation (LPSM),  
CNRS, Sorbonne Universit\'e, Universit\'e de Paris,
4, Place Jussieu, 75005 Paris, France}
\email{baladi@lpsm.paris}

\author{Mark F. Demers}
\address{Department of Mathematics, Fairfield University, Fairfield CT 06824, USA}
\email{mdemers@fairfield.edu}

\thanks{
Part of this work was carried  out during visits of  MD to ENS Ulm/IMJ-PRG Paris in 2016 and 
to IMJ-PRG in 2017 and 2018,  during a visit of VB to Fairfield University in 2018,
and during the 2018 workshops New Developments in Open Dynamical Systems and their Applications in BIRS Banff, and Thermodynamic Formalism in Dynamical Systems in ICMS Edinburgh.  We are grateful to F.~Ledrappier, C. Matheus,  Y. Lima,  S. Luzzatto, P.-A. Guih\'eneuf, G. Forni, B. Fayad, S. Cantat, R. Dujardin, J.~Buzzi, P. B\'alint, and J. De Simoi for useful comments, to V. Bergelson for encouraging us to
establish the Bernoulli property, and to V.~Climenhaga for insightful comments which
spurred us on to obtain uniqueness. We thank the anonymous referees for many constructive suggestions. MD was partly supported by NSF grants DMS 1362420
and DMS 1800321. VB's research is supported
by the European Research Council (ERC) under the European Union's Horizon 2020 research and innovation programme (grant agreement No 787304).}

\begin{abstract}
The   Sinai billiard map $T$
on the two-torus, i.e., the periodic Lorentz gas, is a discontinuous map.
Assuming finite horizon, we propose a definition $h_*$ for the topological entropy of  $T$.
We prove that $h_*$  is not smaller than the value given by the variational principle,
and that it is equal to the definitions of Bowen using spanning or separating sets.
Under a mild condition of sparse recurrence to the singularities, we get more:
First, using a transfer operator acting on a space of anisotropic
distributions, we   construct an invariant probability measure $\mu_*$ of maximal entropy for $T$ (i.e.,
$h_{\mu_*}(T)=h_*$),  we show that $\mu_*$ has full support and is
Bernoulli,   and we prove that $\mu_*$ is the unique measure
of maximal entropy,
and that it is different from the
smooth invariant measure except if all non grazing periodic orbits  have
multiplier equal to $h_*$. 
Second, $h_*$ is equal to the Bowen--Pesin--Pitskel topological entropy of
the restriction of $T$ to a non-compact  domain of continuity. Last, applying
results of Lima and Matheus, as upgraded by Buzzi, the map $T$
has at least $C e^{nh_*}$ periodic points of period $n$ for all $n \in \mathbb{N}$.
\end{abstract}


\date{Received by the editors August 25, 2018,  and, in revised form, August 19, 2019.  \\
\phantom{ElecDate}Electronically published  January 6, 2020. 
\url{https://doi.org/10.1090/jams/939} First published in Journal of the American Mathematical Society 
in 2020, published by American Mathematical 
Society. \\
License or copyright restrictions may apply to redistribution;  \url{https://www.ams.org/journal-terms-of-use}
}

\maketitle


\section{Introduction}
\label{sec:intro}

\subsection{Bowen--Margulis Measures and Measures of Maximal Entropy}

Half a century ago\footnote{See \cite{Mar} for the full  english text.}, Margulis  \cite{Mar0} 
proved in his dissertation the following analogue of the prime number theorem for the 
closed geodesics $\Gamma$ of a compact
manifold of strictly negative (not necessarily constant) curvature: Let 
$h>0$ be the topological entropy of the geodesic flow; then,
\begin{equation}\label{BM}
\# \{ \Gamma \mbox{ such that }  |\Gamma| \le L \}\sim_{L \to \infty} 
\frac{e^{hL}}{hL} \,  .
\end{equation} 
(I.e. $\lim_{L \to \infty}(hL e^{-hL} \# \{ \Gamma \mbox{ such that }  |\Gamma| \le L \})=1$.)
The main ingredient in the proof is an
 invariant probability measure for the flow,   the Margulis (or Bowen--Margulis \cite{Bow}) measure  $\mu_{\mbox{\tiny{top}}}$.
 This measure --- which coincides with volume in constant curvature,
 but not in general --- is mixing (thus ergodic), and it can
 be written as a local product of its stable and unstable
 conditionals, where  these conditional measures 
 scale by $e^ {\pm ht}$ under the action of the  flow.
 These properties were essential to establish \eqref{BM}.
The measure $\mu_{\mbox{\tiny{top}}}$ enjoys other remarkable properties, such as equidistribution of closed geodesics.
Finally, the measure $\mu_{\mbox{\tiny{top}}}$ is  the unique measure of
maximal entropy of the flow, that is, the unique  invariant measure with Kolmogorov entropy
equal to the topological entropy of the flow.

These results were  extended
to more general smooth uniformly hyperbolic flows and diffeomorphisms, using the thermodynamic
formalism of Bowen, Ruelle, and Sinai. In particular Parry--Pollicott \cite{PaPo83}
obtained a different proof of  \eqref{BM} using  a dynamical zeta function.
Later, based on  Dolgopyat's \cite{Do}  groundbreaking thesis
(proving exponential mixing for the measure and giving a pole-free vertical
strip  for a zeta function),  exponential
error terms were obtained \cite{PS} for the counting asymptotics \eqref{BM} in the case
of surfaces or $1/4$-pinched manifolds. 
Using \cite{Do, PS}, Stoyanov \cite{St} obtained exponential error terms
for the closed orbits of a class of open planar convex billiards, which are smooth hyperbolic flows on 
their nonwandering set, a compact
(fractal) invariant set. We refer to  Sharp's survey  in \cite{Mar} for 
more counting results in uniformly hyperbolic dynamics.
We just mention here that, for some Axiom A flows with slower (non-exponential)
mixing rates,
 it is possible \cite{PS2} to 
 get (weaker) error terms, of the form
  $\frac{e^{hL}}{hL}(1+O(L^{-\delta}))$, for the asymptotics \eqref{BM}, by 
 exploiting relevant operator bounds from
 \cite{Do2} (corresponding
 to   a resonance
 free domain for the transfer operator). This may be relevant
 for the Sinai billiards considered in the present work, as we do not expect
 them to mix exponentially fast for the measure of maximal entropy
 without additional assumptions.

Entropy is a fundamental invariant in dynamics and the study of measures
of maximal entropy is a topic in its own right \cite{Ka}. Let us just mention here the discrete-time
 analogue of the counting theorem \eqref{BM} which has been established in several
 situations (see also \cite{Ka0} for more general results): 
Let $h>0$ be the topological entropy of  uniformly hyperbolic (Axiom A)
diffeomorphism $T$, set $\Fix T^{m}=\{ x \, : \, T^m(x)=x\}$; then Bowen showed \cite{Bo00} 
that $\lim_{m \to \infty} \frac{1}{m}
\log \# \Fix T^m =  h$. In fact \cite{Bow3}, there is a constant
$C>0$ so that
\begin{equation}\label{PNM}
C   e^{  h m} \le  \# \Fix T^m \le C^{-1}  e^{  h m}\, ,
\qquad \forall m\ge 1 \, .
\end{equation}

Uniqueness of the measure of maximal
entropy has been extended to some geodesic
flows in non-positive curvature (i.e. weakening the hyperbolicity
requirement). The breakthrough
result of   Knieper \cite{Kn} for compact rank $1$ manifolds has been recently given a new  dynamical
proof \cite{BCFT} (using  Bowen's ideas as revisited by Climenhaga and Thompson). This
is currently a very active topic, see e.g. \cite{CliWar}.

\medskip

The present paper is devoted to the study of the measure of maximal
entropy in a situation where uniform hyperbolicity holds, but the dynamics
is not smooth: The singular set $\cS_{\pm 1}$,
i.e. those points where the map $T$
(or the flow $\Phi$) or its inverse are not $C^1$, is not empty.  In this
setting, the following integrability condition 
is crucial:
 \begin{equation}\label{adapt}
 \int |\log d(x, \cS_{\pm 1})| \, d\mu_{\mbox{\tiny{top}}}<\infty\, . 
 \end{equation}
 Following Lima--Matheus \cite{LM}, we shall say that a measure 
 $\mu$
 satisfying
 the above integrability condition for a map $T$ is $T$-adapted.
 
 Condition \eqref{adapt} is prevalent in the rich literature
about measures of maximal entropy for meromorphic maps of a compact K\"ahler manifold
(see the survey \cite{Fr}, and e.g. \cite{DG3} and references therein)
such as
birational mappings.  In this work, we are concerned with a different
class of dynamics with singularities: the dispersing billiards introduced by Sinai \cite{Sin70} on the
two-torus.
A Sinai billiard on the torus is the periodic case  of the planar Lorentz gas (1905) model
for the motion of a single dilute electron in a metal. The scatterers (corresponding
to the atoms of the metal) are assumed
to be strictly convex, but they are not necessarily perfect discs.
Such billiards have become foundational models in mathematical physics.
 
The Sinai billiard flow is continuous, but\footnote{In contrast,  
open billiards in the plane which satisfy a non-eclipsing condition do not have any
singularities on their nonwandering set, so that they fit in the Axiom~A category
\cite{St}.} not differentiable: the ``grazing'' orbits (those which are tangent to  a scatterer)
lead to singularities. 
Nevertheless, existence of a measure of maximal entropy for the billiard flow is
granted,  thanks to hyperbolicity. The topological entropy has been  studied for the billiard flow \cite{BFK}. 
However, uniqueness of the measure of maximal entropy,
as well as mixing and 
the adapted condition \eqref{adapt} 
are not known.  Since the transfer operator techniques we use
are simpler to implement  in the discrete-time case,
we study in this paper the Sinai billiard map, which is  the return map 
of the single point particle to the
scatterers.

Sinai billiard maps preserve a smooth invariant measure $\musrb$ which has been studied
extensively:  With respect to $\musrb$, the billiard is uniformly
hyperbolic, ergodic, K-mixing and Bernoulli \cite{Sin70, gallavotti, SC87, ChH}.
The measure  $\musrb$ is $T$-adapted \cite{katok strelcyn}.
Moreover, this measure enjoys exponential decay of correlations \cite{Y98}
and a host of other limit theorems (see e.g. \cite[Chapter 7]{chernov book} or \cite{demzhang11}).
The billiard has many periodic orbits and thus many other ergodic invariant
measures $\mu$, but there are very few results regarding other invariant measures and they
apply only to perturbations of $\musrb$ \cite{CWZ, dzr}.  Since the billiard map
 is discontinuous, the standard results \cite{walters} 
guaranteeing that the supremum of Kolmogorov entropy is attained and coincides with the
topological entropy do not hold.
It is natural to ask whether a measure of maximal
entropy 
exists, and, in the affirmative, whether it is unique, ergodic, and mixing.

Another natural goal is to establish \eqref{PNM}.
Chernov asked (see \cite[Problems 5 and 6]{Gu}) 
whether a slightly weaker property than \eqref{PNM}, namely
$$\lim_{m \to \infty} \frac{1}{m}
\log \# \Fix T^m  =  h_{\mbox{\tiny{top}}}\, ,
$$ holds.
(Chernov \cite{chernov} showed that $\liminf_{m \to \infty} \frac{1}{m}
\log \# \Fix T^m \ge  h_{\musrb}$. For a related  class of billiards,
 Stoyanov \cite{St0} found  finite constants $C$ and $H$ so that
$\# \Fix T^m \le  C e^{H m}$ for all $m \ge 1$.)

A detailed knowledge of the measure of maximal entropy, and the
techniques developed to obtain this information, could  potentially allow us not only to 
establish \eqref{PNM} 
for the billiard
map, but  also eventually to prove a prime number asymptotic of the form \eqref{BM}
for the billiard flow.  Although lifting a measure of maximal entropy for the map
should not directly give a measure of maximal entropy for the flow, we believe
that the techniques
and results of the present paper will be instrumental in understanding
the measure of maximal entropy of the billiard flow.

\medskip
We list our results in Section~\ref{ssec:map}. In a nutshell,
for all finite horizon planar Sinai billiards $T$
satisfying a (mild) condition of ``sparse recurrence'' to the singular set, we construct
a measure of maximal entropy, we show that it is unique, mixing (even Bernoulli),
that it has full support, and that it is $T$-adapted. Our results combined with
those of Lima--Matheus \cite{LM} and a very recent preprint of Buzzi \cite{Bu} give
$C>0$ such that the lower bound in \eqref{PNM} holds.

Finally, we mention that
our technique  for constructing and studying the invariant measure,
which uses transfer operators but avoids coding, is 
reminiscent both of the construction of Margulis \cite{Mar} and
  the techniques of ``laminar currents'' introduced by Dujardin  for birational mappings 
 \cite{Du} (see also \cite{DG3}).


\subsection{Summary of Main Results}
\label{ssec:map}

A Sinai  billiard table $Q$  on the two-torus $\mathbb{T}^2$
is a set  $
 Q=\mathbb{T}^2 \setminus B
 $,
 with   $B=\cup_{i=1}^D B_i$
for some finite number $D\ge 1$ of 
pairwise disjoint closed domains $B_i$ with $C^3$ boundaries 
having strictly positive curvature (in particular, the domains are strictly convex).
The sets $B_i$ are called  scatterers;
see Figure~\ref{fig:tables} for some common examples. The
billiard flow  is 
the motion of a point particle traveling in $Q$ at unit speed 
 and
undergoing elastic (i.e., specular) reflections at the boundary of
the scatterers.   (By definition, at a tangential --- also called grazing --- collision, the
reflection does not change the direction of the particle.)
This is also called a periodic Lorentz gas. 
As mentioned above,
a key feature is that, 
although  the billiard flow is continuous if one identifies
outgoing and incoming angles,  the tangential collisions give
rise to  singularities  in the derivative \cite{chernov book}.

We shall 
be concerned with the associated billiard map $T$, defined to be
the first collision map on the boundary of $Q$.
Grazing
collisions cause discontinuities in the billiard map $T:M \to M$.
We assume, 
as in \cite{Y98}, that the billiard table $Q$ has {\it finite horizon} in the  sense that  the
billiard flow on  $Q$ does not have any trajectories making only tangential collisions.


\medskip

The first step is to find a suitable notion of topological entropy $h_*$ for the discontinuous
map $T$. 

Let $M'\subset M$ be the ($T$-invariant but not
compact)  set of points whose future and past orbits are never grazing.
By definition, $T$ is continuous on $M'$.
 The (Bowen--Pesin--Pitskel) topological entropy $\htop(F|_Z)$ can be defined for 
 a  map $F$ on an 
 non-compact set of continuity $Z$
 (see e.g. \cite{Bow1} and \cite[\S 11 and App. II]{Pesin}).
Chernov \cite{chernov}  studied the topological entropy
for a class of billiard maps including those of the present paper.
 In particular, he gave
  \cite[Thm~2.2]{chernov} a countable symbolic dynamics description
 of  two $T$-invariant  subsets of $M'$ of full Lebesgue measure
 in $M'$, 
 expressing their topological entropy in terms of those of the
 associated Markov chains.  The entropies found there are both bounded above
 by $\htop(T|_{M'})$, although Chernov does not prove their equality.
  
   These existing results are not convenient
   for our purposes, however,  since we have no control
   a priori on the measure of $M\setminus M'$. This is why we introduce  (Definition~\ref{def}) an ad hoc
  definition $h_*$ of the topological entropy for the billiard map 
  $T$ on the compact
  set $M$.

Our first main result (Theorem~\ref{thm:h_*}) says that
the topological entropies of $T$ defined by spanning sets and separating sets coincide with
the topological entropy $h_*$, 
that 
$h_*$ can also be obtained by using the refinements of partitions
of $M$ into maximal connected components on which $T$ and $T^{-1}$ are continuous,  and that 
$
h_*\ge \sup \{ h_\mu(T) : \mbox{$\mu$ is a $T$-invariant Borel probability measure on $M$} \}
$.

\medskip

To state our other main results,
we need to quantify
the recurrence to the singular set: Fix an angle $\vf_0$ close to $\pi/2$ and $n_0 \in \mathbb{N}$. We say that a collision is $\vf_0$-grazing if its 
angle with the
normal is larger than $\vf_0$ in absolute value.
Let $s_0 \in (0,1]$ be the smallest number such that 
\begin{align}
\label{defs0}
\mbox{any orbit of
length $n_0$ has at most $s_0n_0$ collisions which are $\vf_0$-grazing.}
\end{align}
Our sparse recurrence condition is 
\begin{equation}
\label{star}
\mbox{there exist $n_0$ and $\vf_0$ such that } h_*> s_0 \log 2 \, .
\end{equation}
(Due to the finite horizon condition,
we can choose $\vf_0$ and $n_0$ such that $s_0 < 1$. We refer to \S \ref{condstar} for further discussion of the condition.)

Assuming \eqref{star}, our second main result (Theorem~\ref{existu}) 
is that $T$ admits a unique invariant Borel probability measure $\mu_*$  
 of maximal entropy $h_*=h_{\mu_*}(T)$.
In addition,
$\mu_*(O)>0$ for any open set and
$\mu_*$ is\footnote{Recall that Bernoulli implies K-mixing, which implies strong
mixing, which implies ergodic. In practice, we first show K-mixing and then bootstrap
to Bernoulli.} Bernoulli.
Finally, the absolutely continuous invariant measure $\musrb$  
may coincide with $\mu_*$ {\it only}
if all non grazing
periodic orbits have the same Lyapunov exponent, equal to $h_*$. (No dispersing billiards 
which satisfy this condition are known. See also Remark~\ref{BIH}.)

Our third result is (Theorem~\ref{PP}) that $h_*$ coincides with the Bowen--Pesin--Pitskel entropy $\htop(T|_{M'})$
(still assuming \eqref{star}).

\smallskip
Next, Theorem~ \ref{nbhdthm} contains a key 
technical\footnote{This estimate implies that
almost every point approaches the singularity
sets more slowly than any exponential rate \eqref{rateeta}, see e.g. \cite{LM} for
an application of such rates of approach.} 
estimate 
on the measures of neighbourhoods of
singularity sets, \eqref{nbhd},  used  to prove Theorems~\ref{existu} and~\ref{PP} under the assumption \eqref{star}. 
Theorem~ \ref{nbhdthm}
also states  that
 $\mu_*$ has no atoms, that it
gives zero mass to any stable or unstable manifold and any singularity set,
that $\mu_*$ is $T$-adapted (in the sense 
of \eqref{adapt}),
 and that $\mu_*$-almost every $x \in M$ has stable and unstable
manifolds of positive length.

Finally, we obtain a lower bound  $\# \Fix T^{m} \ge C e^{h_* m}$
on the cardinality of the set of periodic orbits (Corollary~\ref{periodic}
and the comments thereafter)
whenever \eqref{star} holds.

\subsection{The Transfer Operator --- Organisation of the Paper}
\label{sec:transfer}

Our tool to construct the measure of maximal entropy is a transfer
operator $\cL=\cL_{\mbox{\tiny{top}}}$ with $\cL f = \frac{f \circ T^{-1}}{J^sT \circ T^{-1}}$
analogous to the transfer operator $\cL_{\mbox{\tiny{SRB}}} f= (f/ |\Det DT|) \circ T^{-1}$ which has proved very successful
\cite{demzhang11}  to study the
measure $\musrb$. An important difference
is that our transfer operator, $\cL f$,  is weighted by an unbounded\footnote{The
naive idea to introduce a bounded cutoff in the weight does not seem to work.} function ($1/J^sT$, where 
the stable
Jacobian $J^sT$ may tend to zero near  grazing
orbits). Using ``exact'' stable leaves instead of admissible 
approximate stable leaves will allow us to get rid of the Jacobian
after a leafwise change of variables  ---  the same change of variables in \cite{demzhang11} 
for the transfer operator $\cL_{\mbox{\tiny{SRB}}}$ associated with $\musrb$ left them
with  $J^sT$, allowing countable sums over homogeneity
layers to control distortion, and thus working with a Banach space giving
a spectral gap and exponential mixing.
In the present work, we relinquish  the homogeneity layers to avoid
unbounded sums (see e.g. the logarithm needed to obtain
the growth Lemma~\ref{lem:growth})  and obtain a bounded operator, with spectral
radius $e^{h_*}$. The price to pay is that we do not
have the distortion control
needed  for H\"older type moduli of continuity in the Banach norms of our
weak and strong spaces $\cB\subset \cB_w$.
The weaker modulus of continuity
than in 
\cite{demzhang11} does not yield a spectral gap. We thus do not
claim exponential mixing properties for the measure of maximal entropy $\mu_*$
constructed (in the spirit of the work of Gou\"ezel--Liverani \cite{GL2} for Axiom~A diffeomorphisms)
by combining right and left maximal eigenvectors 
$\cL \nu=e^{h_*}\nu$ and $\cL^* \tilde \nu=e^{h_*} \tilde \nu$ of the transfer operator.

\medskip

The paper is organised as follows: In Section~\ref{Full}, we give formal statements
of our main results. Section~\ref{sec:proof 1.4} contains the proof of Theorem~\ref{thm:h_*}
about equivalent formulations of $h_*$.  In Section~\ref{transfert},
we define our Banach spaces $\BB$ and $\BB_w$ of anisotropic
distributions, and we state the ``Lasota--Yorke'' type estimates
on our transfer operator $\LL$. 
Section~\ref{sec:growth} contains key combinatorial
growth lemmas,  controlling the growth in complexity of the iterates of a stable curve.
It also contains the definition of  Cantor rectangles
(Section~\ref{supera}.)
We next prove  the ``Lasota--Yorke''  Proposition~\ref{prop:ly}, the 
compact embedding of $\cB$ in $\cB_w$, and
show that the spectral radius of $\cL$ is equal
to $e^{h_*}$ in Section~\ref{sec:proof prop}.
The invariant probability measure $\mu_*$ is constructed 
 in Section~\ref{atlast0} by combining a right and left eigenvector ($\nu$ and $\tilde \nu$) of $\cL$.
Section~\ref{atlast0}  contains the proof of Theorem~\ref{nbhdthm} 
about the measure
 of singular sets.
 Section~\ref{notmixing} contains a key result
of absolute continuity of the unstable foliation with respect to $\mu_*$
as well as the proof that $\mu_*$ has full support, exploiting  $\nu$-almost everywhere positive length of unstable manifolds from Section~\ref{regular}. 
We establish upper and lower bounds on the $\mu_*$-measure of dynamical
Bowen balls in Section~\ref{BoBa}, deducing
from them a necessary condition for $\musrb$ and $\mu_*$ to
coincide.  Using the absolute continuity from Section~\ref{notmixing}, we show in Section~\ref{mixing}
that $\mu_*$ is K-mixing. In this section
we also use the upper bounds on Bowen balls
to see that $\mu_*$ is a measure of maximal entropy  and prove the Bowen--Pesin--Pitskel Theorem~\ref{PP}.
We 
deduce the Bernoulli property from K-mixing and
hyperbolicity in Section~\ref{sec:bernoulli}, adapting\footnote{As pointed out to us by Y. Lima, we could
 instead apply \cite[Thm~ 3.1]{Sarr} to the lift of $\mu_*$ to the symbolic space constructed in \cite{LM}.} 
\cite{ChH}. Finally, we show uniqueness in Section~\ref{sec:un}.

Our Hopf-argument proof of K-mixing requires showing absolute continuity of the
unstable foliation for $\mu_*$, a new result of independent interest,
which is the content of Corollary~ \ref{cor:abs cont}.
The ``fragmentation'' lemmas from Section~\ref{sec:growth}, needed to
get the lower bound on the spectral radius of the transfer operator, are also new.
They imply, in particular, that the length $|T^{-n}W|$ of every local stable manifold $W$
grows at the same exponential rate $e^{n h_*}$ 
(Corollary~\ref{cor:stable growth}).

\smallskip

We conclude this introduction
with two remarks on the finite horizon condition.

\begin{remark} [Finite Horizon and Collision Time $\tau$]
For $x \in M$, let $\tau(x)$ denote the distance from $x$ to $T(x)$.  
If $\tau$ is unbounded, i.e., if there is a collision-free
trajectory for the flow, then there must be a flow trajectory
making only tangential collisions.
The reverse implication, however, is not true. 
Our\footnote{We shall need the slightly stronger version e.g. in Lemmas~\ref{lem:hsep} and~\ref{lem:hspan}.} finite horizon assumption therefore
implies  that $\tau$ is bounded on $M$. Assuming only that
$\tau$ is bounded  is sometimes also 
called finite horizon  \cite{chernov book}. (If the scatterers $B_i$
are viewed as open, then tangential collisions simply do not occur and
the two definitions of finite horizon are reconciled.)
\end{remark}

\begin{remark}
[Billiard with Infinite Horizon]\label{BIH}
Chernov \cite[\S3.4]{chernov} proved that the topological entropy of the Sinai billiard map 
$T$ restricted to the non compact set $M'$ is infinite
if the horizon is {\it not} finite, and together with Troubetskoy \cite{CT} constructed invariant measures  with infinite metric entropy
for this map.
Since the entropy of the smooth measure  $\musrb$ is
finite,
the measure $\musrb$
does not maximise entropy for infinite horizon billiards. Chernov conjectured
\cite[Remark 3.3]{chernov} that this property holds for more general billiards, in particular for
Sinai billiards with finite horizon.
\end{remark}

\section{Full Statement of Main Results} 
\label{Full}

In this section, we formulate  definitions of topological entropy for the billiard map that
we shall prove are equivalent before stating formally all main results of this paper.

\subsection{Definitions of Topological Entropy  $h_*$ of $T$ on $M$}
\label{sec:def}

We first introduce notation:
Adopting the standard coordinates $x = (r,\vf)$, for $T$,  where $r$ denotes arclength along
$\partial B_i$ and $\vf$ is the angle the post-collision trajectory makes with the normal to 
$\partial B_i$, the phase space of the map is  the compact
metric space $M$ given by the disjoint union of cylinders,
$$M := \partial Q \times \left [-\frac{\pi}{2}, \frac{\pi}{2} \right ] 
= \bigcup_{i=1}^D \partial B_i \times \left [-\frac{\pi}{2}, \frac{\pi}{2} \right ]\, .
$$
We denote each connected component of  $M$ by $M_i = \partial B_i \times [-\frac{\pi}{2}, \frac{\pi}{2}]$.
In the coordinates $(r, \vf)$, the billiard map $T:M \to M$ preserves
\cite[\S2.12]{chernov book} the smooth invariant measure\footnote{All measures
in this work are finite Borel measures.} defined by
$\musrb = (2  |\partial Q|)^{-1} \cos \vf \, drd\vf$.

We  discuss next the discontinuity set of $T$:
Letting $\cS_0 =\{ (r, \vf) \in M : \vf = \pm \pi/2 \}$ denote the set of 
tangential collisions, then for each nonzero $n \in \mathbb{N}$,
the set 
$$
\cS_{\pm n} = \cup_{i=0}^n T^{\mp i}\cS_0
$$ 
is the singularity set for $T^{\pm n}$.
In this notation, 
the $T$-invariant (non compact) set $M'$ of
continuity of $T$  is 
$
M'= M \setminus \cup_{n \in \integer} \cS_n 
$.

For $k, n \ge 0$, let $\cM_{-k}^n$ denote the partition of $M \setminus (\cS_{-k} \cup \cS_n)$ into its maximal connected components.  Note that all elements of $\cM_{-k}^n$ are open sets.
The cardinality of the sets $\cM_0^n$ will play a key role in the estimates
on the transfer operator in Section~\ref{transfert}.
We formulate the following definition with the idea that the growth rate of elements in $\cM_{-k}^n$
should define the topological entropy of $T$, by analogy with the definition  using 
a generating open cover (for continuous maps on compact spaces).

\begin{definition}
\label{def}
$
h_* = h_*(T) := \limsup_{n \to \infty} \frac 1n \log \# \cM_0^n 
$.
\end{definition}

The fact that the limsup defining $h_*$ is a limit, as well as several equivalent 
characterizations involving the cardinality of related dynamical partitions or a variational
principle, are proved in
Theorem~\ref{thm:h_*} (see Lemma~\ref{prop:equiv}). 

\begin{remark}[$h_*(T) = h_*(T^{-1})$]\label{2.2}
If $A \in \cM_0^n$, then $T^nA \in \cM_{-n}^0$ since $T^n\cS_n = \cS_{-n}$.  Thus
$\# \cM_0^n = \# \cM_{-n}^0$, and so $h_*(T) = h_*(T^{-1})$.
\end{remark}

It will be convenient to express $h_*$ in terms of the 
rate of growth
of the cardinality of the refinements of a fixed partition, i.e., $\bigvee_0^n T^{-i}\cP$, for some
fixed $\cP$.  Although $\cM_{0}^n$ is not immediately of this form, we will show 
that in fact $h_*$ can be expressed in this fashion, obtaining along the way subadditivity
of $\log \# \cM^n_0$. For this, we introduce two sequences of partitions.
Let $\cP$ denote the partition of $M$ into maximal connected sets on which
$T$ and $T^{-1}$ are continuous.  
Define $\cP_{-k}^n = \bigvee_{i=-k}^n T^{-i} \cP$. Then, $n \mapsto \log \# \cP^{n}_{-k}$
is subadditive for any fixed $k$, in particular the  limit
$
\lim_{n \to \infty} \frac{1}{n}  \log \# \cP^{n}_{0}$ exists.

The interior of each element of $\cP$ 
corresponds to precisely one element of $\cM_{-1}^1$;
however, its refinements $\cP_{-k}^n$ may also contain
some isolated points if three or more scatterers have a common tangential trajectory.
Figure 1 displays two such examples (the pictures are local: we have not represented
all discs needed to ensure finite horizon).

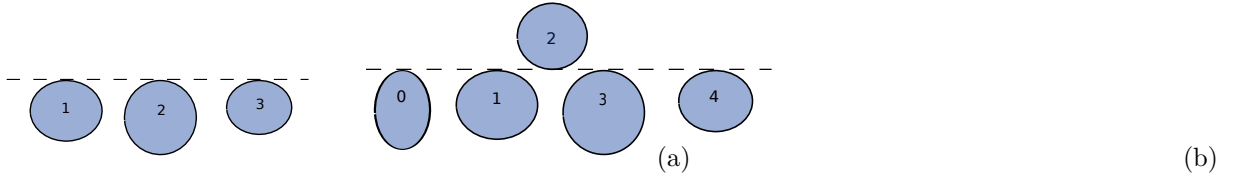
\begin{figure}
\begin{center}
\resizebox{!}{.4 in}{\includegraphics{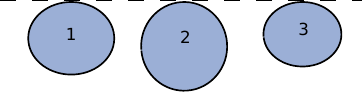}}
\hspace{.2 in}
\resizebox{!}{.8 in}{\includegraphics{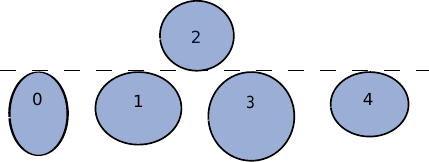}}
\begin{picture}(140,2)
\put(-47,-4){\small (a)}
\put(152, -4){\small (b)}
\end{picture}
\parbox{.9 \textwidth}{\caption{\small (a) The billiard trajectory corresponding to the dotted line has symbolic itinerary 123, but
is an isolated point in $\cP_0^1$.  Any open set with symbolic itinerary 12 cannot land on scatterer 3 (unless it first wraps around the torus).  (b) The billiard trajectory corresponding to the dotted line and having
symbolic trajectory 1234 is not isolated since it belongs to the boundary of an open set with the
same symbolic sequence; however, the addition of scatterer 0 on the common tangency forces the
point with symbolic trajectory 01234 to be isolated.}}
\end{center}
\vspace{-8 pt}
\end{figure}

\smallskip

Let now $\hP_{-k}^n$ denote the collection of
interiors of elements of $\cP_{-k}^n$. Then
$\cP_{-k}^n$ forms a finite partition of $M$, while $\hP_{-k}^n$ forms a partition of 
$M \setminus (\cS_{-k-1} \cup \cS_{n+1})$ into open, connected sets.
(We will show in Lemma~\ref{prop:equiv} that $\hP_{-k}^{n} = \cM_{-k-1}^{n+1}$.)

Finally,  we recall the classical Bowen \cite{walters} definitions of topological entropy for continuous maps using  
$\ve$-separated and $\ve$-spanning sets.  Define the dynamical distance 
\begin{equation}
\label{eq:d_n}
d_n(x,y) := \max_{0 \le i \le n} d(T^ix, T^iy)\, ,
\end{equation}
where $d(x,y)$ is the Euclidean metric on each $M_i$, and $d(x,y) = 10 D \cdot \max_i \diam (M_i)$ if $x$ and $y$ belong
to different $M_i$ (this definition ensures we get a compact set), where $D$ is the number of scatterers.

As usual, 
given $\ve >0$, $n \in \bN$, we call $E$ an {\it $(n,\ve)$-separated set} if
for all $x, y \in E$ such that $x \neq y$, we have $d_n(x,y) > \ve$.
We call $F$ an {\it $(n, \ve)$-spanning set} if for all $x \in M$, there
exists $y \in F$ such that $d_n(x,y) \le \ve$. 

Let $r_n(\ve)$ denote the maximal cardinality of any $(n,\ve)$-separated set, and let $s_n(\ve)$
denote the minimal cardinality of any $(n, \ve)$-spanning set.  We recall two related
quantities:
$$\hsep
= \lim_{\ve \to 0} \limsup_{n \to \infty} \frac 1n \log r_n(\ve)\, , \qquad
\hspan 
= \lim_{\ve \to 0} \limsup_{n \to \infty} \frac 1n \log s_n(\ve)\, .
$$

Although $\lim_{n \to \infty} \frac 1 n \log \# \cP^n_0$,  $\hsep$, and $\hspan$ are typically used for continuous maps, our first main result
is  that
these naively defined quantities for the discontinuous billiard map $T$ all
agree with $h_*$, and
they give an upper bound for the Kolmogorov entropy:

\begin{theorem}[Topological Entropy of the Billiard]
\label{thm:h_*}
The limsup in Definition~\ref{def} is a limit, and in fact the sequence
$\log \# \cM^n_0$ is subadditive. In addition, we have:
\begin{enumerate}
\item\label{cP} $h_*=\lim_{n \to \infty} \frac 1 n \log \# \cP^n_0$;
\item\label{hP} the  sequence $\frac 1 n \log \# \hP^n_0$ also converges to $h_*$ as $n\to \infty$;
  \item\label{seppspann} $h_* = \hsep$ and $h_* = \hspan$;
  \item\label{varprinc1} $h_* \ge \sup \{ h_\mu(T) : \mbox{$\mu$ is a $T$-invariant Borel probability measure on $M$} \}$. 
  \end{enumerate}
  \end{theorem}

The above theorem will follow from
Lemmas~\ref{prop:equiv}, ~\ref{lem:hsep}, ~\ref{lem:hspan}, and ~\ref{lem:geq}.

  (We shall obtain in Lemma~\ref{lem:super} a superadditive property for $\log \# \cM^n_0$.)

\subsection{The Measure $\mu_*$ of Maximal Entropy}
  
  Our next main result,  existence and the Bernoulli property of a unique
  measure of maximal entropy, 
  will be proved in Section~\ref{mmu},   
  using the transfer operator
  $\cL$  studied in Section~\ref{transfert}.
  
  \begin{theorem}[Measure of Maximal Entropy for the Billiard]\label{existu}
   If $h_*>s_0\log 2$ then $$h_*=\max \{ h_\mu(T) : \mbox{$\mu$ is a $T$-invariant Borel probability measure on $M$} \}\, .$$
  Moreover,
  there exists a unique  $T$-invariant Borel probability measure $\mu_*$
  such that $h_* = h_{\mu_*}(T)$. In addition, $\mu_*$ is
  Bernoulli  and  $\mu_*(O)>0$ for all open sets $O$.
  Finally, if there exists a non grazing periodic point $x$ of period
  $p$ such that $\frac{1}{p} \log |\det (DT^{-p}|_{E^s}(x))|\ne h_*$
then $\mu_* \neq \musrb$.  
\end{theorem}

The above theorem  follows from Propositions~\ref{lastitem}, \ref{prop:not eq}, and \ref{Bern}, 
Corollary~\ref{cor:max}, and Proposition~\ref{prop:unique}.
(J. De Simoi has told us
   that
   \cite[\S 4.4]{DKL}  the (possibly empty) set
   of planar billiard tables satisfying a non-eclipsing condition (i.e., open billiards) for which  $\frac{1}{p} \log |\det (DT^{-p}|_{E^s}(x))|= h_*$
      for all $p$ and all non-grazing $p$-periodic points $x$ has infinite
   codimension.)
   
The existence of $\mu_*$ with $h_{\mu_*}(T) =h_*$, together with item \eqref{cP} of Theorem~\ref{thm:h_*}
expressing $h_*$ as a limit involving the refinements
of a single partition, will allow us to interpret
$h_*$ as the Bowen--Pesin--Pitskel topological  entropy of $T|_{M'}$ in Section~\ref{mixing}:
  
   \begin{theorem}[$h_*$ and Bowen--Pesin--Pitskel Entropy]\label{PP}  If $h_*>s_0\log 2$ then $h_* = \htop(T|_{M'})$. 
   \end{theorem}

\subsection{A Key Estimate on Neighbourhood of Singularities}

We call a smooth curve in $M$ a {\it stable curve} if its tangent vector at each point lies in the stable cone,
and define an {\it unstable curve} similarly.
As mentioned in Section~\ref{sec:intro}, the sets $\cS_n$ are the singularity sets for $T^n$,
$n \in \mathbb{Z}\setminus\{0\}$.  The set $\cS_n \setminus \cS_0$ comprises \cite{chernov book}  a finite union of stable
curves for $n > 0$ and a finite union of unstable curves for $n < 0$.
For any $\epsilon >0$ and any set $A \subset M$, 
we denote by $\cN_\epsilon(A)=\{ x \in M \mid d(x,A) < \epsilon\}$ the 
$\epsilon$-neighbourhood of $A$.

\smallskip
The following key  result gives information on the
measure of neighbourhoods of the singularity sets (it is used in the proofs of
Theorem~\ref{existu} and, indirectly, Theorem~\ref{PP}).

\begin{theorem}[Measure of Neighbourhoods of Singularity Sets]\label{nbhdthm}
Assume that $h_* > s_0 \log 2$ 
and let $\mu_*$ be the ergodic measure of maximal entropy constructed
in  \eqref{defm}.  The measure $\mu_*$ has no atoms, and
for any local stable or unstable manifold $W$ we have $\mu_*(W)=0$. 
In addition $\mu_*(\cS_n)=0$
for any $n \in \mathbb{Z}$.

More precisely,  for any $\gamma>0$
so that $2^{s_0\gamma}<e^{h_*}$ and $n \in \mathbb{Z}$, there exist $C$ and $\hat C_n <\infty$
such that for all $\ve > 0$ and
any smooth curve $S$ uniformly transverse to the stable cone, 
\begin{equation}
\label{nbhd}
\mu_*(\cN_\epsilon(S)) < \frac{C}{|\log \epsilon|^{\gamma}} \, ,\quad
\mu_*(\cN_\epsilon(\cS_n)) < \frac{\hat C_n}{|\log \epsilon|^{\gamma} } \, .
\end{equation}
Since $h_*> s_0\log 2$ we may take $\gamma > 1$, and we have
$$\int | \log d(x, \cS_{\pm 1})| \, d\mu_* < \infty\, ,
$$
(i.e., $\mu_*$ is $T$-adapted \cite{LM}), and $\mu_*$-almost every $x \in M$ has stable and unstable
manifolds of positive length.
\end{theorem}

Theorem~\ref{nbhdthm}
follows from Lemma~\ref{lem:approach} and Corollary~\ref{integral}.

This theorem is especially of interest for $\gamma >1$,
since in this case it implies that $\mu_*$-almost every point does not approach the singularity
sets faster than some exponential, see \eqref{rateeta}.  In addition, it allows us to
give a lower bound on the number of periodic
orbits: For $m\ge 1$, let $\Fix T^m$ denote the 
set $\{x \in M\mid T^m(x)=x\}$. 
By \cite{BSC} and \cite[Cor 2.4]{chernov}, there exist $h_C \ge h_{\musrb}(T) >0$
and $C>0$  with  $\# \Fix T^m \ge C e^{h_C m}$ for all $m$. Our result is that 
(possibly up to a period $p$)
we can take $h_C=h_*$ if $h_* > s_0 \log 2$:

\begin{cor}[Counting Periodic Orbits]\label{periodic}
If $h_* > s_0 \log 2$ then  there exist $C>0$ and $p \ge 1$
such that $\# \Fix T^{p m} \ge C e^{h_* p m}$ for all $m\ge 1$.
\end{cor}

\begin{proof}
The corollary follows from the work of Lima--Matheus \cite{LM}, which in turn relies on 
work of Gurevi\v{c} \cite{Gu1, Gu2} (see the proof of \cite[Thm~1.1]{Sar}).
We recall briefly the setup of \cite[Theorem~1.3]{LM}:  Under assumptions (A1)-(A6),
the authors construct for any $T$-adapted measure $\mu$ with positive Lyapunov exponent,
a countable Markov partition 
 that allows them to code a full $\mu$-measure set of points.
Once this partition has been constructed, \cite[Corollary~1.2]{LM} implies the above lower
bound on periodic orbits for $T$ with rate given by $h_\mu(T)$.

\cite[Theorem~1.3]{LM} applies to our measure of maximal entropy
$\mu_*$ since it is $T$-adapted with positive Lyapunov exponent.
In addition, conditions (A1)-(A4) of \cite{LM} are requirements on the smoothness of the
exponential map on the manifold, which are trivially satisfied in our setting since 
$M$ is a finite union of cylinders
and $\cS_{\pm 1}$ is a finite union of curves.
Finally, conditions (A5) and (A6) are requirements on the rate at which $\| DT \|$ and $\| D^2T \|$
grow as one approaches $\cS_1$.  These are standard estimates for billiards
and in the notation of \cite{LM}, if we choose $a=2$, then conditions (A5) and (A6) hold, choosing 
there  $\beta=1/4$ and any $b>1$.
\end{proof}

After the first version of our paper was submitted,
J.~Buzzi \cite[v2]{Bu} obtained results allowing one to
bootstrap from  Corollary~\ref{periodic} by exploiting the fact that $T$ is topologically mixing, to show that
if $h_* > s_0 \log 2$ then  there exists $C>0$ so that $\# \Fix T^{ m} \ge C e^{h_*  m}$ for all $m\ge 1$
\cite[Theorem 1.5]{Bu} .

\subsection{On Condition    \eqref{star} of Sparse Recurrence to Singularities}
\label{condstar}

We are not aware of any dispersing billiard on the torus for which the bound $h_* > s_0 \log 2$ from \eqref{star}
fails.    Let us start by mentioning that if there are no triple tangencies on
the table --- a generic condition --- then $s_0 \le 2/3$.
To discuss this  condition further, 
our starting point is
 claim \eqref{varprinc1} of Theorem~\ref{thm:h_*}, which
implies by the Pesin entropy formula \cite{katok strelcyn},
\begin{equation}
\label{eq:ent}
h_* \ge h_{\musrb}(T) = \int \log J^uT \, d\musrb  \, .
\end{equation}
Thus it suffices to check  $\chi_{\musrb}^+ > s_0 \log 2$ in order
to verify \eqref{star}, where 
 $\chi^+_{\musrb} = \int \log J^uT \, d\musrb$ is the positive Lyapunov exponent of $\musrb$.

\begin{figure}[h]
\begin{tikzpicture}[x=8mm,y=8mm]

 \filldraw[fill=black!20!white, draw=black] (4.7,.7) circle (.6);
\filldraw[fill=black!20!white, draw=black] (6.1,.7) circle (.6);
\filldraw[fill=black!20!white, draw=black] (7.5,.7) circle (.6);
\filldraw[fill=black!20!white, draw=black] (4,1.9) circle (.6);
\filldraw[fill=black!20!white, draw=black] (5.4,1.9) circle (.6);
 \filldraw[fill=black!20!white, draw=black] (6.8,1.9) circle (.6);
  \filldraw[fill=black!20!white, draw=black] (8.2,1.9) circle (.6); 
\filldraw[fill=black!20!white, draw=black] (4.7,3.1) circle (.6);
\filldraw[fill=black!20!white, draw=black] (6.1,3.1) circle (.6);
\filldraw[fill=black!20!white, draw=black] (7.5,3.1) circle (.6);

\draw[dashed] (4.7,3.1) -- (5.4,1.9) -- (6.8,1.9) -- (6.1,3.1) -- cycle;

\draw[decoration={brace,mirror,raise=2pt},decorate]
  (5.4,1.9) -- node[below=1.6pt] {$d$} (6.8,1.9);

\node at (6.5,-.7){\small$(a)$};

\filldraw[fill=black!20!white, draw=black]  (13.6,0) arc (0:90:1.6);
\filldraw[fill=black!20!white, draw=black!20!white]  (13.6,0) -- (12,1.6) -- (12,0) -- cycle;
\filldraw[fill=black!20!white, draw=black]  (16,1.6) arc (90:180:1.6);
\filldraw[fill=black!20!white, draw=black!20!white]  (16,1.6) -- (14.4,0) -- (16,0) -- cycle;

\filldraw[fill=black!20!white, draw=black]  (14.4,4) arc (180:270:1.6);
\filldraw[fill=black!20!white, draw=black!20!white]  (14.4,4) -- (16,2.4) -- (16,4) -- cycle;

\filldraw[fill=black!20!white, draw=black]  (12,2.4) arc (270:360:1.6);
\filldraw[fill=black!20!white, draw=black!20!white]  (12,2.4) -- (13.6,4) -- (12,4) -- cycle;

\draw[dashed] (12,0) rectangle (16,4) ;

\filldraw[fill=black!20!white, draw=black]  (14,2) circle (.9);
\draw (14,2) -- (14.9,2);
\node at (14.4,1.8){\small $\rho$};

\draw (16,4) -- (14.87,2.87);
\node at (15.5, 3) {\small$R$};

\node at (14,-.7){\small$(b)$};

 \end{tikzpicture}
\caption{(a) The Sinai billiard on a triangular lattice studied in \cite{Gaspard} with angle $\pi/3$, scatterer of radius 1, and distance
$d$ between the centers of adjacent scatterers.  (b) The Sinai billiard on a square lattice with scatterers of radii $\rho < R$ studied in \cite{Garrido}. The boundary of a single cell is indicated by dashed lines in both tables.}
\label{fig:tables}
\end{figure}
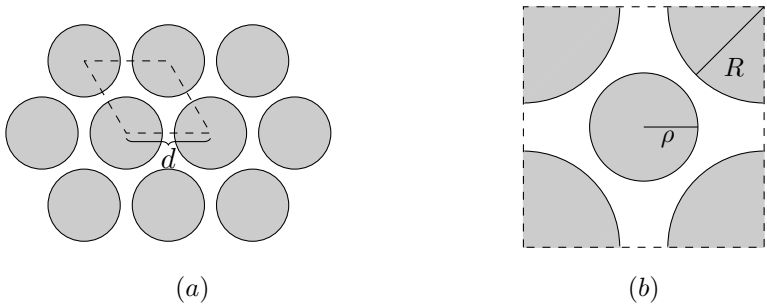

Firstly, we mention two numerical case studies from the literature:

Baras and Gaspard \cite{Gaspard} studied the Sinai billiard 
corresponding to the periodic Lorentz gas with disks of radius $1$ centered in a triangular lattice (Figure~\ref{fig:tables}(a)).  The distance $d$
between points on the lattice is varied from $d=2$ (when the scatterers touch) to $d = 4/\sqrt{3}$ (when the horizon
becomes infinite).  All computed values of the Lyapunov exponent\footnote{The reported values in
\cite{Gaspard} are for the billiard flow.  These can be converted to Lyapunov exponents for the map via the
well-known formula $\chi^+_{map} = \bar \tau \chi^+_{flow}$, where $\bar \tau$ is the average free flight time.   For this
billiard table, $\bar \tau = \frac{d^2\sqrt{3}}{4} - \frac{\pi}{2}$, using \cite[eq. (2.32)]{chernov book}.}
are greater than $\frac{2}{3} \log 2$
\cite[Table 1]{Gaspard}.   (Notably $\chi^+_{\musrb}$ does not decay as 
the minimum free flight-time $\tau_{\min}$ tends to  zero.) For these billiard tables, since every segment with a double tangency is followed by \emph{two} non-tangential
collisions, one can choose $\vf_0$ and $n_0$ so that \eqref{defs0} is satisfied
with $s_0 = 1/2$.  Thus \eqref{star} holds for all computed values in this family of tables.

Garrido \cite{Garrido} studied the Sinai billiard corresponding to the periodic Lorentz gas
with two scatterers of
 radii $\rho < R$  on the unit square lattice 
 (Figure~\ref{fig:tables}(b)).  Setting $R = 0.4$, \cite[Figure~6]{Garrido} computed $\chi_{\musrb}^+$ numerically
for about 20 values of $\rho$ ranging from $\rho=0.1$ (when the scatterers touch) to $\rho = \frac{\sqrt{2}}{2} - 0.4$ (when
the horizon becomes infinite).  All computed values of $\chi_{\musrb}^+$ are greater than $0.8 > \log 2$ so that \eqref{star} holds for all such tables.
(For these tables as well, one can in fact choose $s_0 = 1/2$.)

Secondly,
for the family of tables studied by Garrido, we obtain an open set of pairs of parameters $(\rho, R)$ satisfying \eqref{star} as follows.  
To ensure finite horizon and disjoint scatterers, the constraints are 
$$ \frac 12 < \rho + R < \frac{\sqrt{2}}{2}, \quad \rho < R < \frac 12, \quad \mbox{ and } \quad R > \frac{\sqrt{2}}{4} \, .
$$
Since $\musrb$ is a probability measure, denoting by
$\cK_{\min}>0$ the minimum curvature and using a well-known  \cite[eqs. (4.10) and (4.15)]{chernov book}
bound for the
unstable hyperbolicity exponent (see also \cite[Remark~3.47]{chernov book}) for the relation to entropy),
we  have,
\[
\chi^+_{\musrb} \ge \log (1 + 2 \tau_{\min} \cK_{\min})  \, .
\] 
We find that this is greater than $(1/2) \log 2$ whenever  $\tau_{\min} \cK_{\min} > \frac{\sqrt{2}-1}{2}$.  
If $R > 1 - \frac{\sqrt{2}}{2} + \rho$, then $\tau_{\min} = 1-2R$, and $\cK_{\min} = R^{-1}$,
 so that $\tau_{\min}\cK_{\min} =  R^{-1} -2$.  Thus if $R < \frac{2}{3+\sqrt{2}}$, then \eqref{star} holds.  
On the other hand if $R < 1 - \frac{\sqrt{2}}{2} + \rho$, then $\tau_{\min} = \frac{\sqrt{2}}{2} - R - \rho$ so that
$\tau_{\min}\cK_{\min} = \frac{\sqrt{2}}{2R} - 1 - \frac{\rho}{R}$.  Thus \eqref{star} holds whenever 
$R< \frac{\sqrt{2} - 2 \rho}{1 + \sqrt{2}}$.
 The union of these two sets is defined by
the inequalities 
$$
\frac{\sqrt{2}}{4} < R < \frac{2}{3+\sqrt{2}}, \quad R < \frac{\sqrt{2}-2 \rho}{1 + \sqrt{2}}, \quad \mbox{ and } \quad \rho + R > \frac 12 \, . 
$$
We remark that this region intersects the line $R + \sqrt{2}\rho = \frac{\sqrt{2}}{2}$.  
This line corresponds to the set of tables which admit a period 8 orbit making 4 grazing
collisions around the disk of radius $\rho$ and 4 collisions at angle $\pi/4$ with the disk of radius $R$.  For these tables, 
$s_0 = 1/2$, and we see that \eqref{star} admits tables with grazing periodic orbits.

\smallskip
Thirdly, it seems true
 that if there are no periodic orbits making at least one grazing collision then, for any $\epsilon >0$, the constants $n_0$ and $\varphi_0$ can be chosen to ensure $s_0<\epsilon$. This has led 
 P.-A.~Guih\'eneuf to conjecture that there exists a natural topology\footnote{For a fixed number of scatterers, a candidate 
 is given  by the distance defined in \cite[\S2.2, \S3.4,  Remark~2.9(b)]{demzhang13}.} on the set of billiard tables so that,  for any $\epsilon >0$, the
 set of tables for which $s_0<\epsilon$ is generic (that is, open and dense).
 This would immediately imply that our condition \eqref{star} is generically satisfied.
 
 \smallskip 
 
Finally, we mention that Diller, Dujardin, and Guedj
 \cite[Example 4.6]{DG2} construct 
 a birational map $F$ having a measure of maximal entropy
 which is mixing but not $F$-adapted, by showing that $F$ violates the Bedford--Diller \cite{BeDi} 
 recurrence condition. The Bedford--Diller condition does not
 have a natural analogue in our setting since double tangencies always
 occur.  One could interpret our sparse recurrence condition $h_* >  s_0 \log 2$
 as its replacement. It would be interesting to find billiards
 for which  $h_* \le   s_0 \log 2$ and which admit a
 non $T$-adapted measure of maximal entropy.
 

\section{Proof of Theorem~\ref{thm:h_*} (Equivalent Formulations of $h_*$)}
\label{sec:proof 1.4}

In this section, we shall prove Theorem~\ref{thm:h_*} through 
Lemmas~\ref{prop:equiv}, ~\ref{lem:hsep}, ~\ref{lem:hspan}, and ~\ref{lem:geq}.

We first recall some facts about the uniform hyperbolicity of $T$ 
to  introduce notation which will be used throughout.
It is well known \cite{chernov book}  that $T$ is uniformly hyperbolic in the following sense:  
First, the cones
$C^u = \{ (dr, d\vf) \in \mathbb{R}^2 : \cK_{\min} \le d\vf/dr \le \cK_{\max} + 1/\tau_{\min} \}$
and $C^s = \{ (dr, d\vf) \in \mathbb{R}^2 : - \cK_{\min} \ge d\vf/dr \ge -\cK_{\max} - 1/\tau_{\min} \}$,
are strictly invariant under $DT$ and $DT^{-1}$, respectively, whenever these derivatives exist.  Here, $\cK_{\max}$ 
represent the maximum  curvature of the scatterer boundaries and $\tau_{\max}<\infty$
is the largest free flight time between collisions.
Second, recalling that $\cK_{\min}>0$, $\tau_{\min} > 0$ denote the minimum curvature and the minimum free flight time, and
setting
$$\Lambda := 1 + 2\cK_{\min} \tau_{\min}
\, ,
$$ 
there exists $C_1>0$ such that for all $n \ge 0$,
\begin{equation}
\label{eq:hyp}
\| DT^n(x) v \| \ge C_1 \Lambda^n \| v\| \, , \, \forall v \in C^u\, , \quad \| DT^{-n}(x) v \| \ge C_1 \Lambda^n \|v \| \, , \, \forall v \in C^s\, , 
\end{equation}
for all $x$ for which $DT^n(x)$, or respectively $DT^{-n}(x)$, is defined,
so that $\Lambda$ is a lower bound\footnote{\label{lastf}Therefore, 
$
h_{\musrb}(T) = \int \log J^uT \, d\musrb >   \log \Lambda
$
and the bound $\log (1 + 2\cK_{\min} \tau_{\min}) > s_0 \log 2$
implies \eqref{star}, as in Section~\ref{condstar}.}
on the hyperbolicity constant of the map $T$.

\subsection{Preliminaries}

The following lemma provides important information regarding the structure of the partitions
$\cP_{-k}^n$, which we will use to make an explicit connection between $\cM_{-k}^n$ and $\hP_{-k}^n$ in Lemma~\ref{prop:equiv}.

\begin{lemma}
\label{lem:conn}
The elements of $\cP_{-k}^n$ are connected
sets for all $k\ge 0$ and $n \ge 0$.
\end{lemma}

\begin{proof}
The statement is true by definition for $\cP = \cP_0^0$.  We will prove the general statement by induction
on $k$ and $n$ using the fact that $\cP_{-k}^{n+1} = \cP_{-k}^n \bigvee T^{-1} \cP_{-k}^n$, and
$\cP_{-k-1}^n = \cP_{-k}^n \bigvee T \cP_{-k}^n$.

Fix $k, n \ge 0$, and assume the elements of $\cP_{-k}^n$ are connected sets.
Let $A_1, A_2 \in \cP_{-k}^n$.  If $T^{-1}A_1 \cap A_2$ is empty or is an isolated point, then it is
connected.  So suppose $T^{-1}A_1 \cap A_2$ has nonempty interior.  

Clearly, $T^{-1}A_1$ is connected since $T^{-1}$ is continuous on elements of $\cP_{-k}^n$ for all
$k, n \ge 0$.  Notice that the boundary of $A_1$ is comprised of finitely many smooth
stable and unstable curves in $\cS_{-k} \cup \cS_n$,
as well as possibly  a subset of
$\cS_0$ (\cite[Prop~4.45 and Exercise~4.46]{chernov book}, see also \cite[Fig~4.17]{chernov book}).
We shall refer to these as the \emph{stable} and \emph{unstable parts of the boundary} of $A_1$.
Similar facts apply to the boundaries of $A_2$ and $TA_1$.
 
We consider whether a stable part of the boundary of $T^{-1}A_1$ can cross a stable part of
the boundary of $A_2$, and create two or more connected components of
$T^{-1}A_1 \cap A_2$.  Call these two boundary components $\gamma_1$ and $\gamma_2$
and notice that such an occurrence would force $\gamma_1$ and $\gamma_2$ to intersect in
at least two points.

We claim the following fact:  If a stable curve $S_i \subset T^{-i}\cS_0$ intersects $S_j \subset T^{-j}\cS_0$
for $i<j$, then $S_j$ must terminate on $S_i$.  This is because $T^iS_i \subset \cS_0$, while
$T^iS_j \subset T^{i-j}\cS_0$ is still a stable curve, terminating on $\cS_0$.  A similar property holds for unstable
surves in $\cS_{-i}$. and $\cS_{-j}$.

The claim implies that $\gamma_1$ and $\gamma_2$ both belong to $T^{-j}\cS_0$ 
for some $1 \le j \le n$.  But
when such curves intersect, again, one must terminate on the other (crossing would violate
injectivity of $T^{-1}$).

A similar argument precludes the possibility that unstable parts of the boundary cross one another multiple times.  
It follows that the only intersections allowed are stable/unstable boundaries of $T^{-1}A_1$ terminating on corresponding 
stable/unstable boundaries of $A_2$, or transverse intersections between stable components of $\partial (T^{-1}A_1)$ 
and unstable components of $\partial A_2$, and vice versa.  This last type of intersection cannot produce multiple connected
components due to the {\em continuation of singularities}, which states that every stable curve in
$\cS_{-n}\setminus \cS_0$ is part of a monotonic and piecewise smooth decreasing curve which terminates on $\cS_0$ (see \cite[Prop~~4.47]{chernov book}).  A similar fact holds for 
unstable curves in $\cS_n \setminus \cS_0$.
This implies that $T^{-1} A_1 \cap A_2$ is a connected set, and since $A_1$ and $A_2$ were arbitrary, that
$\cP_{-k}^{n+1}$ is comprised entirely of connected sets.

Similarly, considering $TA_1 \cap A_2$ proves that all elements of $\cP_{-k-1}^n$ are connected.
\end{proof}

From the proof of Lemma~\ref{lem:conn}, we can see that, aside from isolated points, elements of $\cP_{-k}^n$ consist
of connected cells which are roughly ``convex'' and have boundaries comprised of stable and
unstable curves.

\begin{lemma}
\label{lem:card}
There exists $C>0$, depending on the table $Q$, such that for any 
$k, n \in \mathbb{N}$, $\# \hP_{-k}^n \le \# \cP_{-k}^n \le \# \hP_{-k}^n + C(n+k+1)$.
\end{lemma}

\begin{proof}
It is clear from the definition of $\hP_{-k}^n$ and $\cP_{-k}^n$ that 
\[
\# \cP_{-k}^n = \# \hP_{-k}^n + \# \{ \mbox{isolated points} \} \, ,
\]
where the isolated points in $\cP_k^n$ can be created by multiple tangencies aligning in a
particular manner, as described above (see Figure~1).    Thus the first inequality is trivial.

The set of isolated points created at each forward iterate is contained in $\cS_0 \cap T^{-1} \cS_0$,
while the set of isolated points created at each backward iterate is contained in
$\cS_0 \cap T \cS_0$.  We proceed to estimate the cardinality of these sets.

Let $r_0$ be sufficiently small such that for any segment $S \subset \cS_0$ of length $r_0$, 
the image
$TS$ comprises at most $\tmax/\tmin$ connected curves on which $T^{-1}$ is smooth
\cite[Sect. 5.10]{chernov book}.  
For each $i$, the number of points in $\partial B_i \cap \cS_0 \cap T^{-1} \cS_0$
is thus bounded by $2 |\partial B_i| \tmax/(\tmin r_0)$, where the factor $2$ comes from
the top and bottom boundary of the cylinder.  Summing over $i$, we have
$\# (\cS_0 \cap T^{-1} \cS_0) \le 2 |\partial Q| \tmax/(\tmin r_0)$.
Due to reversibility, a similar estimate holds for
$\# (\cS_0 \cap T \cS_0)$.  Since this bound holds at each iterate, the second inequality holds with
$C = \frac{2 |\partial Q| \tmax}{\tmin r_0}$.
\end{proof}

\subsection{Formulations of $h_*$ Involving $\cP$ and $\hP$}

The following lemma gives claims \eqref{cP} and \eqref{hP} of Theorem~\ref{thm:h_*}:
\begin{lemma}
\label{prop:equiv}
The following holds for every $k\ge 0$.
We have $\hP_{-k}^{n} = \cM_{-k-1}^{n+1}$ for every  $n \ge 0$.
Moreover, the following limits exist and are equal to $h_*$:
\[
h_* = \lim_{n \to \infty} \frac 1n \log \# \cM_{-k}^n = \lim_{n \to \infty} \frac 1n \log \# \hP_{-k}^n = \lim_{n \to \infty} \frac 1n \log \# \cP_{-k}^n \, .
\]
Finally,  the sequence $n\mapsto \log \# \cM^n_{-k}$ is subadditive.
\end{lemma}

\begin{proof}
First notice that by Lemma~\ref{lem:conn}, the elements of $\hP_{-k}^n$ are open, connected sets
whose boundaries are curves in $\cS_{-k-1} \cup \cS_{n+1}$.  Since the elements of 
$\cM_{-k-1}^{n+1}$ are the maximal open, connected sets with this property, it must be that
$\hP_{-k}^n$ is a refinement of $\cM_{-k-1}^{n+1}$.
Now suppose that the union of  $O_1, O_2 \in \hP_{-k}^n$ is contained in a single element 
$A \in \cM_{-k-1}^{n+1}$.  This is impossible since $\partial O_1, \partial O_2 \subset \cS_{-k-1} \cup \cS_{n+1}$, and at least part of these boundaries must lie inside $A$, contradicting the
definition of $A$.  So in fact, $\hP_{-k}^n = \cM_{-k-1}^{n+1}$.

We next show that the
limit in terms of $\# \cP_{-k}^n$ exists and is independent of $k$. It will follow that the
limits in terms of $\# \cM_{-k}^n$ and $\# \hP_{-k}^n$ exist and coincide using
the relation $\hP_{-k}^n = \cM_{-k-1}^{n+1}$ and Lemma~\ref{lem:card}.

Note that $\# \cP_{-j}^n \le \# \cP_{-k}^n$ whenever $0 \le j \le k$.
For fixed $k$,  we have
$
\# \cP_{-k}^{n+m} \le \# \cP_{-k}^n \cdot \# \Big(\bigvee_{i=1}^m T^{-n-i} \cP \Big) \, ,
$
and since $\# (\bigvee_{i=1}^m T^{-n-i} \cP) = \# (\bigvee_{i=1}^m T^{-i} \cP)$ because $T$ is invertible,
it follows that $\# \cP_{-k}^{n+m} \le \# \cP_{-k}^n \cdot \# \cP_{-k}^m$.
Thus $\log \# \cP_{-k}^n$ is subadditive as a function of $n$, and
the limit in $n$ converges for each $k$.  Applying this to $k=0$ implies that
the limit defining $h_*$ in Definition~\ref{def} exists.

Similar considerations show that
$\# \cP_{-k}^n \le \# \cP_{-k}^0 \cdot \# \cP_0^n$, and so 
\[
h_* = \lim_{n \to \infty} \frac 1n \log \# \cP_0^n \le \lim_{n \to \infty} \frac 1n \log \# \cP_{-k}^n \le 
\lim_{n \to \infty} \frac 1n (\log \# \cP_{-k}^0 + \log \# \cP_0^n) = h_* \, ,
\] 
so that the limit exists and  is independent of $k$.

For the final claim, we shall see that $\log \# \hP^n_{-k}$ is subadditive for essentially the same
reason as $\log \# \cP^n_{-k}$:
Take an (nonempty) element $P$ of $\hP^{n+m}_1$.  It is the interior of an
intersection of elements of the form
$T^{-j}A_j$ for some $A_j$ in $\cP$, for $j = 1$ to $n+m$.
This is equal to the intersection of the interiors of $T^{-j}A_j$.
But, since $P$ is nonempty, none of the $T^{-j}A_j$ can have empty interior and
so none of the $A_j$ can have empty interior.  Thus the interiors of $A_j$ are
in $\hP$ as well.
Now, splitting the intersection of the first $n$ sets from the last $m$, we see
that the intersection of the first $n$ sets form an element of $\hP^n_1$.  For
the last $m$ sets, we can factor out $T^{-n}$ at the price of making the set a
bit bigger:
$$\mbox{int}\, (T^{-n-j}(A_{-n-j})) \subseteq T^{-n}(\mbox{int}\,(T^{-j}(A_{-n-j})))\, ,
$$
where $\inter(\cdot)$ denotes the interior of a set.
Doing this for $j = 1$ to $m$, we see that this intersection is contained
in $T^{-n}$ of an element of $\hP^m_1$.
It follows that 
$\# \hP^{n+m}_1 \le \# \hP^n_1 \cdot \# \hP^m_1$,
so taking logs, the sequence is subadditive.  And then so is the sequence
with $\cM_0^n$ in place of $\hP_1^{n-1}$.
\end{proof}

\subsection{Comparing $h_*$ with the Bowen Definitions}

We set $\diam^s(\cM_{-k}^n)$ equal to the maximum length of a stable curve in any element of
$\cM_{-k}^n$.  Similarly, $\diam^u(\cM_{-k}^n)$ denotes the maximum length of an unstable curve
in any element of $\cM_{-k}^n$ while $\diam(\cM_{-k}^n)$ denotes the maximum diameter of
any element of $\cM_{-k}^n$.

The following lemma gives the first claim of \eqref{seppspann} in Theorem~\ref{thm:h_*}:

\begin{lemma}
\label{lem:hsep}
$h_* = \hsep$.
\end{lemma}

\begin{proof}
Fix $\ve > 0$.  Let $\Lambda = 1 + 2\cK_{\min} \tau_{\min}$ denote the lower bound on the hyperbolicity
constant for $T$ as in \eqref{eq:hyp}.  Choose $k_\ve$ large enough that $\diam^s(\cM_{-k_\ve-1}^0) \le C_1^{-1}\Lambda^{-k_\ve} < c_1 \ve$, for
some $c_1 > 0$ to be chosen below.  It follows that 
$$\diam^u(\cM_{-k_\ve-1}^{n+1}) \le C_1^{-1}\Lambda^{-n} < c_1 \ve$$
 for each $n \ge k_\ve$.  Using the uniform transversality of stable and unstable cones, we may choose $c_1>0$ such that $\diam(\cM_{-k_\ve-1}^{n+1}) < \ve$ for all $n \ge k_\ve$.

Now for $n \ge k_\ve$, let $E$ be an $(n,\ve)$-separated set.  Given
$x, y \in E$, we will show that $x$ and $y$ cannot belong to the same set 
$A \in \hP_{-k_\ve}^{k_\ve + n}$.

Since $x,y \in E$, there exists $j \in [0,n]$ such that $d(T^j(x), T^j(y)) > \ve$.  If $x \in A \in \hP_{-k_\ve}^{k_\ve +n}$, then $x \in \cap_{i=-k_\ve}^{k_\ve+n} \inter(T^{-i}P_i)$ for some choice of $P_i \in \cP$.  Then
\begin{equation}
\label{eq:include}
T^jx \in \cap_{i=-k_\ve-j}^{k_\ve + n-j} T^{-i}P_{i+j} \subset \cap_{-k_\ve}^{k_\ve} T^{-i}P_{i+j} \in \cP_{-k_\ve}^{k_\ve}\,  .
\end{equation}
Note that the element of $\cP_{-k_\ve}^{k_\ve}$ to which $T^j (x)$ belongs must have nonempty interior
since $T^{-i}P_i$ has non-empty interior for each $i \in [-k_\ve, k_\ve +n]$.
If $y \in A$, then $T^jy$ would belong to the same element of $\cP_{-k_\ve}^{k_\ve}$, which
is impossible since $\diam(\hP_{-k_\ve}^{k_\ve}) < \ve$ and taking the closure of such sets
does not change the diameter.  

Thus $x, y \in E$ implies that $x$ and $y$ cannot belong to the same
element of $\cP_{-k_\ve}^{k_\ve + n}$ with nonempty interior.  On the other hand, if $x$ belongs
to an element of $\cP_{-k_\ve}^{k_\ve+n}$ with empty interior, then indeed the element containing
$x$ is an isolated point, and $y$ cannot belong to the same element.
Thus $\# E \le \# \cP_{-k_\ve}^{k_\ve+n}$.

Since this bound holds for every $(n, \ve)$-separated set, we have
$r_n(\ve) \le \# \cP_{-k_\ve}^{k_\ve+n}$.  Thus,
\[
\lim_{n\to \infty} \frac 1n \log r_n(\ve) \le \lim_{n \to \infty} \frac 1n \log \# \cP_{-k_\ve}^{k_\ve+n}
=  h_* \,  .
\]
Since this bound holds for every $\ve > 0$, we conclude $\hsep \le h_*$.

To prove the reverse inequality, we claim that there exists $\ve_0>0$, 
independent of $n \ge 1$ and depending only on the
table $Q$, such that 
\begin{equation}
\label{eq:epsilon 0}
\mbox{ if $x,y$ lie in different elements of $\cM_0^n$, then $d_n(x,y) \ge \ve_0$.}
\end{equation}
To each point $x$  in an element of $\cM_0^n$, we can associate an itinerary
$(i_0, i_1, \ldots i_n)$ such that $T^{i_j}(x) \in M_{i_j}$.  If $x,y$ have different itineraries,
then for some $0 \le j \le n$, the points $T^j(x)$ and $T^j(y)$ lie in different components $M_i$, and so
by definition \eqref{eq:d_n} we have, $d_n(x,y) = 10 D \cdot \max_i \diam(M_i)$.  

Now suppose $x, y$ lie in different elements of $\cM_0^n$, but have the 
same itinerary.  By definition of $\cM_0^n$, the elements containing $x$ and $y$ are separated
by curves in $\cS_n$.  Let $j$ be the minimum index of such a curve.  Then $T^{j-1}(x)$ and
$T^{j-1}(y)$ lie on different sides of a curve in $\cS_1 \setminus \cS_0$.  
Due to the finite horizon condition (our slightly stronger version is needed here), there
exists $\ve_0 > 0$, depending only on the structure of $\cS_1$, such that the two one-sided
$\ve_0$-neighbourhoods of each curve in $\cS_1 \setminus \cS_0$ are mapped at least $\ve_0$ apart.
Thus either $d(T^{j-1}(x), T^{j-1}(y)) \ge \ve_0$ or $d(T^j(x), T^j(y)) \ge \ve_0$. 

With the claim proved, fix $n \in \mathbb{N}$ and $\ve \le \ve_0$, and define $E$ to be a set comprising exactly one point
from each element of $\cM_0^n$.  Then by the claim, $E$ is $(n, \ve)$-separated, so that 
$\# \cM_0^n \le r_n(\ve)$ for each $\ve \le \ve_0$.  Taking $n \to \infty$ and $\ve \to 0$
yields $h_* \le \hsep$.
\end{proof}

The following lemma gives the second claim of \eqref{seppspann} in Theorem~\ref{thm:h_*}:

\begin{lemma}
\label{lem:hspan}
$h_* = \hspan$.
\end{lemma}

\begin{proof}
Fix $\ve > 0$ and choose $k_\ve$ as in the proof of Lemma~\ref{lem:hsep} so that 
$$\diam(\cM_{-k_\ve-1}^{n+1}) < \ve$$ for all $n \ge k_\ve$.  
Choose one point $x$ in each element of $\cP_{-k_\ve}^{k_\ve+n}$, and let $F$ denote the
collection of these points.  We will show that $F$ is an $(n,\ve)$-spanning set for $T$.

Let $y \in M$ and let $B_y$ be the element of $\cP_{-k_\ve}^{k_\ve+n}$ containing $y$.
If $B_y$ is an isolated point, then $y \in F$ and there is nothing to prove.  Otherwise,
let $x_y = F \cap B_y$.  For each $j \in [0,n]$, using the analogous calculation as in \eqref{eq:include},
we must have $T^j(y) , T^j (x_y) \in B_j \in \cP_{-k_\ve}^{k_\ve}$.  Since $\diam(\cP_{-k_\ve}^{k_\ve}) < \ve$,
this implies $d(T^j(y) , T^j(x_y)) < \ve$ for all $j \in [0,n]$.  Thus $F$ is an $(n,\ve)$-spanning set.
We have,
\[
\lim_{n \to \infty} \frac 1n \log s_n(\ve) \le \lim_{n\to \infty} \frac 1n \log \# \cP_{-k_\ve}^{k_\ve+n}
= h_* \, .
\]
Since this is true for each $\ve>0$, it follows that $\hspan \le h_*$.

To prove the reverse inequality, recall $\ve_0$ from the proof of Lemma~\ref{lem:hsep}.
For $\ve < \ve_0$ and $n \in \mathbb{N}$, let $F$ be an $(n,\ve)$-spanning set.  We claim
$\# F \ge \# \cM_0^n$.  Suppose not.  Then there exists $A \in \cM_0^n$
which contains no elements of $F$.  Let $y \in A$ and let $x \in F$.  By the claim in the
proof of Lemma~\ref{lem:hsep}, $d_n(x,y) \ge \ve_0$ since $x$ and $y$ lie in different
elements of $\cM_0^n$.  Since this holds for all $x \in F$,
it contradicts the fact that $F$ is an $(n,\ve)$-spanning set.

Since this is true for each $(n,\ve)$-spanning set for $\ve < \ve_0$, we conclude that
$s_n(\ve) \ge \# \cM_0^n$, and taking appropriate limits, $\hspan \ge h_*$.
\end{proof}

\subsection{Easy Direction of the Variational Principle for $h_*$}

Recall that given a $T$-invariant probability measure $\mu$ and a finite
measurable partition $\mathcal{A}$ of $M$, the entropy of $\mathcal{A}$ with respect to $\mu$ is
defined by $H_\mu(\mathcal{A}) = - \sum_{A \in \mathcal{A}} \mu(A) \log \mu(A)$, and the
entropy of $T$ with respect to $\mathcal{A}$ is
$
h_\mu(T, \mathcal{A}) = \lim_{n \to \infty} \frac 1n H_\mu\left(\bigvee_{i=0}^{n-1} T^{-i} \mathcal{A}\right) 
$.

The following lemma gives the bound
\eqref{varprinc1} in Theorem~\ref{thm:h_*}:

\begin{lemma}
\label{lem:geq}
$h_* \ge \sup \{ h_\mu(T) : \mbox{$\mu$ is a $T$-invariant Borel probability measure} \}$.
\end{lemma}

\begin{proof}
Let $\mu$ be a $T$-invariant probability measure on $M$.  We note that
$\cP$ is a generator for $T$ since $\bigvee_{i=-\infty}^\infty T^{-i} \cP$ separates points in $M$. 
Thus $h_{\mu}(T) = h_{\mu}(T, \cP)$ (see for example \cite[Thm~4.17]{walters}).
Then,
\[
\begin{split}
h_\mu(T, \cP) & = \lim_{n \to \infty} \frac 1n H_\mu\left(\bigvee_{i=0}^{n-1} T^{-i}\cP\right) 
= \lim_{n \to \infty} \frac 1n H_\mu(\cP_0^{n-1}) 
\le \lim_{n \to \infty} \frac 1n \log (\# \cP_0^{n-1}) = h_* \, .
\end{split}
\]
Thus $h_\mu(T) \le h_*$ for every $T$-invariant
probability measure $\mu$. 
\end{proof}

 \section{The Banach Spaces $\BB$ and $\BB_w$ and the Transfer Operator $\LL$}\label{transfert}
   The
measure of maximal entropy for the billiard map
$T$ will be constructed out of  left and right eigenvectors
of  a transfer operator $\cL$ associated with the billiard map and acting on
suitable spaces $\cB$ and $\cB_w$ of anisotropic distributions. In this section we define these
objects, state and prove the main bound, Proposition~\ref{prop:ly}, on the transfer operator, and deduce from it Theorem~\ref{sprad}, showing that
the spectral radius of $\cL$ on $\cB$ is $e^{h_*}$. 

Recalling that the stable Jacobian of $T$ satisfies $J^sT \approx \cos \vf$ \cite[eq.~(4.20)]{chernov book}, 
the relevant transfer operator is defined on measurable functions $f$ by 
\begin{equation}\label{opp}
\cL f = \frac{f \circ T^{-1}}{J^sT \circ T^{-1}}\, .
\end{equation}

In order to define the Banach spaces of distributions on which the operator $\cL$ will act,
we need preliminary notations:
 Let
$\cW^s$ denote the set of all nontrivial connected subsets $W$
of stable manifolds for $T$ so that $W$ has  length at most $\delta_0 >0$,
where $\delta_0< 1$ will be chosen after \eqref{eq:m}, using the growth Lemma~\ref{lem:growth}.  Such curves have curvature bounded above
by a fixed constant \cite[Prop~4.29]{chernov book}.  Thus,
$T^{-1}\cW^s = \cW^s$, up to subdivision of curves.

For every $W\in \cW^s$, let $C^1(W)$
denote the space of $C^1$ functions on $W$ and for  every $\alpha \in (0,1)$ we
let $C^\alpha(W)$ denote the closure\footnote{Working with
the closure of $C^1$  will give injectivity of the inclusion
of the strong space in the weak.} of $C^1(W)$ for the
$\alpha$-H\"older norm  $|\psi|_{C^\alpha(W)}=\sup_W|\psi|+ H_W^\alpha(\psi)$,
where
\begin{equation}\label{ladefinition}
H_W^\alpha(\psi)=\sup_{\substack{x, y \in W \\ x \neq y}}\frac{|\psi(x)-\psi(y)|}{d(x,y)^\alpha}\, .
\end{equation}
We write 
$\psi \in C^\alpha(\cW^s)$  
if $\psi \in C^\alpha(W)$
for all $W \in \cW^s$,   with uniformly bounded H\"older norm.

\subsection{Definition of Norms and of the Spaces $\cB$ and $\cB_w$}
\label{normdef}

Since the stable cone $C^s$ is bounded away from the vertical, we may view each stable curve
$W \in \cW^s$ as the graph of a function $\vf_W(r)$ of the arclength coordinate $r$ ranging
over some interval $I_W$, i.e.,
\begin{equation}
\label{eq:graph}
W = \{ G_W(r) := (r, \vf_W(r)) \in M : r \in I_W \}\, .
\end{equation}
Given two curves $W_1, W_2 \in \cW^s$, we may use this representation to
define a distance\footnote{$d_{\cW^s}$ is not a metric since it does not satisfy the triangle
inequality; however, it is sufficient for our purposes to produce a usable notion of distance between stable manifolds.
See \cite[Footnote 4]{dzr} for a modification of $d_{\cW^s}$ which does satisfy the triangle inequality.} between them: Define 
\[
d_{\cW^s}(W_1, W_2) = | I_{W_1} \bigtriangleup I_{W_2}| + |\vf_{W_1} - \vf_{W_2}|_{C^1(I_{W_1} \cap I_{W_2})}
\]
if $I_{W_1} \cap I_{W_2} \neq \emptyset$.  Otherwise, set $d_{\cW^s}(W_1, W_2) = \infty$.

Similarly, given two test functions $\psi_1$ and $\psi_2$ on $W_1$ and $W_2$, respectively,
we define a distance between them by 
\[
d(\psi_1, \psi_2) = |\psi_1 \circ G_{W_1} - \psi_2 \circ G_{W_2} |_{C^0(I_{W_1} \cap I_{W_2})} \, ,
\]
whenever $d_{\cW^s}(W_1, W_2) < \infty$. Otherwise, set $d(\psi_1, \psi_2) = \infty$.

\smallskip
We are now ready to introduce the norms used to define the spaces $\cB$ and $\cB_w$.
Besides $\delta_0\in (0,1)$,
and a constant $\varepsilon_0>0$ to appear below, they will depend on positive real numbers $\alpha$, $\beta$, 
$\gamma$, and $\varsigma$ so that, recalling $s_0\in (0,1)$ from\footnote{If  $\gamma > 1$, we can get 
good bounds in Theorem~\ref{nbhdthm}. This is only possible if $h_*>s_0 \log 2$.}
\eqref{defs0},
\begin{equation}
\label{eq:gamma}
0 < \beta < \alpha \le 1/3\, , \, 
\, \quad 1<  2^{s_0\gamma} < e^{h_*} \, , \, \, \quad 0< \varsigma < \gamma \, .
\end{equation}
(The condition $\alpha \le 1/3$
 is used in Lemma~\ref{lem:embed bound} which is used to prove embedding
into distributions. The number $1/3$ comes from the $1/k^2$ decay in the
width of homogeneity strips \eqref{homogs}. The upper bound on $\gamma$ arises from
use of the growth lemma
from Section~\ref{GL0}.  See \eqref{eq:m}.)

For $f \in C^1(M)$, define the weak norm of $f$ by
\[
| f |_w = \sup_{W \in \cW^s} \sup_{\substack{\psi \in C^\alpha(W) \\ |\psi|_{C^\alpha(W)} \le 1}}
\int_W f \, \psi \, dm_W \, .
\]
Here, $dm_W$ denotes unnormalized Lebesgue (arclength) measure on $W$.

Define the strong stable norm of $f$ by\footnote{The logarithmic modulus of continuity 
 in $\|f\|_s$ is used to obtain a finite spectral radius.}
\[
\| f \|_s = \sup_{W \in \cW^s} \sup_{\substack{\psi \in C^\beta(W) \\ |\psi|_{C^\beta(W)} \le |\log |W||^\gamma}}
\int_W f \, \psi \, dm_W \, ,
\]
(note that $| f |_w \le \max \{1, |\log \delta_0|^{-\gamma}\} \| f \|_s$).
Finally, for $\varsigma \in (0,  \gamma)$, define
the strong unstable norm\footnote{The logarithmic modulus of continuity appears in $\|f\|_u$ because
of the logarithmic modulus of continuity in $\|f\|_s$. Its presence in $\|f\|_u$
causes  the loss of the
spectral gap.} of $f$ by
\[
\| f \|_u = \sup_{\ve \le \ve_0} \sup_{\substack{W_1, W_2 \in \cW^s \\ d_{\cW^s}(W_1, W_2) \le \ve}}
\sup_{\substack{\psi_i \in C^\alpha(W_i) \\ |\psi_i|_{C^\alpha(W_i)} \le 1 \\ d(\psi_1, \psi_2) = 0}} 
|\log \ve|^\varsigma \left| \int_{W_1} f \, \psi_1 \, dm_{W_1} - \int_{W_2} f \, \psi_2 \, dm_{W_2} \right| \,  .
\]

\begin{definition}[The Banach spaces]
The space $\cB_w$ is  the completion of $C^1(M)$ with respect to the weak norm
$| \cdot |_w$, while $\cB$ is the completion of $C^1(M)$ with respect to the strong norm,
$\| \cdot \|_{\cB} = \| \cdot \|_s +  \| \cdot \|_u$.
\end{definition}

In the next subsection, we shall prove
 the continuous embeddings $\BB \subset \BB_w \subset (C^1(M))^*$, i.e., elements
 of our Banach spaces are distributions of order at most one (see Proposition~\ref{distembed}).
 Proposition~\ref{cpte}  in  Section~\ref{cembedsec} gives the compact embedding of the unit ball of
 $\cB$ in $\cB_w$. 
 
\subsection{Embeddings into Distributions on $M$}\label{embedsec}

In this section we describe elements of our Banach spaces $\cB\subset \cB_w$
as
distributions of order at most one on $M$.
(This does not follow from the corresponding result in \cite{demzhang11}, 
in particular since
we use exact stable leaves to define our norms.) We will actually show that they belong to the dual
of a space $C^\alpha(\cW^s_{\bH})$ containing $C^1(M)$ that we define next:
We did not require elements of $\cW^s$ to be homogeneous. Now,
defining the usual homogeneity strips 
\begin{equation}\label{homogs}
\bH_k = \big\{ (r, \vf) \in M_i : \tfrac{\pi}{2} - \tfrac{1}{k^2} \le \vf \le \tfrac{\pi}{2} - \tfrac{1}{(k+1)^2} \big\}, \quad k \ge k_0 \, ,
\end{equation}
and analogously for $k \le -k_0$, 
we define $\cW^s_{\bH} \subset \cW^s$ to denote those stable manifolds 
$W \in \cW^s$ such that $T^nW$ 
lies in a single homogeneity strip for all $n \ge 0$.
We write 
$\psi \in C^\alpha(\cW^s_{\bH})$
if $\psi \in C^\alpha(W)$
for all $W \in \cW^s_{\bH}$ with uniformly bounded H\"older norm.
Similarly, we define $C^\alpha_{\cos}(\cW^s_{\bH})$ to comprise the set of functions
$\psi$ such that $\psi \cos \vf \in C^\alpha(\cW^s_{\bH})$.
Clearly $C^\alpha(\cW^s_{\bH}) \subset C^\alpha_{\cos}(\cW^s_{\bH})$.

Due to the uniform hyperbolicity  \eqref{eq:hyp} of $T$ and the invariance of $\cW^s$
and $\cW^s_{\bH}$, 
if $\psi \in C^\alpha(\cW^s)$ (resp. $C^\alpha(\cW^s_{\bH})$),
then $\psi \circ T \in C^\alpha(\cW^s)$ (resp. $C^\alpha(\cW^s_{\bH})$).  Also, since the stable Jacobian of $T$
satisfies $J^sT \approx \cos \vf$ \cite[eq.~(4.20)]{chernov book} and is $1/3$ log-H\"older continuous
on elements of $\cW^s_{\bH}$ \cite[Lemma~5.27]{chernov book}, then 
$\frac{\psi \circ T}{J^s T} \in C^\alpha_{\cos}(\cW^s_{\bH})$ for any $\alpha \le 1/3$. 

We can now state our first embedding result.
An  embedding $\BB_w  \subset (\cF)^*$
(for  $\cF=C^1(M)$ or  $\cF=C^\alpha(\cW^s_{\bH})$)
is understood in the following sense:
 for  $f \in \cB_w$ there exists $C_f<\infty$ such that, letting $f_n\in C^1(M)$ be a sequence 
 converging to $f$ in the $\cB_w$ norm,  
 for every $\psi\in \cF$ the following limit exists
 \begin{equation}\label{defd}
 f(\psi)=\lim_{n \to \infty} \int f_n \psi \, d\musrb
 \end{equation}
 and satisfies $|f(\psi)|\le C_f\|\psi\|_{\cF}$.

\begin{proposition}[Embedding into Distributions]
\label{distembed}
The continuous embeddings 
$$C^1(M) \subset \BB \subset \BB_w  \subset (C^\alpha(\cW^s_{\bH}))^*\subset (C^1(M))^*$$
 hold,
the first two embeddings\footnote{We do not expect the third embedding  to be injective, due to the
 logarithmic weight in the norm.}  being injective. 
Therefore, since $C^1(M)\subset \BB\subset \BB_w$ injectively
and continuously, we have
$$(\BB_w)^*\subset \BB^* \subset (C^1(M))^*\, .$$
\end{proposition}

\begin{remark}[Radon Measures]\label{distembed*}
Proposition~\ref{distembed} has the following important
consequence: If $f \in \BB_w$ is such that $f(\psi)$ defined by \eqref{defd} is nonnegative for all
nonnegative $\psi \in \cF=C^1(M)$, then, by Schwartz's \cite[\S I.4]{Sch} generalisation of the Riesz representation
 theorem, it defines an element of the dual of $C^0(M)$, i.e.,
a Radon measure on $M$. If, in addition, $f(\psi)=1$ for
$\psi$ the constant function $1$, then this measure is a probability measure.
\end{remark}

\smallskip

The following lemma is important for the third inclusion in Proposition~\ref{distembed}.
Recalling \eqref{ladefinition}, we define 
$H^\alpha_{\cW^s_{\bH}}(\psi) = \sup_{W \in \cW^s_{\bH}} H^\alpha_W(\psi)$.

\begin{lemma}
\label{lem:embed bound}
There exists $C>0$ such that for any $f \in \cB_w$ and $\psi \in C^\alpha(\cW^s_{\bH})$, recalling \eqref{defd},
\[
|f(\psi)| \le C |f|_w \big(|\psi|_\infty + H^\alpha_{\cW^s_{\bH}}(\psi) \big) \, .
\]
\end{lemma}

\begin{proof}
By density it suffices to prove the inequality for $f \in C^1(M)$.
Let $\psi \in C^\alpha(\cW^s_{\bH})$.  Since by our convention, 
we identify $f$ with the measure $f d\musrb$, we must estimate,
\[
f(\psi) = \int f \, \psi \, d\musrb \, .
\]
In order to bound this integral, we disintegrate the measure $\musrb$ into
conditional probability measures $\musrb^{W_\xi}$ on maximal homogeneous stable manifolds 
$W_\xi \in \cW^s_{\bH}$
and a factor measure $d\hatmusrb(\xi)$ on the index set $\Xi$ of homogeneous stable manifolds;
thus $\cW^s_{\bH} = \{ W_\xi \}_{\xi \in \Xi}$.  According to the time reversal counterpart
of \cite[Cor~5.30]{chernov book}, the conditional measures $\musrb^{W_\xi}$ have
smooth densities with respect to the arclength measure on $W_\xi$, i.e.,
$d\musrb^{W_\xi} = |W_\xi|^{-1} \rho_\xi dm_{W_\xi}$, where $\rho_{\xi}$ is log-H\"older continuous
with exponent 1/3.  Moreover, $\sup_{\xi \in \Xi} |\rho_{\xi}|_{C^\alpha(W_\xi)} =: \bar C < \infty$
since $\alpha \le 1/3$. 

Using this disintegration, we estimate\footnote{This is  where we use  $f \musrb$:   Replacing $\hatmusrb$ by the
factor measure with respect to Lebesgue, this integral would be infinite.  
Using $\cW^s$ rather than $\cW^s_{\bH}$ may produce a finite integral 
with respect to Lebesgue, but the $\rho_\xi$ may not be
uniformly H\"older continuous on the longer curves.}  the required integral:

\begin{align}
\label{needsplit}|f(\psi)| & = \left| \int_{\xi \in \Xi} \int_{W_\xi} f \, \psi \, \rho_\xi \, |W_\xi|^{-1} dm_{W_\xi} \, d\hatmusrb(\xi) \right|
\\
\nonumber & \le \int_{\xi \in \Xi} |f|_w |\psi|_{C^\alpha(W_\xi)} |\rho_\xi|_{C^\alpha(W_\xi)} |W_\xi|^{-1} d\hatmusrb(\xi) \\
\nonumber & \le \bar C |f|_w \big( |\psi|_\infty + H^\alpha_{\cW^s_{\bH}}(\psi) \big) 
\int_{\xi \in \Xi} |W_\xi|^{-1} d\hatmusrb(\xi) \, .
\end{align}

This last integral is precisely that in \cite[Exercise~7.15]{chernov book} which measures the 
relative frequency of short curves in a standard family.  Due to \cite[Exercise~7.22]{chernov book},
the SRB measure decomposes into a proper family, and so this integral is finite.
\end{proof}

\begin{proof}[Proof of Proposition~\ref{distembed}]
The continuity and injectivity of the embedding of $C^1(M)$ into $\cB$ are clear from the definition.
The inequality $| \cdot |_w \le \| \cdot \|_s$ implies the continuity of $\cB \hookrightarrow \cB_w$,
while the injectivity follows from the definition of $C^\beta(W)$ as the closure of $C^1(W)$
in the $C^\beta$ norm, as described at the beginning of Section~\ref{transfert},
so that $C^\alpha(W)$ is dense in $C^\beta(W)$.

Finally, since $C^1(M) \subset C^\alpha(\cW^s_{\bH})$, the continuity of the third and fourth inclusions
follow from Lemma~\ref{lem:embed bound}. 
 \end{proof}

 \subsection{The Transfer Operator} 
 
 We now move to the key bounds on the transfer operator. First, we revisit the
 definition \eqref{opp} in order to let $\cL$ act on $\cB$ and $\cB_w$:
 We may define the
transfer operator $\cL : (C^\alpha_{\cos}(\cW^s_{\bH}))^* \to (C^\alpha(\cW^s))^*$
by
\[
\cL f (\psi) = f \big(\tfrac{\psi \circ T}{J^sT} \big), \quad \psi \in C^\alpha(\cW^s) \, .
\]

When $f \in C^1(M)$, we identify $f$ with
the measure\footnote{To show the claimed inclusion just use that $d\musrb=(2  |\partial Q|)^{-1} \cos \vf \, drd\vf$.} 
\begin{equation}\label{iddent}
f d\musrb \in (C^\alpha_{\cos}(\cW^s_{\bH}))^* \, . 
\end{equation}
The measure above 
is (abusively) still denoted by $f$.  For $f\in C^1(M)$ 
the transfer operator then indeed takes the  form
$\cL f = (f /J^sT )\circ T^{-1}$ announced in \eqref{opp}
since, due to our identification \eqref{iddent}, we have $\cL f (\psi) = \int \cL f \, \psi \, d\musrb = \int f \, \frac{\psi \circ T}{J^sT} \, d\musrb$.

\begin{remark}[Viewing $f\in C^1$ as a measure]If we viewed instead  $f$ as the measure $f dm$, 
it is not clear whether the embedding Lemma~\ref{lem:embed bound} 
would still hold since the weight $\cos W$ (crucial
to \cite[Lemma~3.9]{demzhang11}) is absent from the norms.  
Along these lines,  we do not claim that Lebesgue measure belongs to
our Banach spaces.

Slightly modifing \cite{demzhang11} due to the lack of homogeneity strips, we could replace $| \psi |_{C^\alpha(W)} \le 1$ 
by $|\psi \cos \vf|_{C^\alpha(W)} \le 1$ in our norms.  Then
it would be natural to view $f$ as $f dm$, and the embedding Lemma~\ref{lem:embed bound} 
would hold, but the transfer operator would have the form
$$
\LL_{cos} f = \frac{ f \circ T^{-1} }{ (J^sT \circ T^{-1})( JT \circ T^{-1})}\, ,
$$
where $JT$ is
the full Jacobian of the map (the ratio of cosines).  We do not make such a change since it would only
complicate our estimates unnecessarily.  Note that
the potentials of the operators $\cL$ and $\cL_{cos}$ differ by a coboundary, giving the
same spectral radius.
\end{remark}

\smallskip
 It follows from submultiplicativity
of $\# \cM^n_0$ that $e^{n h_*} \le \# \cM_0^n$ for all $n$. In Section~\ref{supera}, we shall
prove the supermultiplicativity statement Lemma~\ref{lem:super} from which we deduce
the following upper bound for $\# \cM^n_0$:

\begin{proposition}[Exact Exponential Growth]
\label{cor:exp}
Let $c_1>0$ be given by  Lemma~\ref{lem:super}. Then for all $n \in \mathbb{N}$, we have
$
e^{n h_*} \le \# \cM_0^n \le \tfrac{2}{c_1} e^{n h_*} 
$.
\end{proposition}

The following proposition (proved
in Section~\ref{sec:proof prop}) gives the key norm estimates.

\begin{proposition}
\label{prop:ly}
Let $c_1$ be as in Proposition~\ref{cor:exp}.
There exist $\delta_0$,  $C > 0$, 
and $\varpi \in (0,1)$ such 
that\footnote{\label{pied}In fact the strong stable norm satisfies a stronger inequality:
$\| \cL^n f \|_s \le \frac{C}{c_1 \delta_0} ( \sigma^n \| f \|_s + |f|_w )e^{n h_*}$,
for some $\sigma < 1$.  We omit
the proof since we do not use this.}
 for all  $f \in \cB$,
\begin{align}
| \cL^n f |_w & \le  \frac{C}{c_1 \delta_0}  e^{nh_*} | f |_w \, , \quad \forall n\ge 0  \; ; \label{eq:weak ly} \\
\|\cL^n f\|_s & \leq   \frac{C}{c_1 \delta_0} e^{nh_*}  \|f\|_s   \, , \quad \forall n\ge 0 \; ; \label{cheapeq:stable ly} \\
\| \cL^n f \|_u & \le  \frac{C}{c_1 \delta_0} (\| f \|_u + n^\varpi \| f \|_s)  e^{nh_*}  \, , \quad \forall n \ge 0 \; . \label{eq:unstable ly}
\end{align}
If $h_*> s_0 \log 2$ (where $s_0 <1$ is defined by \eqref{defs0}) then in addition
there exist $\varsigma>0$ and  $C > 0$ such that for all $f \in \cB$
\begin{equation}
\| \cL^n f \|_u  \le  \frac{C}{c_1 \delta_0} (\| f \|_u +  \| f \|_s) e^{nh_*} \, , \quad \forall n \ge 0\; . \label{eq:unstable lyb}
\end{equation}
\end{proposition}

\begin{remark}\label{4.8}
Replacing $|\log \epsilon|$ by $\log |\log \epsilon|$
in the definition of $\|f\|_u$, we can replace $n^\varpi$ by a logarithm in \eqref{eq:unstable ly}.
\end{remark}

In spite of compactness of the
embedding $\cB \subset \cB_w$ (Proposition~\ref{cpte}),
 the above bounds do {\it not} deserve to be called Lasota--Yorke estimates since 
(even  replacing $\| \cdot \|_s +  \| \cdot \|_u$
   by  $ \| \cdot \|_s + c_u \| \cdot \|_u $  for small $c_u$ and using footnote \ref{pied}) they
do not lead to  bounds of the type $\|(e^{-h_*}\cL)^n f\|_\cB \le \sigma^n  \|f\|_\cB + K_n |f|_w$
for some  $\sigma <1$ and finite constants $K_n$.
We will nevertheless sometimes refer to them as  ``Lasota--Yorke'' estimates, in quotation marks.

Proposition~\ref{prop:ly} combined with the following lemma imply that
$\cL$ is a bounded operator on both $\cB$ and $\cB_w$:

\begin{lemma}[Image of a $C^1$ Function]
For any $f \in C^1(M)$ the image $\cL f\in (C^\alpha(\cW^s))^*$ is the limit of
a sequence of $C^1$ functions in
the strong norm $\|\cdot\|_{\cB}$. 
\end{lemma}

\begin{proof} Since our norms are weaker than the
norms of  \cite{demzhang11} (modulo the use of homogeneity layers
there), the statement follows from replacing $\cL_{\mbox{\tiny{SRB}}}$ by $\cL$ 
 in
the proofs of Lemmas 3.7 and 3.8 in \cite{demzhang11}, and checking  that
the absence of homogeneity layers does not affect the computations.
\end{proof}

Proposition~\ref{prop:ly} gives the upper bounds
in the following result 
(the bounds \eqref{spl} and \eqref{spub} are needed to construct a nontrivial maximal
eigenvector in Proposition~\ref{prop:exist}):

\begin{theorem}[Spectral Radius of $\cL$ on $\cB$]\label{sprad}
There
exist $\varpi \in (0,1)$,  $C < \infty$ such that,
\begin{equation}\label{spu}
\| \cL^n\|_\cB \le   C n^\varpi e^{n h_*}\, , \quad \forall n \ge 0\, .
\end{equation}
There exists $C > 0$ such that, letting $1$ be the function
$f \equiv 1$, we have,
\begin{equation}\label{spl}
\| \cL^n 1 \|_s\ge | \cL^n 1 |_w \ge C e^{n h_*} \, , \quad \forall n \ge 0\, .
\end{equation}
Recalling  \eqref{eq:weak ly}, the spectral radius of $\cL$ on $\cB$
and
$\cB_w$ is thus equal to $\exp(h_*)>1$.

If $h_*> s_0 \log 2$ (with $s_0 <1$  defined by \eqref{defs0}) then, if
$\varsigma>0$ and $\delta_0>0$ are small enough, 
there exists $\widetilde C < \infty$ such that, 
\begin{equation}\label{spub}
\| \cL^n\|_\cB \le  \widetilde C e^{n h_*}\, , \quad \forall n \ge 0\, .
\end{equation}
\end{theorem}

The above theorem is proved in Subection~\ref{prspu}.


\section{Growth Lemma and Fragmentation Lemmas}
\label{sec:growth}

This section  contains combinatorial
growth lemmas,  controlling the growth in complexity of the iterates of a stable curve.
They will be used to prove the ``Lasota--Yorke'' Proposition~\ref{prop:ly}, to
show Lemma~\ref{lem:short}, used in Section~\ref{prspu}
to get the lower bound \eqref{spl} on the spectral radius, and  to
show absolute continuity in Section~\ref{notmixing}.

In view of the compact embedding Proposition~ \ref{cpte}, and
also to get Lemma~\ref{lem:long piece} from Lemma~\ref{lem:short}, we must work with a more general class of stable
curves: We define 
a set of cone-stable curves $\hW^s$  whose tangent
vectors all lie in the stable cone for the map, with length at most $\delta_0$ and curvature
bounded above so that $T^{-1} \hW^s \subset \hW^s$, up to subdivision of curves.  
Obviously, $\cW^s \subset \hW^s$. 
We define a set of cone-unstable curves $\hW^u$ similarly.

For $W \in \hW^s$, let $\cG_0(W) = W$.  For $n \ge 1$, define $\cG_n(W)=\cG_n^{\delta_0}(W)$ inductively as 
the smooth components of $T^{-1}(W')$ for $W' \in \cG_{n-1}(W)$, where elements longer than $\delta_0$ are subdivided to have length between $\delta_0/2$ and $\delta_0$.  
Thus $\cG_n(W) \subset \hW^s$ for each $n$ and
$\cup_{U \in \cG_n(W)} U = T^{-n} W$.  Moreover, if $W \in \cW^s$, then $\cG_n(W) \subset \cW^s$.

Denote by $L_n(W)$ those elements of $\cG_n(W)$ having length at least $\delta_0/3$, and define
$\cI_n(W)$ to comprise those elements $U \in \cG_n(W)$ for which $T^iU$ is not contained
in an element of $L_{n-i}(W)$ for $0 \le i \le n-1$.

A fundamental fact \cite[Lemma~5.2]{chernov01} we will use is that the growth in complexity for the billiard is at most linear:
\begin{equation}
\label{eq:complex}
\begin{split}
\mbox{$\exists$ $K >0$ \mbox{ such that } $\forall$ $n \ge 0$, } & \mbox{the number of curves in $\cS_{\pm n}$ that intersect} \\
& \mbox{at a single point is at most $Kn$.}
\end{split}
\end{equation}

\subsection{Growth Lemma}
\label{GL0}

Recall $s_0\in (0,1)$ from \eqref{defs0}.
We shall prove:

\begin{lemma}[Growth Lemma]
\label{lem:growth}
For any $m \in \bN$, there exists $\delta_0=\delta_0(m) \in (0,1)$ such that for all $n \ge 1$, all $\bar\gamma \in [0,\infty)$ and all
$W \in \hW^s$, we have
\begin{itemize}
  \item[a)]  $\displaystyle \sum_{W_i \in \cI_{n}(W)} \left( \frac{\log |W|}{\log |W_i|} \right)^{\bar\gamma} \le 2^{(n s_0+1)\bar\gamma}  (Km+1)^{n/m}$ ;
  \item[b)]  $\displaystyle \sum_{W_i \in \cG_{n}(W)} \left( \frac{\log |W|}{\log |W_i|} \right)^{\bar\gamma} $
  
  $\displaystyle  \quad\le \min \bigl\{ 2 \delta_0^{-1} 2^{(ns_0+1) \bar\gamma}  \# \cM_0^n, \;
  2^{2\bar\gamma +1} \delta_0^{-1} \sum_{j=1}^n 2^{j s_0\bar \gamma} (Km+1)^{j/m} \#\cM_0^{n-j} \bigr\}$.
\end{itemize}
Moreover, if $|W| \ge \delta_0/2$, then both factors $2^{(ns_0+1)\bar\gamma}$ can be replaced by 
$2^{\bar \gamma}$.
\end{lemma}

\begin{proof}
First recall that if $W \in \hW^s$ is short, then 
\begin{equation}
\label{eq:image}
|T^{-1}W| \le C|W|^{1/2}
\quad \mbox{for some constant $C\ge 1$, independent of $W \in \hW^s$},
\end{equation}
\cite[Exercise 4.50]{chernov book}.
The above bound can be iterated, giving $|T^{-\ell}W| \le C'|W|^{2^{-\ell}}$, where $C' \le C^2$, for any number
of consecutive ``nearly tangential'' collisions (collisions with angle $|\vf| > \vf_0$).
Since in every
$n_0$ iterates, we have at most $s_0n_0$ nearly tangential collisions 
and $(1-s_0)n_0$ iterates that expand at most by a constant factor $\Lambda_1 > 1$ depending
only on $\vf_0$, we see that
\begin{align*}
&|T^{-n_0}W| \le C|W|^{2^{-s_0n_0}} \Lambda_1^{(1-s_0)n_0}\\
&\implies |T^{-2n_0}W| \le C^{1+2^{-s_0n_0}} |W|^{2^{-2s_0n_0}} \Lambda_1^{(1-s_0)n_0 2^{-s_0n_0}} \Lambda_1^{(1-s_0)n_0}  \, .
\end{align*}
Iterating this inductively, we conclude
\begin{equation}
\label{eq:control}
|T^{-j}W| \le C'' |W|^{2^{-s_0j}} \quad \mbox{for all $j \ge 1$},
\end{equation}
where $C''\ge 1$ depends only on $n_0$ and $\Lambda_1$.  
Therefore, if $\delta_0$  is smaller than $1/C''$,  we have 
\[
\left( \frac{\log |W|}{\log |W_i|} \right)^{\bar\gamma} \le 
\left( 2^{s_0 n} \Big( 1 - \frac{\log C''}{\log |W_i|} \Big) \right)^{\bar\gamma} \le 2^{(ns_0+1)\bar\gamma}\, ,\, \,
\forall \, \, W_i \in \cG_n(W)\, , 
\]
since $|W_i| \le \delta_0$.
Note that if $|W_i| \le |W|$, then $\frac{\log |W|}{\log |W_i|} \le 1$, so that such curves
do not contribute large terms to the sums in parts (a) and (b) of the lemma.

\smallskip
\noindent
(a)  Using the above argument, for any $W \in \hW^s$, we may bound the ratio of logs by
 $2^{(n+1)s_0 \bar\gamma}$.  Moreover, if $|W| \ge \delta_0/2$, then since $|W_i| \le \delta_0<2$,
 we have
 \[
 \frac{\log |W|}{\log |W_i|} \le \frac{\log (\delta_0/2)}{\log \delta_0} = 1 - \frac{\log 2}{\log \delta_0} \le 2
 \, .
 \]

Now, fixing $m$ and using the linear bound on complexity, choose $\delta_0=\delta_0(m) >0$ such that
if $|W| \le \delta_0$, then $T^{-\ell}W$ comprises at most $K\ell+1$ connected 
components for $0 \le \ell \le 2m$. Such a choice is always possible by \eqref{eq:image}. 
Then for $n = mj + \ell$, we split up the orbit into $j-1$ increments of length $m$ and the last increment
of length $m + \ell$.  Part (a) then follows by a simple induction, 
since elements of $\cI_{mj}(W)$ must be formed
from elements of $\cI_{m(j-1)}(W)$ which have been cut by singularity curves in $\cS_{-m}$.
At the last step, this estimate also holds for elements of which have been cut by singularity
curves in $\cS_{-m-\ell}$ by choice of $\delta_0$.

\smallskip
\noindent
(b) The bound on the ratio of logs is the same as in part (a).  
The first bound on the cardinality
of the sum follows by noting that each element of $\cG_n(W)$ is contained in one element
of $\cM_0^n$.  Moreover, due to subdivision of long pieces, there can be no more than
$2 \delta_0^{-1}$ elements of $\cG_n(W)$ in a single element of $\cM_0^n$.

For the second bound in part (b), we may assume that $|W| < \delta_0/2$; otherwise,
we may bound the sum by $2^{\bar \gamma + 1} \delta_0^{-1} \# \cM_0^n$, which is optimal
for what we need.  For $|W| < \delta_0/2$, let $F_1(W)$ denote those $V \in \cG_1(W)$
whose length is at least $\delta_0/2$.  Inductively, define $F_j(W)$, 
for $2 \le j \le n-1$, to contain those $V \in \cG_j(W)$
whose length is at least $\delta_0/2$, and such that $T^kV$ is not contained in an element of
$F_{j-k}(W)$ for any $1 \le k \le j-1$.  Thus $F_j(W)$ contains elements of $\cG_j(W)$ that are
``long for the first time'' at time $j$.

We group $W_i \in \cG_n(W)$ by its ``first long ancestor'' as follows.  We say $W_i$ has 
\emph{first long ancestor}\footnote{Note that ``ancestor'' refers to the backwards dynamics mapping
$W$ to $W_i$.}
$V \in F_j(W)$ for $1\le j \le n-1$ if $T^{n-j}W_i \subseteq V$.  Note that such a $j$ and $V$ are unique
for each $W_i$ if they exist.  If no such $j$ and $V$ exist, then $W_i$ has been forever short
and so must belong to $\cI_n(W)$.
Denote by $A_{n-j}(V)$ the set of $W_i \in \cG_n(W)$ corresponding
to one $V \in F_j(W)$.   Now
\begin{align*}
&\sum_{W_i \in \cG_n(W)} \left( \frac{\log |W|}{\log |W_i|} \right)^{\bar\gamma}
  \\
  &\quad= \sum_{j=1}^{n-1} \sum_{V_\ell \in F_j(W)} \sum_{W_i \in A_{n-j}(V_\ell)} \left( \frac{\log |W|}{\log |W_i|} \right)^{\bar\gamma} + \sum_{W_i \in \cI_n(W)}  \left( \frac{\log |W|}{\log |W_i|} \right)^{\bar\gamma} \\
&\quad \le \sum_{j=1}^{n-1} \sum_{V_\ell \in F_j(W)}  \left( \frac{\log |W|}{\log |V_\ell|} \right)^{\bar\gamma} 
\sum_{W_i \in A_{n-j}(V_\ell)} \left( \frac{\log |V_\ell|}{\log |W_i|} \right)^{\bar\gamma} 
+ 2^{(ns_0+1)\bar\gamma} (Km+1)^{n/m}
\\
&\quad \le \sum_{j=1}^{n-1} \sum_{V_\ell \in F_j(W)}  \left( \frac{\log |W|}{\log |V_\ell|} \right)^{\bar\gamma}
2^{\bar\gamma +1} \delta_0^{-1} \# \cM_0^{n-j}
+ 2^{(ns_0+1)\bar\gamma} (Km+1)^{n/m} \\
&\quad \le \sum_{j=1}^{n-1} 2^{(js_0+1)\bar\gamma} (Km+1)^{j/m} 2^{\bar\gamma +1} \delta_0^{-1} \# \cM_0^{n-j}
+ 2^{(ns_0+1)\bar\gamma} (Km+1)^{n/m} \displaybreak[0] \\
&\quad \le 2^{2 \bar\gamma +1} \delta_0^{-1} \sum_{j=1}^n 2^{js_0 \bar\gamma} (Km+1)^{j/m}  \# \cM_0^{n-j} \, ,
\end{align*}
where we have applied part (a) from time 1 to time $j$ and the first estimate in part (b) from
time $j$ to time $n$, since each $|V_\ell| \ge \delta_0/2$.  
\end{proof}

With the growth lemma proved, we can choose $m$ and  the length scale $\delta_0$ of curves in $\cW^s$.
Recalling $K$ from  \eqref{eq:complex}
and the condition on $\gamma$ from \eqref{eq:gamma}, we fix $m$ so large that
\begin{equation}
\label{eq:m}
\frac 1m \log(Km+1) < h_* - \gamma s_0 \log 2 \, ,
\end{equation}
and we choose $\delta_0 = \delta_0(m)$ to be the corresponding length scale from Lemma~\ref{lem:growth}.
If $h_* > s_0 \log 2$, then we take $\gamma > 1$, so that in fact
$\frac 1m \log(Km+1) < h_* - s_0 \log 2$.

\subsection{Fragmentation Lemmas}
\label{fragg}
The results in this subsection will be used in Sections~\ref{supera} and~\ref{notmixing}.
For $\delta \in (0, \delta_0)$ and $W \in \hW^s$, define $\cG_n^\delta(W)$ to be the 
smooth components of $T^{-n}W$, with long pieces subdivided to have length between
$\delta/2$ and $\delta$.  (So $\cG_n^\delta(W)$ is defined exactly like $\cG_n(W)$, but with 
$\delta_0$ replaced by $\delta$.)  Let $L_n^\delta(W)$ denote the set of curves in 
$\cG_n^\delta(W)$ that have
length at least $\delta/3$ and let $S_n^\delta(W) = \cG_n^\delta(W) \setminus L_n^\delta(W)$. 
Define $\cI_n^\delta(W)$ to be those curves in $\cG_n^\delta(W)$ that have no 
ancestors\footnote{For $k<n$, we say that $U \in \cG_k^\delta(W)$ is an {\em ancestor} of
$V \in \cG_n^\delta(W)$ if $T^{n-k}V \subseteq U$.}
of length at least $\delta/3$, as in the definition of $\cI_n(W)$ above.
The following lemma and its corollary bootstrap from
Lemma~\ref{lem:growth}(a) and will be crucial to get the lower bound on the spectral radius:

\begin{lemma}
\label{lem:short}
For each $\ve >0$, there exist $\delta \in (0, \delta_0]$ and $n_1 \in \bN$, such that for $n \ge n_1$,
\[
\frac{\# L_n^\delta(W)}{\# \cG_n^\delta(W)} \ge \frac{1 - 2\ve}{1-\ve}, \quad \mbox{for all $W \in \hW^s$ with $|W| \ge \delta/3$.}
\]
\end{lemma}

\begin{proof} 
Fix $\ve>0$ and choose $n_1$ so large that $3 C_1^{-1}(Kn_1+1) \Lambda^{-n_1} < \ve$
and $\Lambda^{n_1} > e$.  Next, choose
$\delta>0$ sufficiently small that if $W \in \hW^s$  with $|W| < \delta$, then $T^{-n}W$ comprises
at most $Kn+1$ smooth pieces of length at most $\delta_0$ for all $n \le 2n_1$.

Let $W \in \hW^s$ with $|W| \ge \delta/3$.
We shall prove the following equivalent inequality for $n \ge n_1$:
\[
\frac{\# S_{n}^\delta(W)}{\# \cG_{n}^\delta(W)} \le \frac{\ve}{1-\ve} \, .
\]
For $n \ge n_1$, write $n = kn_1 + \ell$ for some $0 \le \ell < n_1$.
If $k=1$, the above inequality is clear since $S_{n_1 + \ell}^\delta(W)$ contains at most $K(n_1+\ell)+1$ components by assumption
on $\delta$ and $n_1$, while 
$|T^{-(n_1+\ell)}W| \ge C_1 \Lambda^{n_1+\ell}|W| \ge C_1 \Lambda^{n_1 + \ell} \delta/3$.  Thus
$\cG_n^\delta(W)$ must contain at least $C_1 \Lambda^{n_1+\ell}/3$ curves since each has length at most $\delta$.  Thus,
\[
\frac{\# S_{n_1 + \ell}^\delta(W)}{\# \cG_{n_1+\ell}^\delta(W)} \le 3 C_1^{-1} \frac{K(n_1 + \ell) + 1}{\Lambda^{n_1 + \ell}} \le 3 C_1^{-1} \frac{Kn_1 + 1}{\Lambda^{n_1}} < \ve\, ,
\]
where the second inequality holds for all $\ell \ge 0$ as long as $\frac{1}{n_1} \le \log \Lambda$, 
which is
true by choice of $n_1$.

For $k >1$, we split $n$ into $k-1$ blocks of length $n_1$ and the last block of length $n_1 + \ell$.
We group elements $W_i \in S_{kn_1+\ell}^\delta(W)$ by most recent\footnote{We only consider
what happens at the beginning of a block of length $n_1$. It does not affect our argument if
 $W_i$ belongs to a long piece at an intermediate time, since we only consider the cardinality
 of short pieces that can be created in each block of length $n_1$ according to our choice of 
 $\delta$.} long ancestor
$V_j \in L_{q n_1}^\delta(W)$:  $q$ is the greatest index in $[0, k-1]$
such that $T^{(k-q)n_1 + \ell}W_i \subseteq V_j$ and $V_j \in L_{q n_1}^\delta(W)$.
Note that since $|V_j| \ge \delta/3$, then $\cG^\delta_{(k-q)n_1 + \ell}(V_j)$ must contain at least
$C_1 \Lambda^{(k-q)n_1}/3$ curves since each has length at most $\delta$.
Thus
using Lemma~\ref{lem:growth}(a) with $\bar \gamma = 0$, we estimate
\begin{equation}
\label{eq:ancestor}
\begin{split}
\frac{\# S_{kn_1+\ell}^\delta(W)}{\# \cG_{kn_1+\ell}^\delta(W)}
& = \frac{\sum_{W_i \in \cI_{kn_1+\ell}^\delta(W)} 1 }{\# \cG_{kn_1+\ell}^\delta(W)}+ 
\frac{\sum_{q = 1}^{k-1} \sum_{V_j \in L_{q n_1}^\delta(W)}
\sum_{W_i \in \cI_{(k - q)n_1+ \ell}^\delta(V_j)} 1}{\# \cG_{kn_1+\ell}^\delta(W)}  \\
& \le \frac{(Kn_1+1)^k}{C_1 \Lambda^{kn_1}/3}
+ \sum_{q=1}^{k-1}  \frac{\sum_{V_j \in L_{q n_1}^\delta(W)} (Kn_1 +1)^{k-q}}{\sum_{V_j \in L_{q n_1}^\delta(W)} C_1 \Lambda^{(k-q)n_1}/3} \\
& \le 3 C_1^{-1} \sum_{q = 1}^k (Kn_1+1)^q \Lambda^{-q n_1}
\le \sum_{q = 1}^k \ve^q
\le \frac{\ve}{1-\ve}\,  .
\end{split}
\end{equation}
\end{proof}

The following corollary is used  in Corollary~\ref{cor:abs cont} and in Lemma~\ref{lem:disint}:

\begin{cor}
\label{cor:short}
There exists $C_2 > 0$ such that for any $\varepsilon$, $\delta$ and $n_1$ as in Lemma~\ref{lem:short},
\[
\frac{\# L_n^\delta(W)}{\# \cG_n^\delta(W)} \ge \frac{1 - 3\ve}{1-\ve}\, ,\,\,\,
\forall W \in \hW^s\,, \, \, \forall n \ge C_2 n_1 \frac{|\log(|W|/\delta)|}{|\log \ve|}\, .
\]
\end{cor}

\begin{proof}
The proof is essentially the same as that for Lemma~\ref{lem:short}, except that for curves
shorter than length $\delta/3$ one must
wait $n \sim | \log (|W|/\delta)|$ for at least one component of $\cG_n^\delta(W)$ to belong to 
$L_n^\delta(W)$.

More precisely, fix $\ve >0$ and the corresponding $\delta$ and $n_1$ from Lemma~\ref{lem:short}.
Let $W \in \hW^s$ with $|W| < \delta/3$ and take $n > n_1$.  Decomposing $\cG_n^\delta(W)$
as in Lemma~\ref{lem:short}, we  estimate the second term of \eqref{eq:ancestor}
as before.  

For the first term  of \eqref{eq:ancestor}, $\# \cI_n^\delta(W) / \# \cG_n^\delta(W)$, 
for $\delta$ sufficiently small, notice that since the flow is continuous, either
$\# \cG_\ell^\delta(W) \le K \ell + 1$ by \eqref{eq:complex}
or at least one element of $\cG_\ell^\delta(W)$ has length at least
$\delta/3$.  Let $n_2$ denote the first iterate $\ell$ at which 
$\cG_\ell^\delta(W)$ contains at least
one element of length more than $\delta/3$.  By the complexity estimate  \eqref{eq:complex}
 and the fact that
$|T^{-n_2}W| \ge C_1 \Lambda^{n_2}|W|$ by \eqref{eq:hyp}, there exists
$\bar{C}_2>0$, independent of $W \in \hW^s$, such that $n_2 \le \bar{C}_2 | \log (|W|/\delta)|$.

Now for $n \ge n_2$, and some $W' \in \cG_{n_2}^\delta(W)$,
\[
\# \cI_n^\delta(W) \le (K n_2 + 1) \# \cI_{n-n_2}^\delta(W') \le (Kn_2+1) (K n_1 +1)^{\lfloor (n-n_2)/n_1 \rfloor}\, ,
\]
while
\[
\# \cG_n^\delta(W) \ge C_1 \Lambda^{n-n_2}/3\,  .
\]
Putting these together, we have,
\[
\frac{\# \cI_n^\delta(W)}{\# \cG_n^\delta(W)} \le \frac{(K n_2 + 1)(Kn_1+1)^{\lfloor n/n_1 \rfloor}}{C_1 \Lambda^n/3} \Lambda^{n_2} \le \ve^{\lfloor n/n_1 \rfloor} (Kn_2 + 1) \Lambda^{n_2} \, .
\]
Since $n_2 \le \bar C_2 | \log (|W|/\delta) |$, we may make this expression $< \ve$ by choosing
$n$ so large that $n/n_1 \ge C_2 \frac{\log (|W|/\delta)}{\log \ve}$, for some $C_2 > 0$.
For such $n$, the estimate \eqref{eq:ancestor} is bounded by 
$\ve + \frac{\ve}{1-\ve} \le \frac{2 \ve}{1- \ve}$, which completes the proof of the corollary.
\end{proof}

Choose $\ve = 1/4$ and let $\delta_1\le \delta_0$ and $n_1$ be the corresponding
$\delta$ and $n_1$ from Lemma~\ref{lem:short}.  With this choice, we have
\begin{equation}
\label{eq:delta1}
\# L_n^{\delta_1}(W) \ge \tfrac 23 \#\cG_n^{\delta_1}(W), \qquad \mbox{for all $W \in \hW^s$
with $|W| \ge \delta_1/3$
and $n \ge n_1$.}
\end{equation}

Notice that for $W \in \cW^s$, each element $V \in \cG_n^{\delta_1}(W)$ is contained in one
element of $\cM_0^n$ and its image $T^nV \subset W$ is contained in one element of
$\cM_{-n}^0$.  Indeed, there is a one-to-one correspondence between elements of $\cM_0^n$
and elements of $\cM_{-n}^0$.

The boundary of the partition formed by $\cM_{-n}^0$ is comprised of unstable curves
belonging to $\cS_{-n} = \cup_{j=0}^n T^j(\cS_0)$.  Let 
$\Lo(\cM_{-n}^0)$ denote the elements of
$\cM_{-n}^0$ whose unstable 
diameter\footnote{Recall from Section~\ref{sec:proof 1.4} that the unstable diameter of a set is the length of the longest unstable
curve contained in that set.} 
is at least $\delta_1/3$.  Similarly, let $\Los(\cM_0^n)$ denote the elements of 
$\cM_0^n$ whose stable diameter is at least $\delta_1/3$.

The following lemma will be used to get both lower and upper bounds on the
spectral radius via Proposition~\ref{prop:good growth}:

\begin{lemma}
\label{lem:long piece} Let $\delta_1$ and $n_1$ be  associated with $\ve = 1/4$ by
Lemma~\ref{lem:short}.
There exist $C_{n_1}>0$ and $n_2 \ge n_1$ such that for 
all $n \ge n_2$, 
\[
\# \Lo(\cM_{-n}^0) \ge C_{n_1} \delta_1 \# \cM_{-n}^0 \quad \mbox{ and } \quad
\# \Los(\cM_0^n) \ge C_{n_1} \delta_1 \# \cM_0^n \, .
\]
\end{lemma}

\begin{proof}
We prove the lower bound for $\Lo(\cM_{-n}^0)$.  The lower bound for $\Los(\cM_0^n)$
then follows by time reversal.

Let $\Sh(\cM_{-n}^0)$
denote the elements of $\cM_{-n}^0$ whose unstable 
diameter 
is less than $\delta_1/3$.  Clearly, $\Sh(\cM_{-n}^0) \cup \Lo(\cM_{-n}^0) = \cM_{-n}^0$.
Similarly, Let 
$\Sh(T^j\cS_0)$ denote the set of unstable curves in
$T^j(S_0)$ whose length is less than $\delta_1/3$.

We first prove the following claim:  $\# \Sh(\cM_{-n}^0) \le 2 \sum_{j=1}^n \# \Sh(T^j\cS_0) + K_2 n$.
Recall that the boundaries of elements of $\cM_{-n}^0$ are comprised of elements of 
$\cS_{-n} = \cup_{i=0}^n T^{i} \cS_0$, which are unstable curves for $i \ge 1$.  
We use the following property established in Lemma~\ref{lem:conn}: If a smooth unstable curve $U_i \subset T^i\cS_0$ 
intersects a smooth curve $U_j \subset T^j\cS_0$, for 
$i < j$, then  $U_j$ must terminate on $U_i$. Thus if $A \in \Sh(\cM_{-n}^0)$, then either the boundary
of $A$ contains a short curve in $T^j(\cS_0)$ for some $1 \le j \le n$, or $\partial A$ contains
an intersection point of two curves in $T^j(\cS_0)$ for some $1 \le j \le n$ (see
Figure~\ref{fig:sing}).
But 
such intersections
of curves within $T^j(\cS_0)$ are images of intersections of curves within $T(\cS_0)$, 
and the cardinality of cells created by such intersections 
is bounded by some uniform constant $K_2 > 0$ depending only on $T(\cS_0)$. 
Then, since each short curve in
$T^j(\cS_0)$ belongs to the boundary of at most two $A \in \Sh(\cM_{-n}^0)$, the claim follows.

\begin{figure}[ht]
\begin{centering}
\begin{tikzpicture}[x=6mm,y=6mm]

\draw[thick] (10,10) to[out=220, in=35] (7.8,8.3) to[out=210, in=15] (5.7,7);
\draw[thick] (9,11) to[out=245, in=55] (8.3,9.6) to[out=235, in=50] (6.70,7.49);

\draw[thick] (7.62,8.65) to[out=220, in=30] (6.3,7.8) to[out=200, in=15] (4, 7);
\draw[thick] (6.2,7.75) to[out=240, in=40] (5.7,7) to[out=220, in=30] (4,5.7);

\node at (6.4,7.65){\scriptsize $A$};
\node at (7,9.5){\small $T^j(\cS_0)$};

 \end{tikzpicture}
\caption{A short cell $A \in I_u(\cM_{-n}^0)$ created by long elements of $T^j(\cS_0)$.  }  
\label{fig:sing}
\end{centering}
\end{figure}
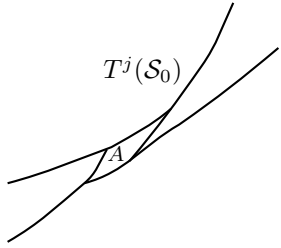

Next, subdivide $\cS_0$ into $\ell_0$ horizontal segments $U_i$ such that $TU_i$ is an unstable curve
of length between $\delta_1/3$ and $\delta_1$ for each $i$.  
Analogous to stable curves, let $\cG_j^{\delta_1}(U)$ denote the decomposition of the union of unstable curves comprising
$T^jU$ at length scale $\delta_1$.
Then for $j \ge n_1$ using the time reversal of \eqref{eq:delta1}, we have
\begin{equation}
\label{eq:short}
\# \Sh(T^j \cS_0) = \sum_{i = 1}^{\ell_0} \# \Sh(\cG_{j-1}^{\delta_1}(TU_i)) 
\le \tfrac 12 \sum_{i=1}^{\ell_0} \# \Lo(\cG_{j-1}^{\delta_1}(TU_i))\, .
\end{equation}

Using the claim and \eqref{eq:short} we split the sum over $j$ into 2 parts,
\begin{equation}
\label{eq:j split}
\# \Sh(\cM_{-n}^0) \le K_2 n + 2 \sum_{j=1}^{n_1-1} \# \Sh(T^j \cS_0) + \sum_{j = n_1}^{n}
\sum_{i =1}^{\ell_0} \# \Lo(\cG_{j-1}^{\delta_1}(TU_i)) \, .
\end{equation}

The cardinality of the sum over the first $n_1$ terms is bounded by a fixed constant depending 
on $n_1$, but not on $n$; let us call it $\bar C_{n_1}$.  We want to relate the sum over the terms
for $j \ge n_1$ to $\Lo(\cM_{-n}^0)$.  To this end, we follow the proof of Lemma~\ref{lem:short}
and split $n-j$ into blocks of length $n_1$.

For each $n_1 \le j \le n-n_1$, write $n-j = kn_1 + \ell$, for some $k \ge 1$.  If
$V \in \Lo(\cG_{j-1}^{\delta_1}(TU_i))$, then $|T^{n-j}V| \ge C_1 \Lambda^{n-j} \delta_1/3$,
while $T^{n-j}V$ can be cut into at most $(Kn_1+1)^k$ pieces.  Since we have chosen
$\ve = 1/4$ in the application of Lemma~\ref{lem:short}, by choice of $n_1$, 
\[
\# \Lo(\cG_{n-1}^{\delta_1}(TU_i)) \ge 4^k \# \Lo(\cG_{j-1}^{\delta_1}(TU_i))
\; \mbox{for each $n_1 \le j \le n-n_1$ and $k = \left \lfloor \frac{(n-j)}{n_1} \right \rfloor$.} 
\]
For $n-n_1 < j \le n$, we perform the same estimate, but relating $j$ with $j + n_1$,
\[
\# \Lo(\cG_{j+n_1-1}^{\delta_1}(TU_i)) \ge 4 \# \Lo(\cG_{j-1}^{\delta_1}(TU_i))
\; \; \mbox{for each $n-n_1 + 1\le j \le n$.}
\]
Gathering these estimates together and using \eqref{eq:j split}, we obtain,
\begin{equation}
\label{eq:short up}
\begin{split}
&\# \Sh(\cM_{-n}^0) \\
&\qquad  \le K_2 n + \bar C_{n_1} + \sum_{j=n_1}^{n-n_1}
4^{-\lfloor (n-j)/n_1 \rfloor} \# \Lo(T^n\cS_0) + \sum_{j=n-n_1+1}^n \tfrac 14
\# \Lo(T^{j+n_1} \cS_0) \\
&\qquad  \le 2K_2 n + \bar C_{n_1} + C \delta_1^{-1} n_1 \# \Lo(\cM_{-n}^0) + \sum_{j=n-n_1+1}^n C\delta_1^{-1} 
\# \Lo(\cM_{-j-n_1}^0) \, ,
\end{split}
\end{equation}
where the second inequality uses $\# \Lo(T^\ell\cS_0) \le C \delta_1^{-1} \Lo( \cM_{-\ell}^0) + K_2$
for $\ell \ge n$,
which stems from the same non-crossing property used earlier:  a curve in $T^\ell(\cS_0)$
must terminate on a curve in $T^i(\cS_0)$ if the two intersect for $i<\ell$.

To estimate the final sum in \eqref{eq:short up}, note that if $A \in \Lo(\cM_{-n-1}^0)$, then
$A \subseteq A' \in \Lo(\cM_{-n}^0)$.  Moreover, there exists a constant $B>0$, independent of $n$,
such that each $A' \in \Lo(\cM_{-n}^0)$ can contain at most $B$ elements of $\Lo(\cM_{-n-1}^0)$.
(Indeed by Lemma~\ref{prop:equiv}, $B$ is at most $| \hP|$, and depends only on
$\cS_1$.)  Inductively then, 
\[
\sum_{j=1}^{n_1} \# \Lo(\cM_{-n-j}^0) \le  \sum_{j=1}^{n_1} B^j \# \Lo(\cM_{-n}^0)
\le C B^{n_1} \# \Lo(\cM_{-n}^0) \, .
\]
Putting this estimate together with \eqref{eq:short up} yields,
\[
\# \Sh(\cM_{-n}^0)
\le \# \Lo(\cM_{-n}^0) C \delta_1^{-1} (n_1 + B^{n_1}) + C_{n_1} + 2K_2n \, .
\]
Using $\# \cM_{-n}^0 = \# \Lo(\cM_{-n}^0) + \# \Sh(\cM_{-n}^0)$, this implies,
\[
\# \Lo(\cM_{-n}^0) \ge \frac{ \# \cM_{-n}^0 - C_{n_1} - 2K_2n}{1 + C\delta_1^{-1}(n_1 + B^{n_1})} \, .
\]
Since $\# \cM_{-n}^0$ increases at an exponential rate and $n_1$ is fixed, there exists $n_2 \in \mathbb{N}$ such that $\# \cM_{-n}^0 - \bar C_{n_1} - 2K_2n \ge \frac 12 \# \cM_{-n}^0$, for $n \ge n_2$.
Thus there exists $C_{n_1} >0$ such that for $n \ge n_2$,
$
\# \Lo(\cM_{-n}^0) \ge C_{n_1} \delta_1 \# \cM_{-n}^0
$,
as required.
\end{proof}

\subsection{Exact Exponential Growth of $\# \cM_0^n$ ---  Cantor Rectangles}
\label{supera}

It follows from submultiplicativity
of $\# \cM_0^n$ that $e^{n h_*} \le \# \cM_0^n$ for all $n$. In this subsection, we shall
prove  a supermultiplicativity statement (Lemma~\ref{lem:super}) from which we deduce
the upper bound for $\# \cM^n_0$ in Proposition~\ref{cor:exp} giving the upper bound in
Proposition~\ref{prop:ly}, 
and ultimately the upper bound on the spectral radius of $\cL$ on $\cB$.

The following key estimate is a lower bound on the rate of growth of stable curves having a certain length.
The proof will crucially use the fact that the SRB measure is mixing in order
to bootstrap from  Lemma~\ref{lem:long piece}.

\begin{proposition}
\label{prop:good growth}
Let $\delta_1$ be the value of  
$\delta$ from Lemma~\ref{lem:short} associated with $\ve = 1/4$ (see \eqref{eq:delta1}).
There exists $c_0>0$ such that for all $W \in \widehat{\cW}^s$ with $|W| \ge \delta_1/3$
and $n \ge 1$, we have $\# \cG_n(W) \ge c_0 \# \cM_0^n$.
The constant $c_0$ depends on $\delta_1$.
\end{proposition}

This will be used for the lower bound  in Section~\ref{prspu}.
It also has the following important consequence. 

\begin{lemma}[Supermultiplicativity]
\label{lem:super}
There exists $c_1>0$ such that $\forall n,j \in \mathbb{N}$, with $j \le n$, we have 
\[
\# \cM_0^n \ge c_1 \# \cM_0^{n-j} \# \cM_0^j \, .
\]
\end{lemma}

We next introduce Cantor rectangles.
Let $W^s(x)$ and $W^u(x)$ denote the maximal smooth components of the local stable and unstable 
manifolds of $x \in M$.

\begin{definition}[(Locally Maximal) Cantor Rectangles]\label{MCR}
A solid rectangle $D$ in $M$ is a closed region whose boundary comprises precisely four 
nontrivial curves:
two stable manifolds and two unstable manifolds.
Given a solid rectangle $D$, the {\em locally maximal Cantor rectangle} $R$ in $D$ is formed by taking the union of all
points in $D$ whose local stable and unstable manifolds completely cross $D$.
Locally maximal Cantor rectangles have a natural product structure:
for any $x, y \in R$, $W^s(x) \cap W^u(y) \in R$, where $W^{s/u}(x)$ is the local
stable/unstable manifold containing $x$.  
It is proved in \cite[Section 7.11]{chernov book} that
such rectangles are closed and as such contain their outer boundaries, which coincide with the boundary of $D$.
We shall refer to this pair of stable 
and unstable manifolds as the stable and unstable boundaries of $R$.  In this case, we denote $D$ by 
$D(R)$ to emphasize that it is the smallest solid rectangle containing $R$.  
We shall sometimes drop the words ``locally maximal'' 
referring simply to Cantor rectangles $R$.
\end{definition}

\begin{definition}[Properly Crossing a (Locally Maximal) Cantor Rectangle]\label{PCCR}
For  a (locally maximal)  Cantor rectangle $R$ such that
\begin{equation}\label{**}
\inf_{x \in R} \frac{m_{W^u}(W^u(x) \cap R)}{m_{W^u}(W^u(x) \cap D(R))} \ge 0.9\, ,
\end{equation}
we\footnote{This is a  version
of Definition 7.85 of \cite{chernov book} formulated with
stable (instead of unstable) curves crossing $R$.  We have also dropped any mention of homogeneous components, which are
used in the construction in \cite{chernov book}.}
say a stable curve $W \in \widehat{\cW}^s$ 
{\em properly crosses}  $R$ if

 a) $W$ crosses both unstable sides of $R$;

 b) for every $x \in R$, the intersection $W \cap W^s(x) \cap D(R) = \emptyset$, i.e., $W$ does not cross any
  stable manifolds in $R$;
 
c) for all $x \in R$, the point $W \cap W^u(x)$ divides the curve $W^u(x) \cap D(R)$ in a 
  ratio between $0.1$ and $0.9$, i.e., $W$ does not come too close to either unstable boundary of $R$.
\end{definition}

\begin{remark}
The (unstable analogue of) condition b) is not needed in its full strength, even in the proof of 
\cite[Lemma~7.90]{chernov book}.  What is used there is that the fake unstable  is trapped between two real unstable that it does not cross.  Since the real unstable intersect and fully cross the target rectangle, this forces the fake unstable to do so as well.  
For us, we reverse time and consider stable manifolds.  For real stable manifolds, condition (b) 
is not needed at all: If a real stable fully crosses the initial rectangle, then, when it intersects the target rectangle under iteration
by $T^{-n}$, it must intersect a real stable manifold, and it must fully cross.  (Otherwise, the preimage of a singularity 
would lie on a real stable manifold in the interior of the target rectangle.  But this cannot be since real stable manifolds are never cut going forward and so do not intersect the preimages of singularity curves except at their end points.)
When discussing proper crossing for real stable manifolds, we will drop condition (b) and allow $W \in \cW^s$ to be
one of the stable manifolds defining $R$.
\end{remark}

\begin{proof}[Proof of Proposition~\ref{prop:good growth}]
Using \cite[Lemma~7.87]{chernov book}, we may cover $M$ by  Cantor rectangles 
$R_1, \ldots R_k$ satisfying \eqref{**} whose stable and unstable boundaries have length  
at most $\frac{1}{10} \delta_1$, with the property that any 
stable curve of length at least $\delta_1/3$ properly crosses at least one of them.
The cardinality $k$ is fixed, depending only on $\delta_1$.

Recall that $\Lo(\cM_{-n}^0)$ denotes the elements of $\cM_{-n}^0$ whose unstable
diameter is longer than $\delta_1/3$.
We claim that for all $n \in \mathbb{N}$, at least one 
$R_i$ is fully crossed in the unstable direction by at least $\frac 1k \# \Lo(\cM_{-n}^0)$ 
elements of
$\cM_{-n}^0$.  Notice that if $A \in \cM_{-n}^0$, then $\partial A$ is comprised of unstable
curves belonging to $\cup_{i=1}^n T^i\cS_0$, and possibly $\cS_0$.  By definition of unstable
manifolds, $T^i\cS_0$ cannot intersect the unstable boundaries of the $R_i$; 
thus if $A \cap R_i \neq \emptyset$, 
then either $\partial A$ 
terminates inside $R_i$ or $A$ fully crosses $R_i$.  Thus elements of $\Lo(\cM_{-n}^0)$
fully cross at least one $R_i$ and so at least one $R_i$ must be fully crossed by $1/k$ of them,
proving the claim.

For each $n \in \mathbb{N}$, denote by $i_n$ the index of a rectangle $R_{i_n}$ which is fully
crossed by at least $\frac 1k \# \Lo(\cM_{-n}^0)$ elements of $\cM_{-n}^0$.
The main idea at this point will be to force every stable curve to properly cross $R_{i_n}$ in
a bounded number of iterates and so to intersect all elements of $\cM_{-n}^0$ that fully cross
$R_{i_n}$.

To this end, fix $\delta_* \in (0, \delta_1/10)$ and for $i=1, \ldots k$, choose
a ``high density'' subset $R_i^* \subset R_i$ satisfying the
following conditions: $R_i^*$ has nonzero Lebesgue measure,
and for any unstable manifold
$W^u$ such that $W^u \cap R_i^* \neq \emptyset$ and $|W^u| < \delta_*$, we have
$\frac{m_{W^u}(W^u \cap R_i^*)}{|W^u|} \ge 0.9$.  (Such a $\delta_*$ and $R_i^*$ exist due to the
fact that  $m_{W^u}$-almost every 
 $y \in R_i$ is a Lebesgue density point
of the set $W^u(y) \cap R_i$ and
the unstable foliation is absolutely continuous
 with respect to $\musrb$ or, equivalently, Lebesgue.)

Due to the mixing property of $\musrb$ and the finiteness of the number of rectangles $R_i$, there
exist $\ve >0$ and  $n_3 \in \mathbb{N}$ such that for all $1 \le i, j \le k$ and all $n \ge n_3$,
$\musrb(R_i^* \cap T^{-n}R_j) \ge \ve$.  If necessary, we increase $n_3$ so that the unstable 
diameter of the set $T^{-n}R_i$ is less than $\delta_*$ for each $i$, and $n \ge n_3$.

Now let $W \in \widehat{\cW}^s$ with $|W|\ge \delta_1/3$ be 
arbitrary.  Let $R_j$ be a 
 Cantor rectangle 
 that is properly
crossed by $W$.  Let $n \in \mathbb{N}$ and let $i_n$ be as above.  By mixing, 
$\musrb(R_{i_n}^* \cap T^{-n_3} R_j) \ge \ve$.  By \cite[Lemma~7.90]{chernov book}, there is
a component of $T^{-n_3}W$ that fully crosses $R_{i_n}^*$ in the stable direction.  Call
this component $V \in \cG_{n_3}^{\delta_1}(W)$.  By choice of $R_{i_n}$, this implies that
$\# \cG_n(V) \ge \frac 1k \# \Lo(\cM_{-n}^0)$, and thus
\[
\# \cG_{n+n_3}(W) \ge \tfrac 1k \# \Lo(\cM_{-n}^0)
\implies \# \cG_n(W) \ge \tfrac{C'}{k} \# \Lo(\cM_{-n}^0)\, ,
\]
where $C'$ is a constant depending only on $n_3$ since at each refinement of $\cM_{-j}^0$
to $\cM_{-j-1}^0$, the cardinality of the partition increases by a factor which is at most $|\hP|$,
as noted in the proof of Lemma~\ref{lem:long piece}.  The final estimate needed is
$\# \Lo(\cM_{-n}^0) \ge C_{n_1} \delta_1 \# \cM_{-n}^0$, for $n \ge n_2$ from Lemma~\ref{lem:long piece}.
Thus the proposition holds for $n \ge \max\{ n_2, n_3 \}$.  It extends to all $n \in \mathbb{N}$ since
$\# \cM_0^n \le (\# \cM_0^1)^n$ and there are only finitely many values of $n$ to correct for. 
\end{proof}

\begin{proof}[Proof of Lemma~\ref{lem:super}]
Recall the singularity sets defined for $n, k \in \mathbb{N}$ by $\cS_n = \cup_{i=0}^n T^{-i}\cS_0$
and $\cS_{-k} = \cup_{i=0}^k T^i\cS_0$.  Due to the relation,
$
T^{-k}(\cS_{-k} \cup \cS_n) = \cS_k \cup T^{-k}\cS_n = \cS_{n+k}
$,
we have a one-to-one correspondence between elements of $\cM_{-k}^n$ and $\cM_0^{n+k}$.

Now fix $n, j \in \mathbb{N}$ with $j < n$.  Using the above relation, we have,
\[
\# \cM_0^n = \# \cM_{-j}^{n-j} = \# \big( \cM_0^{n-j} \vee \cM_{-j}^0 \big)\, .
\]
In order to prove the lemma, it suffices to show that a positive fraction (independent of $n$ and $j$)
of elements of $\cM_0^{n-j}$ intersect a positive fraction of elements of $\cM_{-j}^0$.
Note that $\partial \cM_0^{n-j}$ is comprised of stable curves, while $\partial \cM_{-j}^0$ is comprised
of unstable curves.

Recall that $\Lo(\cM_{-j}^0)$ denotes the elements of $\cM_{-j}^0$ whose unstable
diameter is longer than $\delta_1/3$.
Similarly, $\Los(\cM_0^{n-j})$ denotes those elements
of $\cM_0^{n-j}$ whose stable diameter is longer than $\delta_1/3$.
By Lemma~\ref{lem:long piece},
\[
\# \Los(\cM_0^{n-j}) \ge C_{n_1} \delta_1 \# \cM_0^{n-j}, \quad \mbox{for $n-j \ge n_2$}\, .
\]

Let $A \in \Los(\cM_0^{n-j})$ and let $V \in \widehat{\cW}^s$ be a stable curve in 
$A$ with length at least $\delta_1/3$.  By Proposition~\ref{prop:good growth},
$\# \cG_j(V) \ge c_0 \# \cM_0^j$.  Each component of $\cG_j(V)$ corresponds to one component
of $V \setminus \cS_{-j}$ (up to subdivision of long pieces in $\cG_j(V)$).  
Thus $V$ intersects at least $c_0 \# \cM_0^j = c_0 \# \cM_{-j}^0$ elements of
$\cM_{-j}^0$.  Since this holds for all $A \in \Los(\cM_0^{n-j})$, we have
\[
\# \cM_0^n = \# \big( \cM_0^{n-j} \vee \cM_{-j}^0 \big) \ge 
\# \Los(\cM_0^{n-j}) \cdot c_0 \# \cM_0^j \ge C_{n_1} \delta_1 c_0 \# \cM_0^{n-j} \# \cM_0^j\, ,
\]
proving the lemma with $c_1 = c_0 C_{n_1} \delta_1$ when $n-j \ge n_2$.  For $n-j \le n_2$,
since $\# \cM_0^{n-j} \le (\# \cM_0^1)^{n-j}$,
we obtain the lemma  by decreasing $c_1$ since there are only finitely many values
to correct for.
\end{proof}

\begin{proof}[Proof of Proposition~\ref{cor:exp}]
Define $\psi(n) = \# \cM_0^n e^{-n h_*}$, and note that $\psi(n) \ge 1$ for all
$n$.  From Lemma~\ref{lem:super} it follows that
\begin{equation}
\label{eq:sup psi}
\psi(n) \ge c_1 \psi(j) \psi(n-j), \quad \mbox{for all $n \in \mathbb{N}$, and $0 \le j \le n$.}
\end{equation}
Suppose there exists $n_1 \in \mathbb{N}$ such that $\psi(n_1) \ge 2/c_1$.  Then
using \eqref{eq:sup psi}, we have
\[
\psi(2n_1) \ge c_1 \psi(n_1)\psi(n_1) \ge \frac{4}{c_1}\,  .
\]
Iterating this bound, we have inductively for any $k \ge 1$,
\[
\psi(2kn_1) \ge c_1 \psi(2n_1)\psi(2(k-1)n_1) \ge c_1 \frac{4}{c_1} \frac{4^{k-1}}{c_1} = \frac{4^k}{c_1}\, . 
\]
This implies that $\lim_{k \to \infty} \frac{1}{2kn_1} \log \psi(2kn_1) \ge \frac{\log 4}{2 n_1}$, which 
contradicts the definition of $\psi(n)$
(since $\lim_{n \to \infty} \frac 1n \log \psi(n) = 0$).  We conclude that $\psi(n) \le 2/c_1$ for all $n \ge 1$.
\end{proof}

Our final result of this section demonstrates the uniform exponential rate of growth enjoyed by all
stable curves of length at least $\delta_1/3$.

\begin{cor}
\label{cor:stable growth}
For all stable curves $W \in \widehat{\cW}^s$ with $|W| \ge \delta_1/3$ and all $n \ge n_1$, we have
 \[
\frac{2 \delta_1 c_0}{9} e^{nh_*} \le  |T^{-n}W| \le \frac{4}{c_1} e^{nh_*} \, .
\]
\end{cor}

\begin{proof}
For $W \in \widehat{\cW}^s$ with $|W| \le \delta_1/3$,
Lemma~\ref{lem:growth}(b) with $\bar\gamma=0$ together with Propositions~\ref{cor:exp}
and \ref{prop:good growth} yield,
\[
c_0 e^{n h_*} \le
c_0 \# \cM_0^n \le \# \cG_n(W) \le 2\delta_0^{-1} \# \cM_0^n 
\le \tfrac{4}{c_1 \delta_0} e^{n h_*} \, .
\] 
The upper bound of the corollary is completed by noting that
\[
|T^{-n}W| = \sum_{W_i \in \cG_n(W)} |W_i| \le \delta_0 \# \cG_n(W) \, .
\]
The lower bound follows using \eqref{eq:delta1} since $\# \cG_n^{\delta_1}(W) \ge \# \cG_n(W)$,
\begin{equation}
\label{eq:length lower}
|T^{-n}W| = \sum_{W_i \in \cG_n^{\delta_1}(W)} |W_i| 
\ge \frac{\delta_1}{3} \# L_n^{\delta_1}(W) 
\ge \tfrac{2\delta_1}{9} \# \cG_n^{\delta_1}(W) 
\ge \tfrac{2\delta_1c_0}{9} e^{n h_*} \, .
\end{equation}
\end{proof}


\section{Proof of the ``Lasota--Yorke''  Proposition~\ref{prop:ly} --- Spectral Radius}
\label{sec:proof prop}

\subsection{Weak Norm and Strong Stable Norm Estimates}

We start with the weak norm estimate \eqref{eq:weak ly}.
Let $f \in C^1(M)$, $W \in \cW^s$, and $\psi \in C^\alpha(W)$ be such that
$|\psi|_{C^\alpha(W)} \leq 1$.	 For $n \geq0$ we use the definition of the weak
norm on each $W_i \in \cG_n(W)$ to estimate
\begin{equation}
\label{eq:start}
\begin{split}
\int_W \cL^nf \, \psi \, dm_W
& =\sum_{W_i \in\cG_n(W)}\int_{W_i}f \, \psi \circ T^n \, dm_W   \le \sum_{W_i \in\cG_n(W)} |f|_w |\psi\circ T^n|_{C^\alpha(W_i)} \,  .
\end{split}
\end{equation}

Clearly, $\sup |\psi\circ T^n|_{W_i}\le \sup_W |\psi|$.
For $x,y \in W_i$, we have,
\begin{align}
\label{eq:C1 C0}
\frac{|\psi (T^nx) - \psi (T^ny)|}{d_W(T^nx,T^ny)^\alpha}
\cdot \frac{d_W(T^nx,T^ny)^\alpha}{d_W(x,y)^\alpha} &\leq C |\psi|_{C^\alpha(W)} 
|J^sT^n|^\alpha_{C^0(W_i)}\\
 \nonumber   &\leq C\Lambda^{-\alpha n} |\psi|_{C^\alpha(W)}\, ,
\end{align}
so that  $H_{W_i}^\alpha(\psi\circ T^n)
\le  C\Lambda^{-\alpha n} H_{W}^\alpha(\psi)$ and thus $|\psi \circ T^n|_{C^\alpha(W_i)} \leq C |\psi|_{C^\alpha(W)}$.
Using this estimate and Lemma~\ref{lem:growth}(b) with $\bar\gamma = 0$ 
in equation \eqref{eq:start}, we obtain
\[
\int_W \cL^n f \, \psi \, dm_W \; \leq \;  \sum_{W_i \in \cG_n(W)} C |f|_w
\le C \delta_0^{-1} |f|_w (\# \cM_0^n)  \,  .
\]
Taking the supremum over $W \in \cW^s$ and $\psi \in C^\alpha(W)$ with
$|\psi|_{C^\alpha(W)} \leq 1$
yields \eqref{eq:weak ly}, using the upper bound 
on $\# \cM_0^n$ in Proposition~\ref{cor:exp}.
\smallskip

We now prove the strong stable norm estimate \eqref{cheapeq:stable ly}.
Recall that our choice of $m$ in \eqref{eq:m} implies
$2^{s_0\gamma} (Km+1)^{1/m} < e^{h_*}$, where $K$ is from \eqref{eq:complex}.
Define
\begin{equation}
\label{eq:Dn}
D_n = D_n(m,\gamma) :=  2^{2\gamma +1} \delta_0^{-1} \sum_{j=1}^n 2^{j s_0 \gamma} (Km+1)^{j/m} \# \cM_0^{n-j} \, .
\end{equation}
We claim that it follows from Proposition~\ref{cor:exp} that  
\begin{equation}\label{ccor:exp}
D_n\le C e^{nh_*}\, .
\end{equation}
Indeed, by choice of $\gamma$ and $m$, setting
$\ve_1 := h_* - \log( 2^{s_0\gamma} (Km+1)^{1/m}) > 0$, we have
\begin{align*}
D_n & = 2^{2\gamma +1} \delta_0^{-1} \sum_{j=1}^n 2^{j s_0 \gamma} (Km+1)^{j/m} \# \cM_0^{n-j}
 \le
2^{2\gamma +1} \delta_0^{-1} \sum_{j=1}^n e^{(h_* - \ve_1)j}  \frac{2}{c_1} e^{(n-j)h_*}
\\
\nonumber &\le 2^{2\gamma +1} \delta_0^{-1} \frac{2}{c_1} e^{n h_*}
\sum_{j=1}^n e^{ - \ve_1 j}  \, .
\end{align*}

To prove the strong stable bound,
let $W \in \cW^s$ and $\psi \in C^\beta(W)$ with $|\psi|_{C^\beta(W)} \le |\log|W||^\gamma$.
Using equation \eqref{eq:start}, and applying the strong stable norm on each $W_i \in \cG_n(W)$, 
we write
\[
\int_W\cL^n f\, \psi \, dm_W	 =
\sum_{i} \int_{W_i}f \, \psi \circ T^n \, dm_W 
\le \sum_i \|f\|_s |\log |W_i||^{-\gamma} |\psi \circ T^n |_{C^\beta(W_i)} \, .
\]
From the estimate analogous to \eqref{eq:C1 C0}, we
have
$|\psi \circ T^n |_{C^\beta(W_i)} \leq C |\psi|_{C^\beta(W)} \leq C  |\log |W||^\gamma$. (Note that the contraction coming from the negative power of $\Lambda$ in \eqref{eq:C1 C0} cannot be exploited, 
see footnote~\ref{pied} and the comments after Remark \ref{4.8}.)

Thus,
\[
\int_W \cL^n f \, \psi \, dm_W  \leq
C  \| f\|_s \sum_{W_i \in \cG_n(W)}  
\left( \frac{\log |W|}{\log |W_i|} \right)^\gamma 
\; \leq \; C  \|f\|_s D_n \, ,
\]
where we have used 
Lemma~\ref{lem:growth}(b) with $\bar\gamma = \gamma$.

Taking the
supremum over $W$ and $\psi$ and recalling \eqref{ccor:exp}
proves
\eqref{cheapeq:stable ly},  since we have shown that
$\|\cL^n f\|_s \leq C D_n \|f\|_s$.


\subsection{Unstable Norm Estimate}
\label{unstable norm}

Fix $\ve \le \ve_0$ and
consider two curves $W^1, W^2 \in\cW^s$ with $d_{\cW^s}(W^1,W^2) \leq \ve$.
For $n \geq 1$, we describe how to partition $T^{-n}W^\ell$ into
``matched'' pieces $U^\ell_j$ and
``unmatched'' pieces $V^\ell_i$, $\ell=1,2$.

Let $\omega$ be a connected component of $W^1 \setminus \cS_{-n}$.
To each
point $x \in T^{-n}\omega$, we associate a vertical line segment $\gamma_x$ of length
at most $C\Lambda^{-n}\ve$ such that its image $T^n\gamma_x$,
if not cut by a singularity,
will have length $C\ve$.  By \cite[\S 4.4]{chernov book}, all the tangent vectors to $T^i\gamma_x$
lie in the unstable cone
$C^u(T^ix)$ for each $i \ge 1$ so that they remain uniformly transverse to the stable cone
and enjoy the minimum expansion given by $\Lambda$.

Doing this for each connected component of $W^1 \setminus \cS_{-n}$,
we subdivide $W^1 \setminus \cS_{-n}$ into a countable
collection of subintervals of points for which $T^n\gamma_x$ intersects
$W^2 \setminus \cS_{-n}$ and subintervals for which this is not the case.
This in turn induces a corresponding partition on $W^2 \setminus \cS_{-n}$.

We denote by $V^\ell_i$ the pieces in $T^{-n}W^\ell$ which are not matched up by this process
and note that
the images $T^nV^\ell_i$ occur either at the endpoints of $W^\ell$ or because the vertical segment
$\gamma_x$ has been cut by a singularity.  In both cases, the length of the
curves $T^nV^\ell_i$ can be at most $C\ve$ due to the uniform transversality of
$\cS_{-n}$ with the stable cone and of $C^s(x)$ with $C^u(x)$.

In the remaining
pieces the foliation $\{ T^n\gamma_x \}_{x \in T^{-n}W^1}$ provides a one-to-one correspondence
between points in $W^1$ and $W^2$.
We further subdivide these pieces in
such a way that the lengths of their images under $T^{-i}$
are less than $\delta_0$ for each $0 \le i \le n$ and
the pieces are pairwise matched by
the foliation $\{\gamma_x\}$. We call these matched pieces $U^\ell_j$.
Since the stable cone is bounded away from the vertical direction, we can
adjust the elements of $\cG_n(W^\ell)$ created by artificial subdivisions
due to length so that $U^\ell_j \subset W^\ell_i$ and $V^\ell_k \subset
W^\ell_{i'}$ for some $W^\ell_i, W^\ell_{i'} \in \cG_n(W^\ell)$ for all
$j,k \ge 1$ and $\ell = 1,2$, without changing the cardinality
of the bound on $\cG_n(W^\ell)$.
There is at most one $U^\ell_j$ and two $V^\ell_j$ per $W^\ell_i \in
\cG_n(W^\ell)$.

In this way we write $W^\ell = (\cup_j T^nU^\ell_j) \cup (\cup_i T^nV^\ell_i)$.
Note that the images $T^nV^\ell_i$ of the unmatched pieces must be short
while the images of the matched pieces
$U^\ell_j$
may be long or short.

We have arranged a pairing
of the pieces $U^\ell_j=G_{U^\ell_j}(I_j) $, $\ell=1, 2$,  with the  property:
\begin{equation}
\label{eq:match}
\begin{split}
\mbox{If } \; &
U^1_j 
= \{ (r, \vf_{U^1_j}(r)) : r \in I_j \}\, \, 
\mbox{then } 
U^2_j 
= \{ (r, \vf_{U^2_j}(r)) : r \in I_j \} \, ,
\end{split}
\end{equation}
so that the
point $x = (r, \vf_{U^1_j}(r))$ is associated with the point
$\bar x = (r, \vf_{U^2_j}(r))$ by the vertical
segment $\gamma_x \subset \{(r,s)\}_{s\in[-\pi/2, \pi/2]}$, for each $r \in I_j$.

Given $\psi_\ell$ on $W^\ell$ with $|\psi_\ell|_{C^\alpha(W^\ell)} \leq 1$ and
$d(\psi_1, \psi_2) \leq \ve$,
we must estimate
\begin{equation}
\label{eq:unstable split}
\begin{split}
& \left|\int_{W^1} \cL^n f \, \psi_1 \, dm_W - \int_{W^2} \cL^n f \, \psi_2 \, dm_W \right|
  \; \leq \; \sum_{\ell,i} \left|\int_{V^\ell_i} f \, \psi_\ell\circ T^n \, dm_W \right|\\
  & \qquad\qquad\qquad \qquad\qquad\quad+ \sum_j \left| \int_{U^1_j} f \, \psi_1\circ T^n \, dm_W
    - \int_{U^2_j} f \, \psi_2\circ T^n \, dm_W \right| \, .
\end{split}
\end{equation}
We first estimate the differences of matched pieces
$U^\ell_j$.
The function $\phi_j= \psi_1 \circ T^n \circ G_{U^1_j} \circ G_{U^2_j}^{-1}$ is well-defined on $U^2_j$, and we can estimate,
\begin{equation}
\label{eq:stepone}
\left |\int_{U^1_j} f \, \psi_1\circ T^n-
\int_{U^2_j} f \, \psi_2\circ T^n \right |
\leq \left |\int_{U^1_j} f \, \psi_1\circ T^n-
\int_{U^2_j} f \,\phi_j \right |
+\left |\int_{U^2_j} f (\phi_j -  \psi_2\circ
T^n) \right | \,  .
\end{equation}

We bound the first term in equation~\eqref{eq:stepone} using the strong unstable norm.
As before,
\eqref{eq:C1 C0} implies
$|\psi_1 \circ T^n|_{C^\alpha(U^1_j)} \le
C |\psi_1|_{C^\alpha(W^1)} \le C$.
We have $|G_{U^1_j} \circ G_{U^2_j}^{-1}|_{C^1} \le C_g$,
for some $C_g > 0$ due to the fact that each curve $U^\ell_j$ has uniformly bounded
curvature and slopes bounded away from infinity.
Thus
\begin{equation}
\label{eq:psi_1}
|\phi_j |_{C^\alpha(U^2_j)} \le C C_g |\psi_1|_{C^\alpha(W^1)} \, .
\end{equation}
Moreover, $d(\psi_1\circ T^n, \phi_j)
= \left|\psi_1\circ T^n  \circ G_{U^1_j}
  - \phi_j \circ G_{U^2_j} \right|_{C^0(I_j)} \; = \; 0$ 
by the definition of $\phi_j$.
  
  To complete the bound on the first term of \eqref{eq:stepone},
we need the following estimate from \cite[Lemma~4.2]{demzhang11}:
There exists $C>0$, independent of $W^1$ and $W^2$, such that
\begin{equation}
\label{eq:graph contract}
d_{\cW^s}(U^1_j,U^2_j)\leq C\Lambda^{-n}n \ve =: \ve_1\, , \qquad
\forall j  \, .
\end{equation}

In view of \eqref{eq:psi_1}, we renormalize the test functions by
$C C_g$.
Then we apply the definition of the strong unstable norm with
$\ve_1$ in place of $\ve$.	Thus,
\begin{equation}
\label{eq:second unstable}
\sum_j \left|\int_{U^1_j} f \, \psi_1\circ T^n -
\int_{U^2_j} f  \, \phi_j \; \right|	 
 \leq (C C_g) C \delta_0^{-1} |\log \ve_1|^{-\varsigma} \|f\|_u (\# \cM_0^n) \, ,
\end{equation}
where we used Lemma~\ref{lem:growth}(b) with $\bar\gamma = 0$ since there is at most
one matched piece $U^1_j$ corresponding to each component $W^{1}_i \in \cG_n(W^1)$ of $T^{-n}W^1$.

It remains to estimate the second term in \eqref{eq:stepone} using the
strong stable norm.
\begin{equation}
\label{eq:unstable strong}
 \left|\int_{U^2_j} f (\phi_j - \psi_2 \circ T^n) \right|		 
\leq  \|f\|_s | \log |U^2_j||^{-\gamma}
       \left|\phi_j -  \psi_2 \circ T^n\right|_{C^\beta(U^2_j)} \, .
\end{equation}
In order to estimate the $C^\beta$-norm of the function in \eqref{eq:unstable strong}, we 
use  that $|G_{U^2_j}|_{C^1}\le C_g$ and $ |G_{U^2_j}^{-1}|_{C^1} \le C_g$ to write
\begin{equation}
\label{eq:diff}
| \phi_j - \psi_2 \circ T^n|_{C^\beta(U^2_j)} \;
\leq \; C_g  | \psi_1
\circ T^n\circ G_{U^1_j} -\psi_2 \circ T^n\circ
G_{U^2_j}|_{C^\beta(I_j)} \, .
\end{equation}
The difference can now be bounded by the following estimate from 
\cite[Lemma~4.4]{demzhang11}
\begin{equation}
\label{lem:test}
|\psi_1 \circ T^n \circ G_{U^1_j} - \psi_2 \circ T^n \circ G_{U^2_j} |_{C^\beta(I_j)}
\le C  \ve^{\alpha - \beta}\, . 
\end{equation}

Indeed, using \eqref{lem:test} together with \eqref{eq:diff}
yields by \eqref{eq:unstable strong}
\begin{equation}
\label{eq:unstable three}
\begin{split}
& \sum_j \Big| \int_{U^2_j} f (\phi_j - \psi_2 \circ T^n ) \, dm_W \Big| \\
&\qquad \le C \|f\|_s \sum_j |\log |U^2_j||^{-\gamma} 
\, \ve^{\alpha - \beta}
\le C   |\log \delta_0 |^{-\gamma} \|f\|_s \ve^{\alpha - \beta} 2 \delta_0^{-1}(\# \cM_0^n)\, ,
\end{split}
\end{equation}
where used (as in \eqref{eq:second unstable}) Lemma~\ref{lem:growth}(b) with $\bar\gamma = 0$ since there is at most
one matched piece $U^2_j$ corresponding to each component $W^{2}_i \in \cG_n(W^2)$ of $T^{-n}W^2$. Since $\delta_0<1$ is fixed,
this completes the estimate on the second term of matched pieces in \eqref{eq:stepone}.

\smallskip
We next estimate over the unmatched pieces $V^\ell_i$ 
in \eqref{eq:unstable split}, using the strong stable norm.  Note that
by \eqref{eq:C1 C0}, $|\psi_\ell \circ T^n|_{C^\beta(V^\ell_i)} \leq C |\psi_\ell|_{C^\alpha(W^\ell)} \leq C$.  
The relevant sum for unmatched pieces in $\cG_n(W^1)$ is
\begin{equation}
\label{eq:unmatched}
\sum_i \int_{V^1_i} f \psi_1 \circ T^n \, dm_{V^1_i} \, ,
\end{equation}
with a similar sum for unmatched pieces in $\cG_n(W^2)$.

We say an unmatched curve $V^1_i$ is created at time $j$, $1 \le j \le n$, if  $j$ is the 
first time that $T^{n-j}V^1_i$ is not part of a matched element of $\cG_j(W^1)$.  Indeed, there may
be several curves $V^1_i$ (in principle exponentially
many in $n-j$) such that $T^{n-j}V^1_i$ belongs to the same unmatched element 
of $\cG_j(W^1)$.
Define
\begin{align*}
A_{j,k} = \{ i : V^1_i &\mbox{ is created at time $j$} 
\\ &
\mbox{and $T^{n-j}V^1_i$ belongs to the
unmatched curve $W^1_k \subset T^{-j}W^1$} \} \, .
\end{align*}
Due to the uniform hyperbolicity of $T$,  and, again,
uniform transversality of
$\cS_{-n}$ with the stable cone and of $C^s(x)$ with $C^u(x)$, we have $|W^1_k| \le C \Lambda^{-j} \ve$.

 Let $\delta_1$  be the value of  
$\delta\le \delta_0$ from Lemma~\ref{lem:short} associated with $\ve = 1/4$ (recall  \eqref{eq:delta1}).
For a certain time, 
the iterate $T^{-q}W^1_k$ remains shorter than length
$\delta_1$.  In this case, by Lemma~\ref{lem:growth}(a) for $\bar \gamma=0$,
 its complexity grows subexponentially,
\begin{equation}
\label{eq:slow grow}
\# \cG_q(W^1_k) \le (Km+1)^{q/m} \, .
\end{equation}
We would like to establish  the maximal value of $q$ as a function of $j$.

More precisely, we want to find $q(j)$ so that any $q\le q(j)$ satisfies the  conditions:

  (a) $T^{-q}W^1_k$ remains shorter than length $\delta_1$; \qquad
 
  (b) $\displaystyle \frac{|\log |T^{-q}W^1_k||^{-\gamma}}{|\log \ve|^{-\varsigma}} \le 1$.

For (a), we use \eqref{eq:control} together with the fact that $|W^1_k| \le C \Lambda^{-j} \ve$ to estimate
\[
|T^{-q}W^1_k| \le \delta_1 \impliedby C''|W^1_k|^{2^{-s_0q}} \le \delta_1
\impliedby C'' \Lambda^{-j 2^{-s_0q}} \ve^{2^{-s_0q}} \le \delta_1 \, . 
\]
Omitting the $\ve^{2^{-s_0q}}$ factor and solving the last inequality for $q$ yields,
\begin{equation}
\label{eq:max q a}
q \le \frac{\log j}{s_0 \log 2} + C_2 \, ,
\mbox{ where } C_2 = \frac{\log ( \frac{\log \Lambda}{|\log(\delta_1/C'')|})}{s_0 \log 2}\, .
\end{equation}

For (b), we again use \eqref{eq:control} to bound $|T^{-q}W^1_k| \le C'' (\Lambda^{-j} \ve)^{2^{-s_0 q}}$, 
so that
\begin{equation}
\label{eq:max q b0}
\frac{|\log (\Lambda^{-j} \ve)^{2^{-s_0q}}|^{-\gamma}}{|\log \ve|^{-\varsigma}} \le 1
\implies 2^{\gamma s_0 q} |\log \ve|^{\varsigma} \le (|\log \ve| + j \log \Lambda)^\gamma \, .
\end{equation}
implies (b). In turn, \eqref{eq:max q b0} is implied by
\begin{equation}
\label{eq:max q b}
q \le \frac{(\gamma - \varsigma) \log j}{\gamma s_0 \log 2} \, .
\end{equation}
Since the bound in \eqref{eq:max q b} is smaller than that in \eqref{eq:max q a} for $j$ larger than some
fixed constant depending only on $\delta_1$, $s_0$ and $C''$, we will use \eqref{eq:max q b}
to define $q(j)$.

Now we return to the estimate in \eqref{eq:unmatched}.  Grouping the unmatched pieces
$V^1_i$ by their creation times $j$, we estimate,
\footnote{When we sum the
integrals in the first line over the different $T^{n-j}V_i^1$, we find the
integral over $W^1_k$ since the union of those pieces is precisely $W^1_k$.}
\[
\begin{split}
\sum_i \int_{V^1_i} & f \, \psi_1 \circ T^n \, dm_{V^1_i}  \\
&=
\sum_{j=1}^n \sum_{i \in A_{j,k}} \int_{T^{n-j}V^1_i} (\cL^{n-j} f )\, \psi \circ T^j 
= \sum_{j=1}^n \sum_k \int_{W^1_k} (\cL^{n-j} f) \, \psi \circ T^j \\
& \le \sum_{j=1}^n \sum_k \sum_{V_\ell \in \cG_{q(j)}(W^1_k)} \int_{V_\ell} (\cL^{n-j-q(j)}f )\, \psi \circ T^{j+q(j)} \\
& \le \sum_{j=1}^n \sum_k \sum_{V_\ell \in \cG_{q(j)}(W^1_k)}
\| \cL^{n-j-q(j)} f \|_s C |\log |V_\ell||^{-\gamma} \\
& \le C \|f\|_s \sum_{j=1}^n \# \cM_0^j \# \cM_0^{n-j-q(j)} (Km+1)^{q(j)/m} 
|\log (\Lambda^{-j} \ve )^{2^{-s_0q(j)}}|^{-\gamma} \, ,
\end{split}
\]
where we have used \eqref{eq:slow grow} to bound $\# \cG_{q(j)}(W^1_k)$, 
the cardinality $\# \cM_0^j$ to bound
the cardinality of the possible pieces $W^1_k \subset T^{-j}W^1$, the estimate 
$\| \cL^{n-j-q(j)} f \|_s \le C \# \cM_0^{n-j-q(j)} \|f \|_s$, and, again $|T^{-q}W^1_k| \le C'' (\Lambda^{-j} \ve)^{2^{-s_0 q}}$.  We also have, by the supermultiplicativity
Lemma~\ref{lem:super},
\[
\# \cM_0^j \# \cM_0^{n-j-q(j)} \le C e^{- q(j) h_*} \# \cM_0^n \, .
\]
Thus using (b) in the definition of $q(j)$ (or, more precisely, \eqref{eq:max q b0}), we estimate
\begin{equation}
\label{eq:unmatched pieces}
\sum_i \int_{V^1_i} f\psi_1 \circ T^n \, dm_{V^1_i}
\le C \| f \|_s |\log \ve|^{-\varsigma} \# \cM_0^n \sum_{j=1}^n (Km+1)^{q(j)/m} e^{-q(j) h_*} \, .
\end{equation}
For the final sum over $j$, we let $\ve_2 = \frac 1m \log (Km+1)$ and use
\eqref{eq:max q b},
\begin{align*}
\sum_{j=1}^n (Km+1)^{q(j)/m} e^{-q(j) h_*} & = \sum_{j=1}^n e^{-q(j) (h_* - \ve_2)} 
\le \sum_{j=1}^n e^{- (h_* - \ve_2) \frac{(\gamma - \varsigma) \log j}{\gamma s_0 \log 2}} \\
&= \sum_{j=1}^n j^{-(h_* - \ve_2) \frac{\gamma - \varsigma}{\gamma s_0 \log 2}}\,  .
\end{align*}
Then by \eqref{eq:unmatched pieces},  since the exponent of $j$ in the above sum is strictly negative by choice of $m$ (see \eqref{eq:m}), 
there exist $C<\infty$ and  $\varpi \in [0, 1)$ such that the contribution to $\|\cL^n f\|_u$ of the 
unmatched pieces is bounded by
\begin{equation}\label{eq:first unstable}
\sum_{\ell,i}  \left|\int_{V^\ell_i} f \,\psi_\ell\circ T^n \, dm_W \right| \leq C |\log \ve|^{-\varsigma} 
n^{\varpi} \# \cM_0^n \| f \|_s  \, .
\end{equation}

Now we use  \eqref{eq:first unstable} together with \eqref{eq:second unstable}
and \eqref{eq:unstable three}
 to estimate \eqref{eq:unstable split}
\begin{align*}
& \left|\int_{W^1} \cL^n f \, \psi_1 \, dm_W - \int_{W^2} \cL^n f  \, \psi_2 \, dm_W \right|\\
&  \qquad \leq   C \delta_0^{-1} \|f\|_u |\log \ve_1|^{-\varsigma} \# \cM_0^n 
+ C \delta_0^{-1} ( n^{\varpi} \|f\|_s |\log \ve|^{- \varsigma} + \|f\|_s \ve^{\alpha - \beta} ) \# \cM_0^n \, .
\end{align*}
Dividing through by $|\log \ve|^{-\varsigma}$ and taking
the appropriate suprema, we complete the proof of \eqref{eq:unstable ly},
recalling Proposition~\ref{cor:exp}.

Finally, we study the consequences of the additional assumption  $h_* > s_0 \log 2$
on the estimate over unmatched pieces.   
In this case,  again recalling \eqref{eq:m} and following, we may choose
$\varsigma>0$ small enough such that
\[
\ve_1 := h_* - \frac 1m \log (Km+1) - \frac{\gamma}{\gamma - \varsigma} s_0 \log 2 > 0 \, .
\] 
Then
\[
\begin{split}
\sum_{j=1}^n j^{-(h_* - \ve_2) \frac{\gamma - \varsigma}{\gamma s_0 \log 2}}
= \sum_{j=1}^n j^{-1 - \ve_1 \frac{\gamma - \varsigma}{\gamma s_0 \log 2}} < \infty \, .
\end{split}
\]
Thus, by \eqref{eq:unmatched pieces}, the contribution to $\|\cL^n f\|_u$ of the unmatched pieces is bounded by 
\begin{equation}
\label{eq:first unstableb}
\sum_{\ell,i}  \left|\int_{V^\ell_i} f \,\psi_\ell\circ T^n \, dm_W \right| 
\leq C |\log \ve|^{-\varsigma}  \# \cM_0^n \| f \|_s
\end{equation}
 if $h_* > s_0 \log 2$.
 So we find \eqref{eq:unstable lyb}  for $h_* > s_0 \log 2$ by
replacing \eqref{eq:first unstable} with \eqref{eq:first unstableb}.

\subsection{Upper and Lower Bounds on the Spectral Radius}
\label{prspu}

We now deduce the bounds of Theorem~\ref{sprad} 
from the inequalities of Proposition~\ref{prop:ly} and the rate of growth 
of stable curves proved in Proposition~\ref{prop:good growth}.

\begin{proof}[Proof of  Theorem~\ref{sprad}]
The upper bounds \eqref{spu}  and \eqref{spub} are immediate consequences of Proposition~\ref{prop:ly}.
To prove the lower bound
on $| \cL^n 1 |_w$, 
recall the choice of $\delta_1=\delta >0$
from Lemma~\ref{lem:short} for  $\ve = 1/4$, giving \eqref{eq:delta1}.
Let $W \in \cW^s$ with $|W| \ge \delta_1/3$ and
set the test function $\psi \equiv 1$.  For $n \ge n_1$, 
\begin{equation}
\begin{split}
\label{eq:lower}
\int_W \cL^n 1 \, dm_W & = \sum_{W_i \in \cG_n^{\delta_1}(W)} \int_{W_i} 1 \, dm_{W_i}
= \sum_{W_i \in \cG_n^{\delta_1}(W)} |W_i| 
\ge \frac{2\delta_1}{9} c_0 e^{n h_*} \, ,
\end{split}
\end{equation}
by \eqref{eq:length lower}.
Thus,
\begin{equation}
\label{lowerb}
\| \cL^n 1 \|_s\ge  | \cL^n 1 |_w 
\ge  \frac{2 \delta_1}{9} c_0 e^{n h_*} \, . 
\end{equation}
 Letting $n$ tend to infinity, one obtains
 $\lim_{n \to \infty} \| \cL^n \|_{\cB}^{1/n} \ge e^{h_*}$.  
\end{proof}

\subsection{Compact Embedding}
\label{cembedsec}
 
The following compact embedding property is crucial to exploit Proposition~\ref{prop:ly} in
order to construct
 $\mu_*$ in Section~\ref{atlast0}.
 
\begin{proposition}[Compact Embedding]\label{cpte}
The embedding of the unit ball of $\cB$ in $\cB_w$ is compact.
\end{proposition}

\begin{proof}
Consider the set $\hW^s$ of (not necessarily homogeneous) cone-stable curves with uniformly bounded
curvature and the distance $d_{\cW^s}(\cdot , \cdot)$ between them defined in Section~\ref{normdef}.  
According to \eqref{eq:graph}, each of these curves can be viewed as
graphs of $C^2$ functions of the position coordinate $r$ with uniformly bounded second derivative,
$W = \{ G_W(r) \}_{r \in I_w} = \{ (r, \vf_W(r)) \}_{r \in I_W}$.  Thus
they are compact in the $C^1$ distance $d_{\cW^s}$.  Given $\ve>0$, we may choose finitely many
$V_i \in \hW^s$, $i = 1, \ldots N_\ve$, such that  
the balls of radius $\ve/2$ in the $d_{\cW^s}$ metric
centered at the curves
$\{ V_i \}_{i=1}^{N_\ve}$ form a covering of $\hW^s$.  

Since $\cW^s \subset \hW^s$, we proceed as follows.  
In each ball $B_{\ve/2}(V_i)$ centered at $V_i$ in
the space of $C^1$ graphs, if $B_{\ve/2}(V_i) \cap \cW^s \neq \emptyset$, then we choose
one representative $W_i \in B_{\ve/2}(V_i) \cap \cW^s$.  Otherwise, we discard $B_{\ve/2}(V_i)$.
The 
balls of radius $\ve$ in the $d_{\cW^s}$ metric
centered at the curves
 $\{ W_i \}_{i=1}^{N_\ve}$ constructed in this way form a covering of $\cW^s$.  (There may be fewer than $N_\ve$ such curves due to some balls having been discarded, but we will continue to use the symbol $N_\ve$ in any case.)  

We now argue one component of the phase space, $M_\ell = \partial B_\ell \times [-\pi/2, \pi/2]$, at a time.
Define $\mathbb{S}^1_\ell$ to be the circle of length $|\partial B_\ell|$ and let $C_g$ be the
graph constant from \eqref{eq:psi_1}.  
Since the ball of radius $C_g$ in the $C^\alpha(\mathbb{S}^1_\ell)$
norm is compactly embedded in $C^\beta(\mathbb{S}^1_\ell)$, we may choose finitely many functions
$\bpsi_j \in C^\alpha(\mathbb{S}^1_\ell)$ such that 
the balls of radius $\ve$ in the $C^\beta(\mathbb{S}^1_\ell)$ metric
centered at the functions
$\{ \bpsi_j \}_{j=1}^{L_\ve}$ form a covering
of the ball of radius $C_g$ in  $C^\alpha(\mathbb{S}^1_\ell)$.

Now let $W = G_W(I_W) \in \cW^s$, and $\psi \in C^\alpha(W)$ with $|\psi|_{C^\alpha(W)} \le 1$.
Viewing $I_W$ as a subset of $\mathbb{S}^1_\ell$, we define the push down of $\psi$ to 
$I_W$ by $\bpsi = \psi \circ G_W$.  We extend $\bpsi$ to $\mathbb{S}^1_\ell$ by linearly interpolating
between its two endpoint values on the complement of $I_W$ in $\mathbb{S}^1_\ell$.  Since
$I_W$ is much shorter than $\mathbb{S}^1_\ell$, this
can be accomplished while maintaining $|\bpsi|_{C^\alpha(\mathbb{S}^1_\ell)} \le C_g$.

Choose $W_i = G_{W_i}(I_{W_i})$ such that $d_{\cW^s}(W, W_i) < \ve$ and
$\bpsi_j$ such that $|\bpsi - \bpsi_j|_{C^\beta(\mathbb{S}^1_\ell)} < \ve$.  Define
$\psi_j = \bpsi_j \circ G_{W_i}^{-1}$ and $\tpsi_j = \bpsi_j \circ G_W^{-1}$ to be the lifts of $\bpsi_j$ to 
$W_i$ and $W$, respectively.  Note that $|\psi_j|_{C^\beta(W_i)} \le C_g$, $|\tpsi_j|_{C^\beta(W)} \le C_g$, while
\[
d(\psi_j , \tpsi_j) = |\psi_j \circ G_{W_i} - \tpsi_j \circ G_W|_{C^0(I_{W_i} \cap I_W)}  = 0 ,
\quad \mbox{and} \quad |\psi - \tpsi_j|_{C^\beta(W)} \le C_g \ve \, .
\]
Thus,
\begin{align*}
&\left| \int_W f \psi \, dm_W - \int_{W_i} f \psi_j \, dm_{W_i} \right|\\
 &\qquad\qquad \le \left| \int_W f (\psi - \tpsi_j) \, dm_W \right| + \left| \int_W f \tpsi_j \, dm_W - \int_{W_i} f \psi_j \, dm_{W_i} \right|  \\
&\qquad\qquad \le \| f \|_s |\log|W||^{-\gamma} |\psi - \tpsi_j|_{C^\beta(W)} + |\log \ve|^{-\varsigma} \| f \|_u C_g 
 \le 2C_g  \| f\|_{\cB}  |\log \ve|^{-\varsigma}\,  .
\end{align*}

We have proved that for each $\ve>0$, there exist finitely many bounded linear functionals 
$\ell_{i,j}(\cdot) = \int_{W_i} \cdot \, \, \psi_j \, dm_{W_i}$, such that for all $f \in \cB$,
\[
| f |_w \le \max_{i \le N_\ve, j \le L_\ve} \ell_{i,j}(f) + 2C_g  \| f\|_{\cB}  |\log \ve|^{-\varsigma} \, ,
\]
which implies the relative compactness of $\cB$ in $\cB_w$.
\end{proof}

 \section{The Measure $\mu_*$}
 \label{mmu}

\noindent\emph{In this section, we assume throughout that $h_*> s_0 \log 2$ (with $s_0 <1$ defined by \eqref{defs0}).}


\subsection{Construction of the Measure $\mu_*$ --- Measure
 of Singular Sets (Theorem~\ref{nbhdthm})}
\label{atlast0}
 
 In this section, we construct a $T$-invariant probability measure $\mu_*$
 on $M$  by combining in \eqref{defm} a maximal eigenvector of
$\LL$ on $\BB$ and a maximal eigenvector of its dual obtained
in Proposition~\ref{prop:exist}.
In addition,  the information on these left and right eigenvectors
will give Lemma~\ref{lem:approach} and Corollary~ \ref{integral}, which immediately
imply  Theorem~\ref{nbhdthm}. 

We first show that such maximal eigenvectors exist and are in fact nonnegative Radon measures
(i.e., elements of the dual of $C^0(M)$).

\begin{proposition}
\label{prop:exist} 
If $h_*> s_0 \log 2$  then 
there exist $\nu \in \BB_w$ and $\tilde \nu \in \BB_w^*$ such that
$\LL \nu=e^{h_*} \nu$ and $\LL^* \tilde \nu=e^{h_*} \tilde \nu$. In addition\footnote{Recall Proposition~\ref{distembed} and Remark~\ref{distembed*}.}  $\nu$ and $\tilde \nu$
take nonnegative values on nonnegative $C^1$ functions on $M$ and are thus
nonnegative Radon measures. Finally, $\tilde \nu(\nu)\ne 0$ and $\| \nu \|_u \le \bar{C}$.
\end{proposition}

\begin{remark}
 The norm of the space $\BB$ depends on the parameter $\gamma$ and
is used in the proof of the proposition. However, this proof provides $\nu$ and $\tilde \nu$ which do not
depend on $\gamma$  (as soon as  $2^{s_0\gamma} < e^{h_*}$),
and do not depend on the parameters $\beta$ and $\varsigma$
of $\BB$.
\end{remark}

 It is easy to see that 
$|f \varphi|_w\le |\varphi|_{C^1}
|f|_w$ (use 
 $|\varphi \psi|_{C^\alpha(W)}\le |\varphi|_{C^1} |\psi|_{C^\alpha(W)}$). 
Clearly, if $f\in C^1$ and $\varphi \in C^1$ then $f\varphi \in C^1$.
Therefore,  if $h_*> s_0 \log 2$,
a bounded linear map  $\mu_*$ from $C^1(M)$ to  $\mathbb{C}$ can be defined by
taking $\nu$ and $\tilde \nu$ from Proposition~\ref{prop:exist} and setting
\begin{equation}\label{defm}
\mu_*(\varphi)= \frac{ \tilde \nu (\nu \varphi)}{ \tilde \nu (\nu)}\, .
\end{equation}
This map
is nonnegative for all nonnegative $\varphi$ and thus defines a nonnegative measure $\mu_*\in (C^0)^*$,
with  $\mu_*(1)=1$.
Clearly,  $\mu_*$ is a $T$ invariant  probability measure
since for every $\varphi \in C^1$ we have
$$
\tilde \nu(\nu \varphi)=e^{-h_*}\tilde \nu(\varphi \LL(\nu))=e^{-h_*} \tilde \nu(\LL(\nu (\varphi \circ T)))=
\tilde \nu (\nu (\varphi \circ T))=\tilde \nu(\nu)\mu_*(\varphi \circ T) \, .
$$

\begin{proof}[Proof of Proposition~\ref{prop:exist}]
Let $1$ denote the constant function\footnote{We could replace the
seed function $1$ by any $C^1$ positive function $f$ on $M$.} equal to one
on $M$. We will take this as a seed in our construction of
a maximal eigenvector. From  
\eqref{spl} in Theorem~\ref{sprad} we see that
$\| \cL^n 1\|_\BB\ge \| \cL^n 1\|_s\ge | \cL^n 1|_w \ge C \# \cM_0^n \ge C e^{n h_*}$.  Now, 
consider 
\begin{equation}\label{nun}
\nu_n = \frac{1}{n} \sum_{k=0}^{n-1} e^{-k h_*} \cL^k 1  \in \cB\, , \quad n \ge 1\, .
\end{equation}
By construction the $\nu_n$ are nonnegative, and thus Radon measures.
By 
our assumption on $h_*$ and \eqref{spub} in Theorem~\ref{sprad}
they satisfy $\| \nu_n \|_\cB \le \bar{C}$, so  using the relative compactness of $\cB$ in $\cB_w$
(Proposition~\ref{cpte}),
we extract a subsequence $(n_j)$ such that $\lim_j \nu_{n_j} = \nu$ is a nonnegative
measure,
and the convergence
is in $\cB_w$.  (Changing the value
of $\gamma$ does not affect $\nu$ since $\cB_w$
does not depend on $\gamma$.) Since $\cL$ is continuous on $\cB_w$, we may write,
\[
\begin{split}
\cL \nu & = \lim_{j \to \infty}\frac{1} {n_j} \sum_{k=0}^{n_j-1} e^{-k h_*} \cL^{k+1} 1 \\
& = \lim_{j \to \infty} \left( \frac{e^{h_*}} {n_j} \sum_{k=0}^{n_j-1} e^{-k h_*} \cL^k 1 
- \frac{1}{n_j} e^{h_*}  1 + \frac{1}{n_j} e^{-(n_j-1) h_*} \cL^{n_j} 1 \right)  = e^{h_*} \nu \, ,
\end{split}
\]
where we used that the second and third terms go to $0$ (in the $\cB$-norm).
We thus obtain a  nonnegative measure $\nu \in \cB_w$ such that
$\cL \nu = e^{h_*} \nu$.  

Although $\nu$ is not a priori an element of $\cB$, it does inherit bounds
on the unstable 
norm from the sequence $\nu_n$.  The convergence of $(\nu_{n_j})$ to $\nu$ in
$\cB_w$ implies that 
\begin{equation}
\label{eq:norm bound}
\lim_{j \to \infty} \sup_{W \in \cW^s} \sup_{\substack{\psi \in C^\alpha(W) \\ |\psi|_{C^\alpha(W)} \le 1}}
\left( \int_W \nu \, \psi \, dm_W
- \int_W \nu_{n_j} \, \psi \, dm_W \right) = 0 \, .
\end{equation}
Since 
$\| \nu_{n_j} \|_u \le \bar{C}$, it follows that $\| \nu \|_u \le \bar{C}$, as claimed.

Next, recalling the
bound $|\int f \, d\musrb|\le \hat C |f|_w$ from Proposition~\ref{distembed},
setting  $d\musrb\in (\cB_w)^*$ to be the functional
defined on $C^1(M)\subset \cB_w$ by $d\musrb(f)=\int f\, d\musrb$
and extended by density,  we define\footnote{We could again replace the seed $\musrb$
by  $f\musrb$ for any $C^1$ positive function $f$ on $M$.} 
\begin{equation}\label{7.*}
\tilde \nu_n=
\frac 1 {n}
\sum_{k=0}^{n-1} e^{-k h_*}  (\LL^*)^k (d\musrb) \, .
\end{equation}
Then,  we have
$|\tilde \nu_n(f)|\le  C|f|_w$ for all $n$ and all $f \in  \BB_w$.
So $\tilde \nu_n$ is bounded in $(\BB_w)^* \subset \BB^*$. By compactness
of this embedding (Proposition~\ref{cpte}), we can find a subsequence $\tilde\nu_{\tilde n_j}$
converging to $\tilde \nu \in \BB^*$. By the argument above, we have $\LL^* \tilde \nu =e^{h_*}
\tilde \nu$. 
The nonnegativity claim on $\tilde \nu$ follows by 
construction.\footnote{To check $\gamma$-independence of $\tilde \nu$,
note that if $\tilde \gamma>\gamma$
then, since  the dual norms satisfy $\|\tilde\nu_{\tilde n_j}-\tilde \nu\|_{*,\tilde \gamma}
\le \|\tilde\nu_{\tilde n_j}-\tilde \nu\|_{*,\gamma}$, the subsequence converges to $\tilde \nu$
in the $\|\cdot \|_{*,\tilde\gamma}$-norm as well.
If $\tilde \gamma<\gamma$ then  a further subsequence of $\tilde n_j$ must
converge to some $\tilde \nu_{\tilde \gamma}$ in the $\|\cdot \|_{*,\tilde\gamma}$ norm. The domination
then implies $\tilde\nu=\tilde \nu_{\tilde \gamma}$.}

We next check that $\tilde \nu$, which in principle lies
in the dual of $\cB$, is in fact an element of $(\cB_w)^*$.
For this, it suffices to find $\tilde C<\infty$ so that for any
$f\in \cB$ we have
\begin{equation}
\label{eq:bweak}
\tilde \nu(f)\le \tilde C | f|_{w} \, . 
\end{equation}
Now, for $f \in \cB$ and any $n \ge 1$, we have
\[
|\tilde\nu(f)| \le |(\tilde\nu - \tilde\nu_n)(f)| + |\tilde\nu_n(f)| 
\le |(\tilde\nu - \tilde\nu_n)(f)| + |f|_w \, .
\]
Since $\tilde\nu_n \to \tilde\nu$ in $\cB^*$, we conclude $|\tilde\nu(f)| \le |f|_w$ for all
$f \in \cB$.  Since $\cB$ is dense in $\cB_w$, 
by \cite[Thm~I.7]{reed} $\tilde\nu$ extends uniquely to a 
bounded linear functional on $\cB_w$, satisfying \eqref{eq:bweak}.
It only remains to see that $\tilde \nu(\nu)>0$.

Let $(n_j)$ (resp. $(\tilde n_j)$) denote the subsequence such that $\nu = \lim_j \nu_{n_j}$ (resp. $\tilde\nu = \lim_j \tilde\nu_{\tilde n_j}$.)
Since $\tilde\nu$ is continuous on $\cB_w$, we have on the one hand
\begin{equation}\label{7.6a}
\tilde\nu(\nu) = \lim_{j \to \infty} \tilde\nu(\nu_{n_j}) = \lim_j \frac{1}{n_j} \sum_{k=0}^{n_j-1} e^{-k h_*} \tilde\nu(\cL^k1) 
= \lim_j \frac{1}{n_j} \sum_{k=0}^{n_j-1} \tilde\nu(1) = \tilde\nu(1)\, ,
\end{equation}
where we have used that $\tilde\nu$ is an eigenvector for $\cL^*$.
On the other hand,
\begin{equation}
\label{eq:tilde}
\tilde\nu(1) = \lim_{j \to \infty} \frac{1}{\tilde n_j} \sum_{k=0}^{\tilde n_j-1} e^{-k h_*} (\cL^*)^kd\musrb(1) 
= \lim_j \frac{1}{\tilde n_j} \sum_{k=0}^{\tilde n_j-1} e^{-k h_*} \int \cL^k 1 \, d\musrb \,  .
\end{equation}
Next, we disintegrate $\musrb$ as in the proof of Lemma~\ref{lem:embed bound} into conditional measures
$\musrb^{W_\xi}$ on maximal homogeneous stable manifolds $W_\xi \in \cW^s_{\bH}$ and a factor measure $d{\hatmusrb}(\xi)$
on the index set $\Xi$ of stable manifolds.  Recall that $\musrb^{W_\xi} = |W_\xi|^{-1} \rho_\xi dm_W$, where $\rho_\xi$
is uniformly log-H\"older continuous so that
\begin{equation}
\label{eq:log rho}
0< c_\rho \le \inf_{\xi \in \Xi} \inf_{W_\xi} \rho_\xi \le \sup_{\xi \in \Xi} |\rho_\xi|_{C^\alpha(W_\xi)} \le C_\rho 
< \infty \, .
\end{equation}
Let $\Xi^{\delta_1}$ denote those $\xi \in \Xi$ such that $|W_\xi| \ge \delta_1/3$ and note that ${\hatmusrb}(\Xi^{\delta_1}) > 0$.
Then, disintegrating as usual, we get 
by \eqref{eq:lower} for $k \ge n_1$,
\[
\begin{split}
\int \cL^k 1 \, d\musrb &= \int_{\Xi} \int_{W_\xi} \cL^k 1 \, \rho_\xi |W_\xi|^{-1} \, dm_{W_\xi} d{\hatmusrb}(\xi) \\
& \ge \int_{\Xi^{\delta_1}} \int_{W_\xi} \cL^k 1 \, dm_{W_\xi} c_\rho 3 \delta_1^{-1} d{\hatmusrb}(\xi)
\ge c_\rho \frac{2c_0}{3} e^{k h_*} \hatmusrb(\Xi^{\delta_1}) \, .  \\
\end{split}
\]
Combining this with \eqref{7.6a} and \eqref{eq:tilde} yields
$\tilde\nu(\nu) = \tilde\nu(1) \ge \frac{2c_\rho c_0}{3} {\hatmusrb}(\Xi^{\delta_1}) > 0$ as required.
\end{proof}

We next study the measure of neighbourhoods of singularity
sets and stable manifolds, in order to establish \eqref{nbhd} in Theorem~\ref{nbhdthm}.

\begin{lemma}
\label{lem:approach} 
For any $\gamma>0$ such that $2^{s_0 \gamma} <e^{h_*}$ and
any $k \in \mathbb{Z}$, there exists $C_k > 0$ such that 
$$\mu_*(\cN_\ve(\cS_k)) \le C_k |\log \ve|^{-\gamma} \, ,
\qquad \forall \ve > 0\, . $$
In particular, for any $p > 1/\gamma$
(one can choose $p<1$ if $\gamma > 1$), $\eta > 0$, and $k \in \mathbb{Z}$, for
$\mu_*$-almost every $x \in M$, there exists $C>0$ such that 
\begin{equation}
\label{rateeta}
d(T^nx, \cS_k) \ge C e^{-\eta n^p}\, , \quad \forall n \ge 0 \, .
\end{equation}
\end{lemma}

\begin{proof}
First, for each $k \ge 0$, we claim that there exists $C_k > 0$ such that for all  $\ve > 0$,
\begin{equation}
\label{-7}
|\nu(\cN_\ve(\cS_{-k}))| \le C |1_{k, \ve} \nu |_w  \le C_k |\log \ve|^{-\gamma}  \, ,\end{equation}
where $1_{k, \ve}$ is the indicator function of the set $\cN_{\ve}(\cS_{-k})$.
To prove the first inequality in \eqref{-7},
first note that since $\cS_{-k}$ comprises finitely many smooth
curves, uniformly transverse to the stable cone, this also holds
for the boundary curves of the set $\cN_\ve(\cS_{-k})$.  By \cite[Lemma~5.3]{demzhang14}, we have $1_{k ,\ve} f \in \cB$ for $f\in \cB$; similarly (and by a simpler approximation)
 if $f \in \cB_w$, then $1_{k,\ve} f \in \cB_w$. So the first inequality
in \eqref{-7}
follows from  Lemma~\ref{lem:embed bound}.

We next prove the second inequality in \eqref{-7}.
Let $W \in \cW^s$ and $\psi \in C^\alpha(W)$ with $|\psi|_{C^\alpha(W)} \le 1$. 
Due to the uniform transversality of the curves in $\cS_{-k}$ with the stable cone, the 
intersection $W \cap \cN_\ve(\cS_{-k})$
can be expressed as a finite union with
cardinality  bounded by a constant $A_k$ (depending only on $\cS_{-k}$) of stable
manifolds $W_i \in \cW^s$, of lengths  at most $C\ve$.  Therefore,
for any $f\in C^1$,
\begin{equation}
\int_{W_\xi} f \, 1_{k, \ve} \, \psi \, dm_W 
= \sum_i \int_{W_i} f \, \psi \, dm_{W_i} \le \sum_i |f|_w  |\psi|_{C^\alpha(W_i)}
  \le C A_k | f|_w  \, .
\end{equation}
It follows that
$
|1_{k, \ve} f |_w \le A_k |f|_w$ for all $f \in \cB_w$.
Similarly, we have $
|1_{k, \ve} f |_w \le A_k \|f\|_s |\log \ve|^{-\gamma}$ for all $f \in \cB$.
Now recalling $\nu_n$ from \eqref{nun}, 
we estimate,
\[
|1_{k,\ve} \nu|_w \le |1_{k,\ve} (\nu - \nu_n)|_w + |1_{k,\ve} \nu_n|_w
\le A_k |\nu - \nu_n |_w + C_k' |\log \ve|^{-\gamma} \| \nu_n \|_{\cB}\,  .
\]
Since $\| \nu_n \|_{\cB} \le \bar{C}$ for all $n \ge 1$, we take the limit as $n \to \infty$ to conclude
that $|1_{k, \ve} \nu|_w \le C_k |\log \ve|^{-\gamma}$, 
concluding the proof of \eqref{-7}.

Next, applying \eqref{eq:bweak}, we have
\[
\mu_*(\cN_\ve(\cS_{-k})) = \tilde{\nu}(1_{k, \ve} \nu) \le \tilde{C}|1_{k, \ve} \nu|_w \le \tilde{C} C_k |\log \ve|^{-\gamma}, \qquad \forall k \ge 0\,  .
\]
To obtain the analogous bound for $\cN_\ve(\cS_k)$, for $k >0$, we use the invariance of $\mu_*$.
It follows from the time reversal of \eqref{eq:image} 
that 
$T(\cN_\ve(\cS_1)) \subset \cN_{C \ve^{1/2}}(\cS_{-1})$.
Thus,
\[
\mu_*(\cN_\ve(\cS_1)) \le \mu_*(\cN_{C\ve^{1/2}}(\cS_{-1})) \le C_1 |\log (C\ve^{1/2})|^{-\gamma}
\le C_1' |\log \ve|^{-\gamma} .
\]
The estimate for $\cN_\ve(\cS_k)$, for $k \ge 2$, follows similarly since $T^k\cS_k = \cS_{-k}$. 

Finally, fix $\eta > 0$, $k \in \mathbb{Z}$ and $p > 1/\gamma$.  Since
\begin{equation}
\label{eq:borel}
\sum_{n \ge 0} \mu_*(\cN_{e^{-\eta n^p}}(\cS_k)) \le \tilde C
 C_k \eta^{-\gamma} \sum_{n \ge 1} n^{-p \gamma} < \infty,
\end{equation}
by the Borel--Cantelli Lemma, $\mu_*$-almost every $x \in M$ visits $\cN_{e^{-\eta n^p}}(\cS_k)$ only
finitely many times, and the last statement of the lemma follows.
\end{proof}

Lemma~\ref{lem:approach} will imply the following:
\begin{cor}
\label{integral} 
\noindent a)  For any $\gamma>0$ so that
$2^{s_0 \gamma}< e^{h_*}$  and any $C^1$ curve $S$ uniformly transverse to the stable cone,  there exists $C>0$ such that 
$
\nu(\cN_\ve(S)) \le C | \log \ve|^{-\gamma}$
and
$\mu_*(\cN_\ve(S)) \le C | \log \ve|^{-\gamma}$ for all $\ve > 0$.\\
\noindent b) The measures
$\nu$ and $\mu_*$ have no atoms, and $\mu_*(W)=0$ for all $W \in \cW^s$ and $W \in \cW^u$.\\
\noindent c) $\int |\log d(x, \cS_{\pm 1})| \, d\mu_* < \infty$. \\
\noindent d) $\mu_*$-almost every point in $M$ has a stable and unstable manifold of positive length.
\end{cor}

\begin{proof}
a)  This follows immediately from the bounds in the proof of Lemma~\ref{lem:approach} since the only property required
of $\cS_{-k}$ is that it comprises finitely many smooth curves uniformly transverse to the stable cone.

\smallskip
\noindent
b) That $\nu$ and $\mu_*$ have no atoms follows from part (a).  If $\mu_*(W) = a >0$, then by invariance, $\mu_*(T^nW) = a$
for all $n > 0$.  Since $\mu_*$ is a probability measure and $T^n$ is continuous on stable manifolds, 
$\cup_{n \ge 0} T^nW$ must be the union of only finitely many smooth curves. 
Since $|T^nW| \to 0$ there is a subsequence $(n_j)$ such that $\cap_{j \ge 0} T^{n_j}W = \{ x \}$.  Thus $\mu_*(\{ x \}) = a$,
which is impossible.  A similar argument applies to $W \in \cW^u$, using the fact that $T^{-n}$ is continuous on such manifolds.

\smallskip
\noindent
c)  Choose $\gamma>1$ and $p > 1/(\gamma - 1)$.  Then by Lemma~\ref{lem:approach},
\begin{align*}
\int |\log d(x, \cS_1)|& \, d\mu_* 
 = \sum_{n \ge 0} \int_{\cN_{e^{-n^p}}(\cS_1)
\setminus \cN_{e^{-(n+1)^p}}(\cS_1)} |\log d(x, \cS_1)| \, d\mu_* \\
& \le \sum_{n \ge 0} (n+1)^p \mu_*(\cN_{e^{-n^p}}(\cS_1))
\le 1+\sum_{n \ge 1} C_1
n^{p(1-\gamma)}(1+1/n)^p < \infty .
\end{align*}
A similar estimate holds for $\int |\log d(x, \cS_{-1})| \, d\mu_*$.

\smallskip
\noindent
d)  The existence of stable and unstable manifolds for $\mu_*$-almost every $x$ follows from the Borel--Cantelli estimate
\eqref{eq:borel} by a standard argument if we choose
$\gamma >1$, $p=1$ and  $e^{\eta} < \Lambda$ (see, for example, \cite[Sect. 4.12]{chernov book}).
\end{proof}

Lemma~\ref{lem:approach} and Corollary~\ref{integral} prove all the items of Theorem~\ref{nbhdthm}.


\subsection{$\nu$-Almost Everywhere Positive Length of Unstable Manifolds}
\label{regular}

We  establish  almost everywhere positive length of unstable manifolds in
the sense of the measure $\nu$ (the maximal eigenvector of $\cL$).
The proof of this fact, as well as some arguments in subsequent sections, will require
viewing elements of $\cB_w$ as  {\em leafwise distributions}, see Definition~\ref{LW} below.
Indeed, to prove Lemma~\ref{lem:hyp}, 
we make in Lemma~\ref{lem:disint} an explicit connection\footnote{This connection
is used in Section~\ref{notmixing}.} between the element 
$\nu \in \cB_w$ viewed as a measure on $M$, and the
family of  leafwise measures  defined on the set of stable manifolds $\cW^s$.

While $\nu$ is not an invariant measure, the almost-everywhere existence of
positive length unstable manifolds on {\em every} stable manifold $W \in \cW^s$ follows
from the regularity inherited from the strong stable norm.  
This property may have some independent interest
as it has not been proved in previous uses of this type of norm 
\cite{demzhang11, demzhang14}, and it will be important for proving the absolute continuity 
of the unstable foliation
for $\mu_*$ (Corollary~\ref{cor:abs cont}), which relies on the analogous property for the measure $\nu$
(Proposition~\ref{prop:abs cont}).
Lemmas~\ref{lem:hyp} and \ref{lem:disint}  will also be useful to
obtain that $\mu_*$ has full support (Proposition~\ref{lastitem}).

\begin{definition}[Leafwise distributions and leafwise measures]\label{LW}
For $f \in C^1(M)$ and $W \in \cW^s$, the map defined on  $C^\alpha(W)$ by
\[
\psi \mapsto \int_W f \, \psi \, dm_W \, , 
\]
can be viewed as a distribution  of order $\alpha$ on $W$. Since we have the bound $|\int_W f \, \psi \, dm_W| \le |f|_w |\psi|_{C^\alpha(W)}$, the map sending $f\in C^1$ to this distribution
of order $\alpha$ on $W$  
can be extended to $f \in \cB_w$. We denote this extension by
$\int_W \psi \, f$ or $\int_W f \, \psi \, dm_W$, and  we call the corresponding family
of distributions (indexed by $W$) the leafwise distribution $(f, W)_{W\in \cW^s}$
associated with $f\in \cB_w$.
Note that if $f\in \cB_w$ is such that $\int_W \psi \, f\ge 0$ for all
$\psi \ge 0$ then using again \cite[\S I.4]{Sch}, the  leafwise distribution on $W$ extends 
to a bounded linear functional on $C^0(W)$,
i.e., it is a Radon measure. If this holds for all $W \in \cW^s$, the leafwise distribution is called
a leafwise measure.
\end{definition}

\begin{lemma}[Almost Everywhere Positive Length of Unstable Manifolds, for $\nu$]
\label{lem:hyp}
For $\nu$-almost every $x \in M$ the stable and unstable manifolds have positive length.
Moreover, viewing $\nu$ as a leafwise measure, for every $W \in \cW^s$, $\nu$-almost every $x \in W$ has an unstable manifold
of positive length.
\end{lemma}

  Recall  the disintegration of $\musrb$ into conditional
measures $\musrb^{W_\xi}$ on maximal homogeneous stable manifolds $W_\xi \in \cW^s_{\bH}$ 
and a factor measure $d{\hatmusrb}(\xi)$
on the index set $\Xi$ of homogeneous stable manifolds, with $d\musrb^{W_\xi} = |W_\xi|^{-1} \rho_\xi dm_W$, where $\rho_\xi$
is uniformly log-H\"older continuous as in \eqref{eq:log rho}.

\begin{lemma}
\label{lem:disint}  
Let $\nu^{W_\xi}$ and $\hat{\nu}$ denote the conditional measures and factor measure obtained by disintegrating
$\nu$ on the set of homogeneous stable manifolds $W_\xi \in \cW^s_{\bH}$, $\xi \in \Xi$.  Then
for any $\psi \in C^\alpha(M)$, 
\[
\int_{W_\xi} \psi \, d\nu^{W_\xi} = \frac{\int_{W_\xi} \psi \, \rho_\xi \, \nu}{\int_{W_\xi} \rho_\xi \, \nu } \quad \forall \xi \in \Xi,
\mbox{ and} \quad
d\hat{\nu}(\xi)  = |W_\xi|^{-1} \Big( \int_{W_\xi} \rho_\xi \, \nu \Big) \, d{\hatmusrb}(\xi)\,  .
\] 
Moreover, viewed as a leafwise measure, $\nu(W)>0$ for all $W \in \cW^s$.
\end{lemma}

\begin{proof}
First, we we establish the following claim:
For $W \in \cW^s$, we let $n_2 \le \bar{C}_2 |\log (|W|/\delta_1)|$ be the constant from the
proof of Corollary~\ref{cor:short}. (This is the first time $\ell$ such that $\cG_\ell(W)$ has at least one
element of length at least $\delta_1/3$.) Then there exists $\bar{C}>0$ such that for all 
$W \in \cW^s$,
\begin{equation}
\label{eq:lower nu}
\int_W \nu \ge \bar{C} |W|^{h_* \bar{C}_2} \, .
\end{equation}
Indeed, recalling \eqref{nun} and using \eqref{eq:lower}, we have for
 $\bar{C} = \frac{2 c_0}{9} \delta_1^{1-h_* \bar{C}_2}$,
\begin{align*}
\int_W \nu & = \lim_{n_j} \frac{1}{n_j} \sum_{k=0}^{n_j-1} e^{-k h_*} \int_W \cL^k 1 dm_W \\
 &\ge \lim_{n_j} \frac{1}{n_j} \sum_{k=n_2}^{n_j-1} e^{-k h_*} \sum_{W_i \in \cG_{n_2}(W)} \int_{W_i} \cL^{k-n_2} 1 dm_{W_i} \\
& \ge  \lim_{n_j} \frac{1}{n_j} \sum_{k=n_2}^{n_j-1} e^{-k h_*} \tfrac{2 \delta_1}{9} c_0 e^{h_* (k-n_2)}
\ge \tfrac{2 \delta_1}{9} c_0 e^{-h_* n_2} \ge \bar{C} |W|^{h_* \bar{C}_2} \, .
\end{align*}
This proves the last statement of the lemma.

Next, for any $f \in C^1(M)$, according to our convention, we view $f$ as an element of $\cB_w$ by considering it as a measure
integrated against $\musrb$.  Now suppose $(\nu_n)_{n \in \mathbb{N}}$ is the sequence of functions from \eqref{nun}
such that $| \nu_n - \nu|_w \to 0$.  For any
$\psi \in C^\alpha(M)$, we have
\begin{equation}
\label{eq:disint}
\begin{split}
\nu_n(\psi) & = \int_M \nu_n \, \psi \, d\musrb = \int_{\Xi} \int_{W_\xi} \nu_n \, \psi \, \rho_\xi \, dm_{W_\xi} |W_\xi|^{-1} d{\hatmusrb}(\xi) \\
& = \int_{\Xi} \frac{\int_{W_\xi} \nu_n \, \psi \, \rho_\xi \, dm_{W_\xi}}{\int_{W_\xi} \nu_n \, \rho_\xi \, dm_{W_\xi}} \, d{(\hatmusrb)}_n(\xi)\, ,
\end{split}
\end{equation}
where $d{(\hatmusrb)}_n(\xi) = |W_\xi|^{-1} \int_{W_\xi} \nu_n \, \rho_\xi \, dm_{W_\xi} \, d{\hatmusrb}(\xi)$.
By definition of convergence in $\cB_w$ (see for example \eqref{eq:norm bound}) since $\psi, \rho_\xi \in C^\alpha(W_\xi)$,
the ratio of integrals converges (uniformly in $\xi$) to
$\int_{W_\xi} \psi \, \rho_\xi \, \nu / \int_{W_\xi} \rho_\xi \, \nu$, and the factor measure converges to 
$|W_\xi|^{-1} \int_{W_\xi} \rho_\xi \, d\nu \, d{\hatmusrb}(\xi)$.  Note that since $\rho_\xi$ is uniformly log-H\"older, and
due to  \eqref{eq:lower nu},
we have $\int_{W_\xi} \nu \, \rho_\xi \, dm_{W_\xi} >0$ with lower bound 
depending only on the length of $W_\xi$.

Finally, by Proposition~\ref{distembed} and Lemma~\ref{lem:embed bound}, we have $\nu_n(\psi)$ converging to $\nu(\psi)$.
Disintegrating $\nu$ according to the statement of the lemma yields the claimed identifications. 
\end{proof}

\begin{proof}[Proof of Lemma~\ref{lem:hyp}]
The statement about stable manifolds of positive length
follows from the characterization of $\hat{\nu}$ in Lemma~\ref{lem:disint}, since the set of points with stable manifolds of
zero length has zero ${\hatmusrb}$-measure \cite{chernov book}.

We fix $W \in \cW^s$ and prove the statement about $\nu$ as a leafwise measure.  
This will imply the statement regarding unstable manifolds for the measure $\nu$ by Lemma~\ref{lem:disint}.
 
Fix $\ve > 0$ and
$\hLambda \in (\Lambda, 1)$,
 and define
$O = \cup_{n \ge 1} O_n$, where
\[
O_n = \{ x \in W : n = \min j \mbox{ such that } d_u(T^{-j}x, \cS_1) < \ve C_e \hLambda^{-j} \} ,
\]
and $d_u$ denotes distance restricted to the unstable cone.
By \cite[Lemma~4.67]{chernov book}, any $x \in W \setminus O$ has unstable manifold of length at least $2 \ve$.  We proceed
to estimate $\nu(O) = \sum_{n \ge 1} \nu(O_n)$, where equality holds since the $O_n$ are disjoint.  In addition, since
$O_n$ is a finite union of open subcurves of $W$, we have
\begin{equation}
\label{eq:O_n}
\int_W 1_{O_n} \, \nu = \lim_{j \to \infty} \int_W 1_{O_n} \, \nu_{\ell_j} = \lim_{j \to \infty} \ell_j^{-1} \sum_{k=0}^{\ell_j - 1}
e^{-k h_*} \int_W 1_{O_n} \, \cL^k 1 \, dm_W\,  .
\end{equation}
We estimate two cases.

\smallskip
\noindent
{\em Case I: $k < n$.}
Write
$\int_{W \cap O_n} \cL^k1 \, dm_W = \sum_{W_i \in \cG_k(W)} \int_{W_i \cap T^{-k}O_n} 1 \, dm_{W_i}.$

If $x \in T^{-k}O_n$, then $y = T^{-n+k}x$ satisfies $d_u(y, \cS_1) < \ve C_e \hLambda^{-n}$ and thus
we have $d_u(Ty, \cS_{-1}) \le C \ve^{1/2} \hLambda^{-n/2}$.  Due to the uniform transversality of stable and unstable cones, as well
as the fact that elements of $\cS_{-1}$ are uniformly transverse to the stable cone, we have
$d_s(Ty, \cS_{-1}) \le C \ve^{1/2} \hLambda^{-n/2}$ as well, with possibly a larger constant $C$. 

Let $r^s_{-j}(x)$ denote the distance from $T^{-j}x$ to the nearest endpoint of $W^s(T^{-j}x)$, where $W^s(T^{-j}x)$ is the maximal
local stable manifold containing $T^{-j}x$.  From the above analysis, we see that 
$W_i \cap T^{-k}O_n \subseteq \{ x \in W_i : r^s_{-n+k+1}(x) \le C \ve^{1/2} \hLambda^{-n/2} \}$.
The time reversal of the growth lemma \cite[Thm~5.52]{chernov book} gives
$m_{W_i}(r^s_{-n+k+1}(x) \le C \ve^{1/2} \hLambda^{-n/2} ) \le C' \ve^{1/2} \hLambda^{-n/2}$
for a constant $C'$ that is uniform in $n$ and $k$.  Thus, using Proposition~\ref{cor:exp}, we find
\[
\int_{W \cap O_n} \cL^k1 \, dm_W \le \# \cG_k(W) C' \ve^{1/2} \hLambda^{-n/2} \le C e^{k h_*} \ve^{1/2} \hLambda^{-n/2}\,  .
\]

\smallskip
\noindent
{\em Case II: $k \ge n$.}
Using the same observation as in Case I, if $x \in T^{-n+1}O_n$, then $x$ satisfies
$d_s(x, \cS_{-1}) \le C \ve^{1/2} \hLambda^{-n/2}$.  We change variables to estimate the integral precisely at time $-n+1$,
again using Proposition~\ref{cor:exp},
\begin{align*}
&\int_{W \cap O_n} \cL^k1 \, dm_W  = \sum_{W_i \in \cG_{n-1}(W)} \int_{W_i \cap T^{-n+1}O_n} \cL^{k-n+1} 1 \, dm_{W_i} \\
&\quad\le \sum_{W_i \in \cG_{n-1}(W)} |\log |W_i \cap T^{-n+1} O_n||^{-\gamma} \| \cL^{k-n+1} 1 \|_s \\
& \quad\le \sum_{W_i \in \cG_{n-1}(W)} |\log (C \ve^{1/2} \hLambda^{-n/2})|^{-\gamma} C e^{(k-n+1)h_*} 
\le |\log (C \ve^{1/2} \hLambda^{-n/2})|^{-\gamma} C e^{k h_*} \, .
\end{align*}
Using the estimates of Cases I and II in \eqref{eq:O_n} and using the weaker bound, we see that,
\[
\int_W 1_{O_n} \, \nu_{\ell_j} \le C |\log (C \ve^{1/2} \hLambda^{-n/2})|^{-\gamma}\, .
\]
Summing over $n$, we have,
$\int_W 1_O \, \nu_{\ell_j} \le C' |\log \ve|^{1-\gamma}$, uniformly in $j$.  Since $\nu_{\ell_j}$ converges to $\nu$ 
in the weak norm, this bound carries over to $\nu$.  Since $\gamma>1$ and $\ve>0$ was arbitrary, 
this implies $\nu(O) = 0$, completing the proof of the lemma.
\end{proof}


\subsection{Absolute Continuity of $\mu_*$ --- Full Support.}
\label{notmixing}

{\it In this subsection, we assume throughout that $\gamma > 1$ (this is possible since we assumed $h_*>s_0 \log 2$ to construct $\mu_*$).}

Our proof of the Bernoulli property relies on showing first that $\mu_*$
is K-mixing (Proposition~\ref{prop:mixing}).
As a first step, we will prove that $\mu_*$ is ergodic (see the Hopf-type Lemma~\ref{lem:ergodic}). 
This will require establishing absolute continuity of the unstable foliation for $\mu_*$
(Corollary~ \ref{cor:abs cont}), which  will be deduced from the following
absolute continuity
 result for $\nu$:

\begin{proposition}
\label{prop:abs cont}
Let $R$ be a Cantor rectangle.
Fix $W^0 \in \cW^s(R)$ and for $W \in \cW^s(R)$, let $\Theta_W$ 
denote the holonomy map from $W^0 \cap R$ to $W \cap R$ along unstable manifolds in $\cW^u(R)$.
Then $\Theta_W$ is absolutely continuous with respect to the leafwise measure $\nu$.
\end{proposition}

\begin{proof}
Since by Lemma~\ref{lem:hyp} unstable manifolds comprise a set of full $\nu$-measure, it suffices to fix a set $E \subset W^0 \cap R$ with $\nu$-measure zero, and prove that the $\nu$-measure of $\Theta_W(E) \subset W$
is also zero.

Since $\nu$ is a regular measure on $W^0$, for $\ve > 0$, there exists an open set $O_\ve \subset W^0$, $O_\ve \supset E$, 
such that $\nu(O_\ve) \le \ve$.  Indeed, since $W^0$ is compact, we may choose $O_\ve$ to be a finite union of intervals.
Let $\psi_\ve$ be a smooth function which is 1 on $O_\ve$ and 0 outside of an $\ve$-neighbourhood of $O_\ve$.  We may choose
$\psi_\ve$ so that $| \psi_\ve|_{C^1(W^0)} \le 2 \ve^{-1}$.

Using \eqref{eq:C1 C0}, we choose $n = n(\ve)$ such that $| \psi_\ve \circ T^n |_{C^1(T^{-n}W^0)} \le 1$.
Note this implies in particular that $\Lambda^{-n} \le \ve$.
Following the procedure described at the beginning of Section~\ref{unstable norm}, we subdivide $T^{-n}W^0$ and 
$T^{-n}W$ into matched 
pieces $U^0_j$, $U_j$ and
unmatched pieces $V^0_i$, $V_i$.  With this construction, none of the unmatched pieces $T^nV^0_i$ intersect an unstable manifold
in $\cW^u(R)$ since unstable manifolds are not cut under $T^{-n}$.

Indeed, on matched pieces, we may choose a foliation 
$\Gamma_j = \{ \gamma_x \}_{x \in U^0_j}$ such that:

i) $T^n\Gamma_j$ contains all unstable manifolds in $\cW^u(R)$ that intersect $T^nU^0_j$;

ii) between unstable manifolds in $\Gamma_j \cap T^{-n}(\cW^u(R))$, we interpolate via unstable curves;

iii) the resulting holonomy $\Theta_j$ from $T^nU^0_j$ to $T^nU_j$ has uniformly bounded Jacobian\footnote{Indeed,
  \cite{bdl} shows the Jacobian is H\"older continuous, but we shall not need this here.} with
  respect to arc-length, with bound depending on the
  unstable diameter of $D(R)$, by \cite[Lemmas~6.6, 6.8]{bdl};  

iv)  pushing forward $\Gamma_j$ to $T^n\Gamma_j$ in $D(R)$, we interpolate in the gaps using unstable curves;
  call $\bGamma$ the resulting foliation of $D(R)$;

v)  the associated holonomy map $\bTheta_W$ extends $\Theta_W$ and has uniformly bounded Jacobian,
  again by \cite[Lemmas~6.6 and 6.8]{bdl}.

Using the map $\bTheta_W$, we define $\tpsi_\ve = \psi_\ve \circ \bTheta_W^{-1}$, and note that
$| \tpsi_\ve|_{C^1(W)} \le C |\psi_\ve |_{C^1(W^0)}$, where we write $C^1(W)$ for the set of Lipschitz functions on
$W$, i.e., $C^\alpha$ with $\alpha =1$.

Next, we modify $\psi_\ve$ and $\tpsi_\ve$ as follows:  We set them
equal to $0$ on the images of unmatched pieces, $T^nV^0_i$ and $T^nV_i$, respectively.  Since these curves do not intersect
unstable manifolds in $\cW^u(R)$, we still have $\psi_\ve = 1$ on $E$ and $\tpsi_\ve = 1$ on $\Theta_W(E)$.  Moreover, the
set of points on which $\psi_\ve > 0$ (resp. $\tpsi_\ve > 0$) is a finite union of open intervals that cover $E$ (resp. $\Theta_W(E)$).

Following Section~\ref{unstable norm}, we estimate
\begin{equation}
\label{eq:match split}
\begin{split}
\int_{W^0}  \psi_\ve \, \nu - \int_W \tpsi_\ve \, \nu & = e^{-n h_*} \left( \int_{W^0} \psi_\ve \, \cL^n \nu - \int_W \tpsi_\ve \, \cL^n \nu \right)  \\
& = e^{-n h_*} \sum_j \int_{U^0_j} \psi_\ve \circ T^n \, \nu - \int_{U_j} \phi_j \, \nu + \int_{U_j} (\phi_j - \tpsi_\ve \circ T^n) \, \nu \,  , 
\end{split}
\end{equation}
where $\phi_j = \psi_\ve \circ T^n \circ G_{U^0_j} \circ G_{U_j}^{-1}$, and $G_{U^0_j}$ and $G_{U_j}$ represent the functions
defining $U^0_j$ and $U_j$, respectively, defined as in \eqref{eq:match}.  
Next, since $d(\psi_\ve \circ T^n, \phi_j) = 0$
by construction, and using \eqref{eq:graph contract} and the assumption that $\Lambda^{-n} \le \ve$, we have by
\eqref{eq:second unstable},
\begin{equation}
\label{eq:match one}
e^{-n h_*} \left| \sum_j \int_{U^0_j} \psi_\ve \circ T^n \, \nu - \int_{U_j} \phi_j \, \nu \right| \le C 
|\log \ve|^{-\varsigma} \| \nu \|_u \, .
\end{equation}

It remains to estimate the last term in \eqref{eq:match split}.  This we do using the weak norm,
\begin{equation}
\label{eq:match two}
\int_{U_j} (\phi_j - \tpsi_\ve \circ T^n) \, \nu \le |\phi_j - \tpsi_\ve \circ T^n|_{C^\alpha(U_j)} \, |\nu|_w \, .
\end{equation}
By \eqref{eq:diff}, we have
\[
|\phi_j - \tpsi_\ve \circ T^n|_{C^\alpha(U_j)} \le C |\psi_\ve \circ T^n \circ G_{U^0_j} - \tpsi_\ve \circ T^n \circ G_{U_j} |_{C^\alpha(I_j)} \, ,
\]
where $I_j$ is the common $r$-interval on which $G_{U^0_j}$ an $G_{U_j}$ are defined.

Fix $r \in I_j$, and let $x = G_{U^0_j}(r) \in U_j$ and $\bx = G_{U_j}(r)$.  Since $U^0_j$ and $U_j$ are matched, there
exists $y \in U^0_j$ and an unstable curve $\gamma_y \in \Gamma_j$ such that $\gamma_y \cap U_j = \bx$.  By definition
of $\tpsi_\ve$, we have $\tpsi_\ve \circ T^n(\bx) = \psi_\ve \circ T^n(y)$.  Thus,
\[
\begin{split}
|\psi_\ve \circ T^n \circ G_{U^0_j}(r) &- \tpsi_\ve \circ T^n \circ G_{U_j} (r)|\\
& \le  |\psi_\ve \circ T^n (x) - \psi_\ve \circ T^n(y)| + |\psi_\ve \circ T^n(y) - \tpsi_\ve \circ T^n(\bx)| \\
& \le  |\psi_\ve \circ T^n|_{C^1(U^0_j)} d(x,y)  \le C \Lambda^{-n} \le C \ve \, ,
\end{split}
\] 
where we have used the fact that $d(x,y) \le C\Lambda^{-n}$ due to the uniform transversality of stable and unstable curves.

Now given $r, s \in I_j$, we have on the one hand,
\[
|\psi_\ve \circ T^n \circ G_{U^0_j}(r) - \tpsi_\ve \circ T^n \circ G_{U_j} (r) -
\psi_\ve \circ T^n \circ G_{U^0_j}(s) + \tpsi_\ve \circ T^n \circ G_{U_j} (s)|
 \le 2C\ve \, ,
\]
while on the other hand,
\begin{align*}
|\psi_\ve \circ T^n \circ G_{U^0_j}(r)& - \tpsi_\ve \circ T^n \circ G_{U_j} (r) -
 \psi_\ve \circ T^n \circ G_{U^0_j}(s) + \tpsi_\ve \circ T^n \circ G_{U_j} (s)| \\
& \qquad\qquad\qquad  \le (|\psi_\ve \circ T^n|_{C^1(U^0_j)} + |\tpsi_\ve \circ T^n|_{C^1(U_j)} ) C |r-s| \, ,
\end{align*}
where we have used the fact that $G_{U^0_j}^{-1}$ and $G_{U_j}^{-1}$ have bounded derivatives since the
stable cone is bounded away from the vertical.

The difference is bounded by the minimum of these two expressions.  This is 
greatest
when the two are equal, i.e.,
when $|r-s| = C \ve$.  Thus $H^\alpha(\psi_\ve \circ T^n \circ G_{U^0_j} - \tpsi_\ve \circ T^n \circ G_{U_j}) \le C \ve^{1-\alpha}$,
and so
$
|\phi_j - \tpsi_\ve \circ T^n|_{C^\alpha(U_j)} \le C \ve^{1-\alpha} 
$.
Putting this estimate together with \eqref{eq:match one} and \eqref{eq:match two} in \eqref{eq:match split}, we conclude,
\begin{equation}\label{eq:weak small2}
\left| \int_{W^0}  \psi_\ve \, \nu - \int_W \tpsi_\ve \, \nu \right| \le C |\log \ve|^{-\varsigma} \| \nu \|_u + C \ve^{1-\alpha} |\nu|_w \, .
\end{equation}
Now since $\int_{W^0} \psi_\ve \, \nu \le 2\ve$, we have
\begin{equation}
\label{eq:weak small}
\int_W \tpsi_\ve \, \nu \le C' |\log \ve|^{-\varsigma}\,  ,
\end{equation}
where $C'$ depends on $\nu$.  Since $\tpsi_\ve = 1$ on $\Theta_W(E)$ and $\tpsi_\ve > 0$ on an open set containing 
$\Theta_W(E)$ for every $\ve > 0$, we have $\nu(\Theta_W(E)) = 0$, as required.
\end{proof}

We next state our main absolute continuity result:

\begin{cor}
[Absolute Continuity of $\mu_*$ with Respect to Unstable Foliations]
\label{cor:abs cont}
Let $R$ be a Cantor rectangle with $\mu_*(R) > 0$.  Fix $W^0 \in \cW^s(R)$ and for $W \in \cW^s(R)$, let $\Theta_W$ 
denote the holonomy map from $W^0 \cap R$ to $W \cap R$ along unstable manifolds in $\cW^u(R)$.
Then $\Theta_W$ is absolutely continuous with respect to the measure $\mu_*$.
\end{cor}

To deduce  the corollary from Proposition~\ref{prop:abs cont}, we shall
introduce a set $M^{reg}$ of regular points and
a countable cover of this set by Cantor rectangles.
The set $M^{reg}$ is defined by
\[
M^{reg} = \{ x \in M : d(x, \partial W^s(x))>0 \, , \, \, \, d(x, \partial W^u(x)) > 0 \}\,  .
\]
At each $x \in M^{reg}$, by \cite[Prop~7.81]{chernov book}, we construct a (closed) locally maximal\footnote{Recall that, as in Section~\ref{supera}, by locally maximal we mean
that $y \in R_x$ if and only if $y \in D(R_x)$ and $y$ has stable and unstable manifolds that
completely cross $D(R_x)$.}
Cantor rectangle $R_x$, containing $x$,
which is the direct product of local stable and unstable manifolds  (recall Section~\ref{supera}).
By trimming the sides, we may arrange it so that
$
\frac12 \diam^s(R_x) \le \diam^u(R_x) \le 2 \diam^s(R_x) 
$.

\begin{lemma}[Countable Cover  of  $M^{reg}$ by Cantor Rectangles]
\label{lem:cover}
There exists a countable set $\{ x_j \}_{j \in \mathbb{N}} \subset M^{reg}$, such that 
$\cup_{j} R_{x_j} = M^{reg}$ and each $R_j:=R_{x_j}$ satisfies \eqref{**}.
\end{lemma}

\begin{proof}
Let $n_\delta \in \mathbb{N}$ be such that $1/n_\delta \le \delta_0$.  As already mentioned,
in the proof of Proposition~\ref{prop:good growth}, for each $n \ge n_\delta$, by \cite[Lemma~7.87]{chernov book},
there exists a finite number of $R_x$ such that any stable manifold of length at least $1/n$ properly crosses one of the
$R_x$ (see Section~\ref{supera} for the definition of proper crossing, recalling that
each $R_x$ must satisfy \eqref{**}).  This fact follows from the compactness of the set
of stable curves in the Hausdorff metric.  Call this finite set of rectangles $\{ R_{n, i} \}_{i \in \tilde{I}_n}$.

Fix $y \in M^{reg}$ and define $\epsilon = \min \{ d(y, \partial W^s(y)), d(y, \partial W^u(y) \} > 0$.
Choose $n \ge n_\delta$ such that $2/n < \epsilon$.  By construction, there exists $i \in \tilde{I}_n$ such that
$W^s(y)$ properly crosses $R_{n,i}$.  Now $\diam^s(R_{n,i}) \le 1/n$, which implies 
$\diam^u(R_{n.i}) \le 2/n < \epsilon$.  Thus $W^u(y)$ crosses $R_{n,i}$ as well.  By maximality, $y \in R_{n,i}$.
\end{proof}

Let $\{ R_{n,i} : n \ge n_\delta, i \in \tilde{I}_n \}$ be the Cantor rectangles
constructed in the proof of Lemma~\ref{lem:cover}.  Since $\mu_*(M^{reg}) = 1$, by discarding any $R_{n,i}$ of zero measure, we
obtain a countable collection of Cantor rectangles 
\begin{equation}\label{fakecover}
\{R_j\}_{j \in \mathbb{N}}:=\{ R_{n,i} : n \ge n_\delta, i \in I_n \}
\end{equation} 
such that 
$\mu_*(R_{j}) > 0$ for all $j$
and $\mu_*( \cup_{j} R_{j}) =1$.
In the rest of the paper we shall work with this countable collection of rectangles.

Given a Cantor rectangle $R$, define $\cW^s(R)$ to be the set of stable manifolds that completely cross $D(R)$, and similarly for
$\cW^u(R)$.

\begin{proof}[Proof of Corollary~\ref{cor:abs cont}]
 In order to prove absolute continuity of the unstable foliation with respect
to $\mu_*$, we will show that the conditional measures $\mu_*^W$ of $\mu_*$ are equivalent
to $\nu$ on $\mu_*$-almost every $W \in \cW^s(R)$. 

Fix a  Cantor rectangle $R$ satisfying \eqref{**}
with $\mu_*(R)>0$, and $W^0$ as in the statement of the corollary.  
 Let $E \subset W^0 \cap R$ satisfy $\nu(E) = 0$, for the leafwise measure $\nu$.

For any 
$W \in \cW^s(R)$, we have the holonomy map $\Theta_W: W^0 \cap R \to W \cap R$ as in the proof of 
Proposition~\ref{prop:abs cont}.  For $\ve > 0$, we approximate $E$, choose $n$ and construct a foliation $\bGamma$ of the
solid rectangle $D(R)$ as before.
Define $\psi_\ve$ and use the foliation $\bGamma$ to define
$\tpsi_\ve$ on $D(R)$.  We have $\tpsi_\ve = 1$ on $\bar E = \cup_{x \in E} \bar\gamma_x$, where $\bar\gamma_x$ is the
element of $\bGamma$ containing $x$.
We extend $\tpsi_\ve$ to $M$ by setting it equal to $0$ on $M \setminus D(R)$.

It follows from the proof of Proposition~\ref{prop:abs cont}, in particular \eqref{eq:weak small},
that $\tpsi_\ve \nu \in \cB_w$, and
$
|\tpsi_\ve \nu|_w \le C' |\log \ve|^{-\varsigma} 
$.
Now,
\begin{equation}
\label{eq:limit mu}
\begin{split}
\mu_*(\tpsi_\ve) & = \tilde\nu(\tpsi_\ve\nu) = \lim_{j \to \infty} \frac{1}{n_j} \sum_{k=0}^{n_j-1} e^{-k h_*} (\cL^*)^kd\musrb(\tpsi_\ve \nu) \\
& = \lim_{j \to \infty} \frac{1}{n_j} \sum_{k=0}^{n_j-1} e^{-k h_*} \musrb(\cL^k(\tpsi_\ve \nu)) \, .
\end{split}
\end{equation}
For each $k$, using the disintegration of $\musrb$ as in the proof of Lemma~\ref{lem:disint} with the same notation
as there, we estimate, 
\begin{align*}
\musrb(\cL^k(\tpsi_\ve \nu)) & = \int_{\Xi} \int_{W_\xi}  \cL^k(\tpsi_\ve \nu) \, \rho_\xi \, dm_{W_\xi} \, |W_\xi|^{-1} \, d{\hatmusrb}(\xi) \\
& \le C \int_{\Xi} |\cL^k(\tpsi_\ve \nu) |_w \, |W_\xi|^{-1} \, d{\hatmusrb}(\xi)\\
&\le C e^{k h_*} |\tpsi_\ve \nu|_w \le C e^{k h_*} |\log \ve|^{-\varsigma} \, ,
\end{align*}
where we have used \eqref{eq:weak ly} in the last line.
Thus $\mu_*(\tpsi_\ve) \le C |\log \ve|^{-\varsigma}$, for each $\ve >0$, so that $\mu_*(\bar E) = 0$.  

Disintegrating
$\mu_*$ into conditional measures $\mu_*^{W_\xi}$ on $W_\xi \in \cW^s$ and a factor measure $d\hat\mu_*(\xi)$ on the 
index set $\Xi_R$ of stable manifolds in $\cW^s(R)$, it follows that $\mu_*^{W_\xi}(\bar E) = 0$ for $\hat\mu_*$-almost every 
$\xi \in \Xi_R$.  
Since $E$ was arbitrary, 
the conditional measures of $\mu_*$ on $\cW^s(R)$ are absolutely continuous with respect to the leafwise measure
$\nu$.

To show that in fact $\mu_*^W$ is equivalent to $\nu$, suppose now that $E \subset W^0$
has $\nu(E)>0$.  For any $\ve >0$ such that $C'|\log \ve|^{-\varsigma} < \nu(E)/2$, where
$C'$ is from \eqref{eq:weak small}, choose $\psi_\ve \in C^1(W^0)$ such that
$\nu(|\psi_\ve - 1_E|) < \ve$, where $1_E$ is the indicator function of the set $E$. 
As above, we extend $\psi_\ve$ to a function $\tpsi_\ve$
on $D(R)$ via the foliation $\bGamma$, and then to $M$ by setting $\tpsi_\ve = 0$
on $M \setminus D(R)$.

We have $\tpsi_\ve \nu \in \cB_w$ and by \eqref{eq:weak small2}
\begin{equation}
\label{eq:lower psi}
\nu(\tpsi_\ve \, 1_W) \ge \nu(\psi_\ve \, 1_{W^0}) - C'|\log \ve|^{-\varsigma} ,
\qquad \mbox{for all } W \in \cW^s(R) \, .
\end{equation}

Now following \eqref{eq:limit mu} and disintegrating $\musrb$ as usual, we obtain,
\begin{equation}
\label{eq:counting}
\begin{split}
\mu_*(\tpsi_\ve) 
& = \lim_n \frac 1n \sum_{k=0}^{n-1} e^{-k h_*} \int_{\Xi} \int_{W_\xi} \cL^k(\tpsi_\ve \nu)
\, \rho_\xi \, dm_{W_\xi} \, d\hatmusrb(\xi) \\
& = \lim_n \frac 1n \sum_{k=0}^{n-1} e^{-k h_*} \int_\Xi \left( \sum_{W_{\xi, i} \in \cG_k(W_\xi)}  \int_{W_{\xi, i}} \tpsi_\ve
\, \rho_\xi \circ T^k \, \nu \right) \, d\hatmusrb(\xi) \,  .
\end{split}
\end{equation}
To estimate this last expression, we estimate the cardinality of the curves
$W_{\xi, i}$ which properly cross the rectangle $R$.

By Corollary~\ref{cor:short} and the choice of $\delta_1$ in \eqref{eq:delta1}, there exists $k_0$, depending only on the minimum length
of $W \in \cW^s(R)$, such that   
$\# L_k^{\delta_1}(W_\xi) \ge \frac 13 \# \cG_k(W_\xi)$ for all $k \ge k_0$.

By choice of our covering $\{ R_i \}$ from \eqref{fakecover}, 
 all $W_{\xi, j} \in L_k^{\delta_1}(W_\xi)$ properly cross one of finitely
many $R_i$.  By the topological mixing property of $T$, there exists $n_0$, depending
only on the length scale $\delta_1$, such that some smooth component of $T^{-n_0}W_{\xi, j}$
properly crosses $R$.  Thus, letting $\cC_k(W_\xi)$ denote those
$W_{\xi, i} \in \cG_k(W_\xi)$ which properly cross $R$, we have
\[
\# \cC_k(W_\xi) \ge \# L_{k-n_0}^{\delta_1}(W_\xi) \ge \tfrac 13 \# \cG_{k-n_0}(W_\xi) 
\ge \tfrac 13 c e^{(k-n_0)h_*}\, ,
\]
for all $k \ge k_0 + n_0$, where $c>0$ depends on $c_0$ from Proposition~\ref{prop:good growth}
as well as the minimum length of $W \in \cW^s(R)$. 

Using this lower bound on the cardinality together with \eqref{eq:lower psi} yields,
\[
\mu_*(\tpsi_\ve) \ge \tfrac 13 c e^{-n_0 h_*} \big( \nu(\psi_\ve) - C' |\log \ve|^{-\varsigma}\big)
\ge C'' \big( \nu(E) - |\log \ve|^{-\varsigma}\big)\,  .
\]
Taking $\ve \to 0$, we have $\mu_*(\bar E) \ge C'' \nu(E)$, and so $\mu_*^W(\bar E) > 0$
for almost every
$W \in \cW^s(R)$.
\end{proof}

A consequence of the proof of Corollary~\ref{cor:abs cont} is the positivity of $\mu_*$ on open sets.

\begin{proposition}[Full Support]\label{lastitem}
We have $\mu_*(O) > 0$ for any open set $O$. 
\end{proposition}

\begin{proof}
Suppose $R$ is a  Cantor rectangle with index set of stable leaves $\Xi_R$.
We call $I \subset \Xi_R$ an interval if
$a, b \in I$ implies that $c \in I$ for all $c \in \Xi_R$ such that $W_c$ lies between
$W_a$ and $W_b$.\footnote{Notice that if $I \subset \Xi_j$ is an interval such that 
$\hatmusrb(I)>0$, then $\cup_{\xi \in I} W_\xi \cap R_j$ is a 
Cantor rectangle which contains
a subset satisfying the high density condition \eqref{**}, so we can talk about proper crossings.} 
It follows from the proof of Corollary~\ref{cor:abs cont} that for any
interval $I \subset \Xi_R$ such that 
$\hatmusrb(I)>0$, then $\mu_*(\cup_{\xi \in I}W_\xi) > 0$.  Indeed, by Lemma~\ref{lem:disint},
$\hat\nu$ is equivalent to $\hatmusrb$ (since $\nu(W)>0$ for all $W \in \cW^s$, when $\nu$
is viewed as a leafwise measure), so that
$\hatmusrb(I)>0$ implies $\hat\nu(I) >0$.  
Then by Lemma~\ref{lem:hyp} there exists a Cantor rectangle $R'$ with
$D(R')\subset D(R)$ and $\Xi_{R'}\subset I$ such that $\nu(R')>0$.
Then we simply apply \eqref{eq:counting} and the
argument following it with $\tpsi_\ve$ replaced by the characteristic 
function of $\cup_{\xi \in \Xi_{R'}}W_\xi$.

Then if $O$ is an open set in $M$, it contains a Cantor rectangle $R$ such that 
$D(R) \subset O$ and $\musrb(R)>0$.  It follows that
$\hatmusrb(\Xi_R)>0$, and so $\mu_*(\cup_{\xi \in \Xi_R}W_\xi) > 0$.
\end{proof}

\subsection{Bounds on Dynamical Bowen Balls --- Comparing $\mu_*$ and $\musrb$}
\label{BoBa}
In this section we show upper and lower bounds on the $\mu_*$-measure
of dynamical Bowen balls, from which we establish a necessary condition for $\mu_*$ and $\musrb$
to coincide. (The lower bound will use results from Section~\ref{notmixing}.)

For $\epsilon >0$ and $x\in M$, we denote by $B_n(x,\epsilon)$ the dynamical (Bowen) ball at $x$ of
length $n\ge 1$ for
 $T^{-1}$, i.e.,
$$
B_n(x,\epsilon)=\{ y \in M\mid d(T^{-j} (y), T^{-j} (x)) \le \epsilon \, , \, \, \forall \, 0 \le j \le n\} .
$$

For $\eta, \delta > 0$ and $p \in (1/\gamma, 1]$, let $M^{reg}(\eta, p, \delta)$ denote those
$x \in M^{reg}$ such that $d(T^{-n}x, \cS_{-1}) \ge \delta e^{- \eta n^p}$.  It follows from 
Lemma~\ref{lem:approach} that $\mu_*(\cup_{\delta > 0}M^{reg}(\eta, p, \delta)) = 1$.

\begin{proposition}[Topological Entropy and Measure of Dynamical Balls]
\label{prop:max}  Assume that $h_*>s_0 \log 2$.
There exists $A<\infty$ such that for all
$\epsilon >0$ sufficiently small,
$x\in M$, and $n \ge 1$, the measure 
$\mu_*$ constructed in \eqref{defm} 
satisfies
\begin{equation}
\label{eq:upper ball}
\mu_*(B_n(x,\epsilon)) \le \mu_*(\overline{B_n(x,\epsilon) })
\le A e^{-n h_*} \, .
\end{equation}
Moreover, for all $\eta, \delta >0$ and $p \in (1/\gamma, 1]$, for each $x \in M^{reg}(\eta, p, \delta)$, 
and all $\ve > 0$ sufficiently small, there exists
$C(x, \epsilon, \eta, p, \delta)>0$ such that for all $n \ge 1$, 
\begin{equation}
\label{eq:lower ball}
C(x, \epsilon, \eta, p, \delta) \, e^{-n h_* - \eta h_* \bar{C}_2 n^p} \le \mu_*(B_n(x,\epsilon)),
\end{equation}
where $\bar{C}_2>0$ is the constant from the proof of Corollary~\ref{cor:short}. 
\end{proposition}

\begin{proof}
Assume $\gamma >1$.
Fix $\epsilon > 0$ such that $\epsilon \le \min \{ \delta_0, \ve_0 \}$, where $\ve_0$ is from the proof of Lemma~\ref{lem:hsep}.
For $x \in M$ and $n \ge 0$, define $1^B_{n,\epsilon}$  to be the indicator function of the
dynamical ball $B_n(x,\epsilon)$.

Since $\nu$ is attained as the (averaged) limit of $\cL^n 1$ in the weak norm and  since we have
$\int_W(\cL^n 1)\, \psi dm_W \ge 0$ whenever $\psi \ge 0$, it follows that, viewing $\nu$
as a  leafwise distribution,
\begin{equation}
\label{eq:pos}
\int_W \psi \, \nu  \ge 0, \quad \mbox{ for all $\psi \ge 0$.}
\end{equation}
Then the inequality $|\int_W \psi \,\nu| \le \int_W |\psi| \, \nu$ implies that the supremum
in the weak norm can be obtained by restricting to $\psi \ge 0$. 

Let $W \in \cW^s$ be a curve intersecting $B_n(x, \epsilon)$, and let $\psi \in C^\alpha(W)$
satisfy $\psi \ge 0$ and $|\psi|_{C^\alpha(W)} \le 1$.  Then, since $\cL \nu = e^{h_*} \nu$, we have
\begin{equation}
\label{eq:unwrap}
\int_W \psi \, 1^B_{n , \epsilon} \, \nu = \int_W \psi \, 1^B_{n, \epsilon} \, e^{-n h_*} \cL^n \nu 
 = e^{-n h_*} \sum_{W_i \in \cG_n(W)} \int_{W_i} \psi \circ T^n \, 1^B_{n,\epsilon} \circ T^n \, \nu \, .
\end{equation}

We claim that $1^B_{n,\ve} \nu \in \cB_w$ (and indeed in $\cB$).  To see this, note that
\[
1^B_{n, \ve} = \prod_{j=0}^n 1_{\cN_\ve(T^{-j}x)} \circ T^{-j}
= \prod_{j=0}^n \cL_{\mbox{\tiny{SRB}}}^j (1_{\cN_\ve(T^{-j}x)}) \, ,
 \]
 where, as in Section~\ref{sec:transfer}, $\cL_{\mbox{\tiny{SRB}}}$ denotes the transfer operator
 with respect to $\musrb$.  Since $\cL_{\mbox{\tiny{SRB}}}$ preserves $\cB$ and $\cB_w$
 (\cite[Lemma~3.6]{demzhang14}), it suffices to show that $1_{\cN_\ve(T^{-j}x)}$ satisfies the
 assumptions of \cite[Lemma~5.3]{demzhang14}.  This follows from the fact that
 $\partial \cN_{\ve}(T^{-j}x)$ comprises a single circular arc, possibly together with a segment
 of $\cS_0$, which satisfies the weak transversality condition of that lemma with $t_0 = 1/2$. 
 Then applying \cite[Lemma~5.3]{demzhang14} successively for each $j$ yields the claim.

In the proof of Lemma~\ref{lem:hsep}, it was shown that if $x, y$ lie in different elements of
$\cM_0^n$, then $d_n(x,y) \ge \ve_0$, where $d_n(\cdot, \cdot)$ is the dynamical distance 
defined in \eqref{eq:d_n}.  Since $B_n(x, \epsilon)$ is defined with respect to $T^{-1}$, we will use
the time reversal counterpart of this property.  Thus since $\epsilon < \ve_0$, we conclude that
$B_n(x,\epsilon)$ is contained in a single component of $\cM_{-n}^0$, i.e.,
$B_n(x, \epsilon) \cap \cS_{-n} = \emptyset$, so that $T^{-n}$ is a diffeomorphism of
$B_n(x, \epsilon)$ onto its image.  Note that $1^{B}_{n, \epsilon} \circ T^n = 1_{T^{-n}(B_n(x,\epsilon))}$
and that $T^{-n}(B_n(x, \epsilon))$ is contained in a single component of $\cM_0^n$,  denoted $A_{n, \epsilon}$.

It follows that for each $W_i \in \cG_n(W)$ we have $W_i \cap A_{n, \epsilon} = W_i$. 
By \eqref{eq:pos}, we have
$$\int_{W_i} (\psi \circ T^n )\, 1_{T^{-n}(B_n(x, \epsilon))} \, \nu \le \int_{W_i} \psi \circ T^n \, \nu\, .
$$
Moreover, there can be at most
two $W_i \in \cG_n(W)$  having nonempty intersection with $T^{-n}(B_n(x, \epsilon))$.
This follows from the facts that $\epsilon \le \delta_0$, and that, in the absence of any cuts due to 
singularities, the only subdivisions occur when a curve has grown to length longer than 
$\delta_0$ and is subdivided into two curves of length at least $\delta_0/2$.

Using these facts together with \eqref{eq:C1 C0}, we sum over $W_i' \in \cG_n(W)$ such that
$W_i' \cap T^{-n}(B_n(x, \epsilon)) \neq 0$, to obtain
\[
\int_W \psi \, 1^B_{n,\epsilon} \, \nu \le e^{-n h_*} \sum_i \int_{W_i'} \psi \circ T^n \, \nu
\le 2C e^{-n h_*} |\nu|_w \, .
\]
This implies  that  $| 1^B_{n, \ve} \nu |_w \le 2C e^{-n h_*} |\nu|_w$.
Applying \eqref{eq:bweak}, implies \eqref{eq:upper ball}.

Next we prove \eqref{eq:lower ball}. 
Fix $\eta, \delta > 0$ with $e^\eta < \Lambda$ and $p \in (1/\gamma, 1]$, and
let $x \in M^{reg}(\eta, p, \delta)$.  
By \cite[Lemma~4.67]{chernov book} the length of the local stable
manifold containing $x$ is at least $\delta C_1$, where $C_1$ is from \eqref{eq:hyp}.
So by \cite[Lemma~7.87]{chernov book}, there exists a Cantor rectangle
$R_x$ containing $x$ such that $\musrb(R_x)>0$ and
whose diameter depends only on the length scale $\delta C_1$.  
By the proof of Proposition~\ref{lastitem}, we also have $\mu_*(R_x) > 0$.  In particular,
$\hat\mu_*(\Xi_{R_x}) = c_x > 0$, where $\Xi_{R_x}$ is the index set of stable manifolds comprising
$R_x$.
Let $\delta'>0$ denote the minimum length of $W_\xi \cap D(R_x)$ for
$\xi \in \Xi_{R_x}$, where $D(R_x)$ is the smallest solid rectangle containing $R_x$,
as in Definition~\ref{MCR}.

Choose $\epsilon >0$ such that 
$\epsilon \le \min \{ \delta_0, \ve_0,  \delta', \delta \}$.
As above, we note that $B_n(x,\epsilon)$ is contained in a single component of $\cM_{-n}^0$, and
thus $T^{-n}(B_n(x,\epsilon))$ is contained in a single component of $\cM_0^n$.  Moreover,
$T^{-n}$ is smooth on $W^u(x) \cap D(R_x)$.
Now suppose $y \in W^u(x) \cap R_x$.  Then since $x \in M^{reg}(\eta, p , \delta)$,
\[
d(T^{-n}y, \cS_{-1}) \ge d(T^{-n}x, \cS_{-1}) - d(T^{-n}y, T^{-n}x) 
\ge \delta e^{-\eta n^p} - C_1 \Lambda^{-n} \ge \tfrac{\delta}{2} e^{-\eta n^p}\, ,
\]   
for $n$ sufficiently large.
It follows that for each $\xi \in \Xi_{R_x}$, there exists $W_{\xi,i} \in \cG_n(W_\xi)$ such that 
$W'_{\xi, i} = W_{\xi, i} \cap T^{-n}(B_n(x, \epsilon))$ is a single curve and 
$|W'_{\xi, i}| \ge \min\{ \frac{\delta}{2} e^{-\eta n^p}, \epsilon \} \ge \frac{\epsilon}{2} e^{-\eta n^p}$.  
Thus recalling \eqref{eq:lower nu} and following
\eqref{eq:unwrap} with $\psi \equiv 1$,
\[
\int_{W_\xi} 1^B_{n , \epsilon} \, \nu \ge e^{-n h_*} \int_{W'_{\xi, i}} \nu 
\ge \bar{C} e^{-n h_*} |W'_{\xi, i}|^{h_* \bar{C}_2} \ge C' e^{-n h_* - \eta h_* \bar{C}_2 n^p} \, ,
\]
where $C'$ depends on $\epsilon$.

Finally, using the fact from the proof of Corollary~\ref{cor:abs cont} that $\mu_*^W$ is equivalent to 
$\nu$ on $\mu_*$-a.e. $W \in \cW^s$, we estimate,
\[
\begin{split}
\mu_*(B_n(x,\epsilon)) & \ge \mu_*(B_n(x, \epsilon ) \cap D(R_x)) =
\int_{\Xi_{R_x}} \mu_*^{W_\xi}(B_n(x, \epsilon)) \, d\hat\mu_*(\xi) \\
& \ge C \int_{\Xi_{R_x}} \nu(B_n(x, \epsilon) \cap W_\xi) \, d\hat\mu_*(\xi)
\ge C'' e^{-n h_* - \eta h_* \bar{C}_2 n^p} \hat\mu_*(\Xi_{R_x})\, .
\end{split}
\]
\end{proof}

Periodic points whose orbit do not have  grazing collisions belong to $M^{reg}$. We call them {\em regular}.

\begin{proposition}[$\mu_*$ and $\musrb$]
\label{prop:not eq}
Assume $h_*>s_0\log 2$.
If there exists a regular periodic point $x$ of period $p$ such that $\lambda_x = \frac{1}{p} \log |\det (DT^{-p}|_{E^s}(x))|\neq h_*$,
then $\mu_* \neq \musrb$.  
\end{proposition}

Although $h_*$ may not be known a priori, using Proposition~\ref{prop:not eq} it
suffices to find two regular periodic points $x$, $y$ 
such that $\lambda_x \neq \lambda_y$, to conclude that $\mu_* \neq \musrb$. 
(All known examples of dispersing billiard tables satisfy this condition.)

Proposition~\ref{prop:not eq} relies on the following lemma.

\begin{lemma}
\label{lem:srb}
Let $x \in M^{reg}$ be a regular periodic point. There exists $A > 0$ such that for all 
$\epsilon>0$ sufficiently
small, there exists $C(x,\epsilon) >0$ such that for all $n \ge 1$,
\[
C(x, \epsilon) e^{-n \lambda_x} \le \musrb(B_n(x, \epsilon)) \le A e^{- n \lambda_x} \, .
\]
\end{lemma}

\begin{proof}
Let $x$ be a regular periodic point for $T$ of period $p$.  For $\epsilon$ sufficiently small,
$T^{-i}(\cN_\epsilon(x))$ belongs to a single homogeneity strip for $i = 0, 1, \ldots, p$.  Thus if
$y \in B_n(x, \epsilon) \cap W^s(x)$, then the stable Jacobians $J^sT^n(x)$ and $J^sT^n(y)$ 
satisfy the bounded distortion estimate,
$|\log \frac{J^sT^n(x)}{J^sT^n(y)}| \le C_d d(x,y)^{1/3}$,
for a uniform $C_d>0$ \cite[Lemma~5.27]{chernov book}.  It follows that the conditional measure
on $W^s(x)$ satisfies
\begin{equation}
\label{eq:shrink}
C_x^{-1} \epsilon e^{-n \lambda_x} \le  \musrb^{W^s(x)}(B_n(x, \epsilon)) \le C_x \epsilon e^{-n \lambda x}\, ,
\end{equation}
for some $C_x \ge 1$, depending on the homogeneity strips in which the orbit
of $x$ lies.

Next, using again \cite[Prop~7.81]{chernov book}, we can find a Cantor rectangle 
$R_x \subset \cN_\epsilon(x)$ with diameter at most $\ve/(2C_1)$ and
$\musrb(R_x) \ge C \musrb(\cN_\epsilon(x))/(2C_1)^2$, for a constant $C>0$ depending on the distortion
of the measure.  Note that $W^u(x) \cap D(R_x)$ is never cut by $\cS_{-n}$ and lies in 
$B_n(x, \epsilon)$ by \eqref{eq:hyp}.  Thus each $W \in \cW^s(R_x)$ has a component in
$B_n(x, \ve)$ and this component has length satisfying the same bounds as
\eqref{eq:shrink}.  Integrating over $B_n(x,\epsilon)$ as in the proof of Proposition~\ref{prop:max}
proves the lemma.  An inspection of the
proof shows that the constant 
in the upper bound
can be chosen independently of $x$ when $\epsilon$ is sufficiently small, while the constant
in the lower bound cannot.
\end{proof}

\begin{proof}[Proof of Proposition~\ref{prop:not eq}]
If $x$ is a regular periodic point, then the upper and lower bounds
on $\mu_*(B_n(x,\epsilon))$ from Proposition~\ref{prop:max} hold with\footnote{Here, it is convenient to have the role of $\eta$ explicit in \eqref{eq:lower ball}.}  $\eta = 0$
for $\epsilon$ sufficiently small.  If $\lambda_x \neq h_*$,
these do not match the exponential rate in the bounds on $\musrb(B_n(x,\epsilon))$ from
Lemma~\ref{lem:srb}.  Thus for $n$ sufficiently large, 
$\mu_*(B_n(x,\epsilon)) \ne \musrb(B_n(x,\epsilon))$.
\end{proof}


\subsection{K-mixing and Maximal Entropy  of $\mu_*$ --- Bowen--Pesin--Pitskel Theorem~\ref{PP}}
\label{mixing}

In this section we use the absolute continuity results from Section~\ref{notmixing}
to establish K-mixing of $\mu_*$. We also show that $\mu_*$ has maximal
entropy, exploiting the upper bound from Section~\ref{BoBa}. Finally, we show that
$h_*$ coincides with the Bowen--Pesin--Pitskel entropy.

\begin{lemma}[Single Ergodic Component]
\label{lem:ergodic}
If $R$ is a Cantor rectangle with $\mu_*(R)>0$, then the set of stable manifolds $\cW^s(R)$ belongs to a 
single ergodic component of $\mu_*$.  
\end{lemma}

\begin{proof}
We follow the well-known Hopf strategy outlined in \cite[Section~6.4]{chernov book}
of smooth ergodic theory to show that $\mu_*$-almost every stable and unstable manifold has a full measure set of points belonging to a single ergodic component:  Given a continuous function $\vf$ on $M$, let $\bvf_+$, $\bvf_-$ denote the forward and
backward ergodic averages of $\vf$, respectively.  Let $M_\vf = \{ x \in M^{reg} : \bvf_+(x) = \bvf_-(x) \}$. When the two functions agree, denote their common value by $\bvf$.

Now fix a Cantor rectangle $R$ with $\mu_*(R)>0$.
 By Corollary~\ref{integral}, if
$\gamma >1$ then $\mu_*(M^{reg})=1$.
So, by the Birkhoff ergodic theorem, $\mu_*(M_\vf)=1$.  Thus for $\mu_*$ almost every 
$W \in \cW^s(R)$, the conditional measure $\mu_*^W$ satisfies $\mu_*^W(M_\vf)=1$.  Due to the fact that
forward ergodic averages are the same for any two points in $W$, it follows that $\bvf$ is constant on $W \cap M_\vf$.
The analogous fact holds for unstable manifolds in $\cW^u(R)$.

Let 
$$G_\vf = \{ x \in M_\vf : \bvf \mbox{ is constant on a full measure subset of $W^u(x)$ and $W^s(x)$} \}
\, .
$$
Clearly, $\mu_*(G_\vf)=1$, so the same facts apply to $G_\vf$ as $M_\vf$.

Let $W^0, W \in \cW^s(R)$ be stable manifolds with $\mu_*^{W0}(G_\vf) = \mu_*^W(G_\vf)=1$.  Let
$\Theta_W$ denote the holonomy map from $W^0 \cap R$ to $W \cap R$. 
By absolute continuity, Corollary~\ref{cor:abs cont}, $\mu_*^W(\Theta_W(W^0 \cap G_\vf)) > 0$.  Thus $\bvf$ is constant
for almost every point in $\Theta_W(W^0 \cap G_\vf)$.  Let $y$ be one such point and let $x = \Theta_W^{-1}(y)$.
Then since $x \in W^u(y) \cap G_\vf$,
\[
\bvf(x) = \bvf_-(x) = \bvf_-(y) = \bvf(y)\, , 
\]
so that the values of $\bvf$ on a positive measure set of points in $W^0$ and $W$ agree.  Since $\bvf$ is constant
on $G_\vf$, the values of $\bvf$ on a full measure set of points in $W$ and $W^0$ are equal.
Since this applies to any $W$ with $\mu_*^W(G_\vf) =1$, we conclude that $\bvf$ is constant almost everywhere
on the set $\cup_{W \in \cW^s(R)} W$.  Finally, since $\vf$ was an arbitrary continuous function, the set $\cW^s(R)$
belongs (mod 0) to a single ergodic component of $\mu_*$.
\end{proof}

We are now ready to prove  the K-mixing property of $\mu_*$.

\begin{proposition}
\label{prop:mixing}
$(T, \mu_*)$ is K-mixing.
\end{proposition}

\begin{proof}
We begin by showing that $(T^n, \mu_*)$ is ergodic for all $n \ge 1$.
Recall the countable set of (locally maximal)
Cantor rectangles 
$\{ R_i \}_{i \in \mathbb{N}}$ with $\mu_*(R_i) >0$, such that $\cup_i R_i = M^{reg}$ 
from \eqref{fakecover}.

We fix $n$ and let $R_1$ and $R_2$ be two such Cantor rectangles.  By Lemma~\ref{lem:ergodic}, $\cW^s(R_i)$ belongs
(mod 0) to a single ergodic component of $\mu_*$.  Since $T$ is topologically mixing, and using \cite[Lemma~7.90]{chernov book},
there exists $n_0 > 0$ such that for any $k \ge n_0$, a smooth component of $T^{-k}(D(R_1))$ properly crosses $D(R_2)$.
Let us call $D_k$ the part of this smooth component lying in $D(R_2)$.
  
Since the set of stable manifolds is invariant under $T^{-k}$, by the maximality of the set $\cW^s(R_2)$, we have that 
$T^{-k}(\cW^s(R_1)) \cap D_k \supseteq \cW^s(R_2) \cap D_k$.  And since this set of stable manifolds in $R_1$ has positive measure
with respect to $\hat{\mu}_*$, it follows that $\mu_*(T^{-k}(\cW^s(R_1)) \cap \cW^s(R_2)) > 0$.  Thus $R_1$ and $R_2$
belong to the same ergodic component of $T$.  Indeed, since we may choose $k = j n$ for some $j \in \mathbb{N}$, $R_1$
and $R_2$ belong to the same ergodic component of $T^n$.  Since this is true for each pair of Cantor rectangles $R_i$, $R_j$ 
in our countable collection, and $\mu_*(\cup_i R_i) =1$, we conclude that $T^n$ is ergodic.

We shall use  the Pinsker partition 
$$
\pi(T) = \bigvee \{ \xi : \xi \mbox{ finite partition of } M, h_{\mu_*}(T, \xi) = 0 \}\, .
$$
Since $T$ is an automorphism, the sigma-algebra generated by $\pi(T)$ is $T$-invariant.

Given two measurable partitions $\xi_1$ and $\xi_2$, the meet of the two partitions $\xi_1 \wedge \xi_2$ is defined
as the finest measurable partition with the property that $\xi_1 \wedge \xi_2 \le \xi_j$ for $j = 1, 2$.  
All definitions of measurable partitions and inequalities between them are taken to be mod 0, with respect to the measure
$\mu_*$.
It is a standard fact in ergodic theory (see e.g. \cite{sinai rokhlin}) that if $\xi$ is a partition of $M$ such that
(i) $T\xi \ge \xi$ and (ii) $\vee_{n=0}^\infty T^n \xi = \epsilon$, where $\epsilon$ is the partition of $M$ into points, then
$\wedge_{n=0}^\infty T^{-n}\xi \ge \pi(T)$ (mod 0).

Define $\xi^s$ to be the partition of $M$ into maximal local stable manifolds.  If $x \in M$ has no stable manifold or $x$ is
an endpoint of a stable manifold then define $\xi^s(x) = \{x\}$.  Similarly, define $\xi^u$ to be the partition of $M$ into
maximal local unstable manifolds.  Note that $\xi^s$ is a measurable partition of $M$ since it is
generated by the countable family of finite partitions given by the elements of $\cM_0^n$ and their closures.  Similarly,
$\cM_{-n}^0$ provides a countable generator for $\xi^u$.

It is a consequence of the uniform hyperbolicity of $T$ that $\xi^s$ satisfies (i) and (ii) above.  
Also, $\xi^u$ satisfies these conditions with respect to $T^{-1}$, i.e.,
$T^{-1}\xi^u \ge \xi^u$ and $\vee_{n=0}^\infty T^{-n}\xi^u = \epsilon$.  Thus $\wedge_{n = 0}^\infty T^n \xi^u \ge \pi(T)$.

Define $\eta_{\infty} = \wedge_{n=0}^\infty (T^n \xi^u \wedge T^{-n} \xi^s)$, and notice that $\eta_\infty \ge \pi(T)$ by the above.
Then since $\xi^s \wedge \xi^u \ge \eta_\infty$, we have $\xi^s \wedge \xi^u \ge \pi(T)$ as well.

We will show that each Cantor rectangle in our countable family belongs to one element
of $\xi^s \wedge \xi^u$ (mod 0).  This will follow from the product structure of each $R_i$ coupled with the 
absolute continuity of the holonomy map given by Corollary~\ref{cor:abs cont}.

For brevity, let us fix $i$ and set $R = R_i$.   We index the curves $W^s_\zeta \in \cW^s(R)$ by $\zeta \in Z$.
Define $\mu_R = \frac{\mu_*|_R}{\mu_*(R)}$.  We disintegrate the measure $\mu_R$ into a family of conditional 
probability measures
$\mu_{R}^{ W^s}$, $W^s \in \cW^s(R)$, and a factor measure $\hat\mu_R$ on the set $Z$.  Then
\[
\mu_R(A) = \int_{\zeta \in Z} \mu_{R}^{ W^s_\zeta}(A) \, d\hat\mu_R(\zeta), \quad \mbox{for all measurable sets $A$}\, .
\]

The set $R$ belongs to a single element of $\xi^s \wedge \xi^u$ if a full measure set of points can be connected by elements of
$\xi^s$ and $\xi^u$ even after the removal of a set of $\mu_*$-measure 0.  Let 
$N \subset M$ be such that $\mu_*(N) = 0$.  By the above disintegration, it follows that 
for $\hat\mu_R$-almost every $\zeta \in Z$, we have $\mu_{R}^{ W^s_\zeta}(N) = 0$.

Let $W^s_1$ and $W^s_2$ be two elements of $\cW^s(R)$ such that $\mu_{R}^{W^s_j}(N) = 0$, for $j=1, 2$.
For all $x \in W^s_1 \cap R$, $\xi^u(x)$ intersects $W^s_2$, and vice versa.  Let $\Theta$ denote the holonomy map from
$W^s_1$ to $W^s_2$.  Then by Corollary~\ref{cor:abs cont}, we have $\mu_{R} ^{ W^s_2}(\Theta(W^s_1 \cap N)) = 0$
and $\mu_{R}^{ W^s_1}(\Theta^{-1}(W^s_2 \cap N)) = 0$.  Thus the set $\Theta(W^s_1 \setminus N)$ has full measure
in $W^s_2$ and vice versa.
It folllows that $W^s_1$ and $W^s_2$ belong to one element of $\xi^s \wedge \xi^u$.  
This proves that $R$ belongs to a single
element of $\xi^s \wedge \xi^u$ (mod 0).

Since $\xi^s \wedge \xi^u \ge \pi(T)$, we have shown that
each $R_i$ belongs to a single element of $\pi(T)$, mod 0.  Since $\mu_*(R_i) > 0$ and 
$\mu_*(\cup_i R_i) = 1$, the ergodicity of $T$ and the invariance of $\pi(T)$ imply that $\pi(T)$ contains finitely
many elements, all having the same measure, whose union has full measure.  The action of $T$ is simply a permutation
of these elements.  Since $(T^n, \mu_*)$ is ergodic for all $n$, it follows that $\pi(T)$ is trivial.  Thus $(T, \mu_*)$
is K-mixing.
\end{proof}

Now that we know that $\mu_*$ is ergodic,
the upper bound in Proposition~\ref{prop:max} will easily\footnote{It is not much
harder to deduce this fact in the absence of ergodicity,  using only \eqref{eq:lower ball} 
with Theorem~\ref{thm:h_*}.} imply that $h_{\mu_*}(T) = h_*$:

\begin{cor}[Maximum Entropy]
\label{cor:max}
For $\mu_*$ defined as in \eqref{defm},  we have $h_{\mu_*}(T) = h_*$. 
\end{cor}

\begin{proof}
Since 
$\int |\log d(x, \cS_{\pm 1})| \, d\mu_* < \infty$ by Theorem~ \ref{nbhdthm}, and $\mu_*$
is ergodic, we may apply \cite[Prop~3.1]{DWY}\footnote{This is a slight generalization of the Brin-Katok local theorem \cite{brin},
using \cite[Lemma~2]{M}. 
Continuity of the map is not used in the proof of the theorem, and so it applies to our setting.} 
 to $T^{-1}$, which states that for $\mu_*$-almost every $x \in M$,
$$ \lim_{\epsilon \to 0} \liminf_{n \to \infty}
  - \tfrac 1n \log \mu_*(B_n(x, \epsilon)) =  \lim_{\epsilon \to 0} \limsup_{n \to \infty}
  - \tfrac 1n \log \mu_*(B_n(x, \epsilon)) = h_{\mu_*}(T^{-1})\, .
  $$
Using \eqref{eq:upper ball} and \eqref{eq:lower ball} with $p <1$, it follows that
$\lim_{n \to \infty}  - \tfrac 1n \log \mu_*(B_n(x, \epsilon)) = h_*$, for any $\ve>0$ sufficiently small. 
Thus $h_{\mu_*}(T) = h_{\mu_*}(T^{-1}) = h_*$.
\end{proof}

Corollary~\ref{cor:max} next allows us to prove Theorem~\ref{PP} about the
Bowen--Pesin--Pitskel entropy:

\begin{proof}[Proof of Theorem~\ref{PP}]
To show 
$
h_* \le \htop(T|_{M'})
$, we first use  Corollary~\ref{cor:max} and
the fact that $\mu_*(M')=1$ (since $\mu_*(\cS_n)=0$ for every $n$ by 
Theorem~\ref{nbhdthm})
to see that
$$
h_* =h_{\mu_*}(T)=\sup_{\mu: \mu(M')=1} h_\mu(T)\, .
$$
Then we apply the  bound  \cite[(A.2.1)]{Pesin} or
\cite[Thm 1]{PesinP} (by Remarks I and II there, $T$ need not be continuous on $M$) to get
$$
\sup_{\mu: \mu(M')=1} h_\mu(T)\le \htop(T |_{M'})\, .
$$

To show 
$
 \htop(T|_{M'})\le h_*
$, we  use that  \cite[(11.12)]{Pesin} implies\footnote{Just like in \cite[I and II]{PesinP}, it is essential that $M$ is compact, but
the fact that $T$ is not continuous on $M$ is irrelevant. Note also that \cite[(A.3'), p.~66]{Pesin} should be corrected, replacing
``any $\varepsilon > \epsilon>0$'' by ``any $\varepsilon > 1/m>0$.''}
$
\htop(T|_{M'})\le Ch_{M'}(T)
$, where  $Ch_{M'}(T)$ denotes the capacity topological entropy of the (invariant) set $M'$.  
Now, for any $\delta>0$, the elements of $\hP^{k}_{-k}=\cM^{k+1}_{-k-1}$ form an open
cover of $M'$ of diameter $<\delta$, if $k$ is large enough (see the proof of Lemma~\ref{lem:hsep}).
By adding finitely many open sets, we obtain  an open cover $\UU_\delta$ of $M$ of diameter $<\delta$.
Next \cite[(11.13)]{Pesin} gives that 
$$Ch_{M'}(T)=\lim_{\delta\to 0}\lim_{n\to \infty}\frac 1 n
\log \Lambda(M',\UU_\delta, n)\, ,
$$
where $\Lambda(M',\UU_\delta, n)$ is the smallest cardinality of a cover of $M'$
by elements of $\bigvee_{j=0}^ n T^{-j} \UU_\delta$.
Since for any $n\ge 1$, the sets of $\bigvee_{j=0}^ n T^{-j} \hP^{k}_{-k}$ form a cover of $M'$,
the second equality of  Lemma~\ref{prop:equiv} 
(i.e., $\lim_n \frac 1 n \log \# \hP^{k+n}_{-k}=h_*$)   implies that
$
Ch_{M'}(T)\le h_*
$.
\end{proof}


\subsection{Bernoulli Property of $\mu_*$}
\label{sec:bernoulli}

In this section, we prove that $\mu_*$ is Bernoulli by bootstrapping from K-mixing.
The key ingredients of the proof, in addition to K-mixing, are 
Cantor rectangles with a product structure of stable and unstable manifolds, the absolute
continuity of the unstable foliation with respect to $\mu_*$, and the bounds \eqref{nbhd}
on the neighbourhoods of the singularity sets.
First, we recall some
definitions, following Chernov--Haskell \cite{ChH} and the notion of 
very weak Bernoulli partitions introduced by Ornstein \cite{orn}.

Let $(X, \mu_X)$ and $(Y, \mu_Y)$ be two non-atomic Lebesgue probability spaces.
A {\em joining} $\lambda$ of the two spaces, is a measure on $X \times Y$ whose
marginals on $X$ and $Y$ are $\mu_X$ and $\mu_Y$, respectively.
Given finite partitions\footnote{As we shall not need the norms of $\cB$ and $\cB_w$ in this section,  we are free to
use the letters $\alpha$ and $\beta$ to denote partitions instead of real parameters.}
 $\alpha = \{ A_1, \ldots, A_k \}$ of $X$ and $\beta = \{ B_1, \ldots, B_k \}$ of 
$Y$, let $\alpha(x)$ denote the element
of $\alpha$ containing $x \in X$ (and similarly for $\beta$).  Moreover, if $x \in A_j$ and $y \in B_j$
for the same value of $j$ (which depends on the order in which the elements are listed), 
then we will write $\alpha(x) = \beta(y)$.  

The distance $\bd$ defined below considers two
partitions to be close if there is a joining $\lambda$ such that most of the measure
lies on the set of points $(x,y)$ with $\alpha(x) = \beta(y)$:
given two finite sequences of partitions $\{ \alpha_i \}_{i=1}^n$ of $X$ and $\{ \beta_i \}_{i=1}^n$
of $Y$, define
\[
\bd(\{\alpha_i\}, \{\beta_i\}) = \inf_{\lambda} \int_{X \times Y} h(x,y) \, d\lambda\, ,
\] 
where $\lambda$ is a joining of $X$ and $Y$ and $h$ is defined by
\[
h(x,y) = \frac 1n \#\{ i \in [1, \ldots, n] : \alpha_i(x) \neq \beta_i(y) \} \, .
\]

We will adopt the following notation:  If $E \subset X$, then $\alpha | E$ denotes the partition
$\alpha$ conditioned on $E$, i.e., the partition of $E$ given by elements of the form
$A \cap E$, for $A \in \alpha$.  Similarly, $\mu_X(\, \cdot \, | E)$ is the measure $\mu_X$ conditioned on $E$.
If a property holds for all atoms of $\alpha$ except for a collection whose union has measure
less than $\ve$, then we say the property holds for $\ve$-almost every atom of $\alpha$.

If $f : X \to X$ is an invertible, measure preserving transformation of $(X, \mu_X)$,
and $\alpha$ is a finite partition of $X$, then $\alpha$ is said to be 
{\em very weak Bernoullian} (vwB) if for all $\ve > 0$, there exists $N > 0$ such that
for every $n > 0$ and $N_0, N_1$ with $N < N_0 < N_1$, and for $\ve$-almost every atom $A$ of
$\bigvee_{N_0}^{N_1} f^i \alpha$, we have
\begin{equation}
\label{eq:vwB def}
\bd( \{ f^{-i}\alpha \}_{i=1}^n, \{ f^{-i}\alpha | A \}_{i=1}^n ) < \ve \, .
\end{equation}
The following theorem from \cite{orn weiss} provides the essential connection between the
Bernoulli property and vwB partitions.  (See also Theorems 4.1 and 4.2 in \cite{ChH}.)

\begin{theorem}
\label{thm:vwB}
If a partition $\alpha$ of $X$ is vwB, then 
$(X , \bigvee_{n=-\infty}^{\infty} f^{-n} \alpha, \mu_X, f)$ is a Bernoulli shift.
Moreover, if $\bigvee_{n=-\infty}^{\infty} f^{-n} \alpha$ generates the whole $\sigma$-algebra of $X$,
then $(X, \mu_X, f)$ is a Bernoulli shift.
\end{theorem}

We are ready to state and prove the main result of this section.

\begin{proposition}
\label{Bern} 
The measure $\mu_*$ is Bernoulli.
\end{proposition}

\begin{proof} 
First notice that since $f$ is measure preserving in \eqref{eq:vwB def}, then to prove that 
a partition $\alpha$ is vwB, it suffices to show that for every $\ve > 0$, there exist
integers $m$ and $N > 0$ such that for every $n, N_0, N_1$ with $N < N_0 < N_1$, and for
$\ve$-almost every atom $A$ of $\bigvee_{N_0-m}^{N_1-m} f^i\alpha$,
\begin{equation}
\label{eq:vwB alt}
\bd ( \{ f^{-i} \alpha \}_{i=1+m}^{n+m}, \{ f^{-i} \alpha | A \}_{i=1+m}^{n+m} ) < \ve \,  .
\end{equation}

To prove Proposition~\ref{Bern}, we will follow the arguments in Sections 5 and 6 of \cite{ChH}, 
only indicating
where modifications should be made.

First, we remark that \cite{ChH} decomposes the measure $\musrb$ into conditional measures
on unstable manifolds and a factor measure on the set of unstable leaves.  Due 
to Corollary~\ref{cor:abs cont}, we prefer to decompose $\mu_*$ into conditional 
measures on stable manifolds and the factor measure $\hat\mu_*$.  For this reason, we exchange
the roles of stable and unstable manifolds throughout the proofs of \cite{ChH}.

To this end, we take $f = T^{-1}$ in the set-up presented above, and $X = M$.  
Moreover, we set $\alpha = \cM_{-1}^1$, since this (mod 0) partition generates the full $\sigma$-algebra on $M$.
We will follow the proof of \cite{ChH} to show that $\alpha$ is vwB, and so by
Theorem~\ref{thm:vwB}, $\mu_*$ will be Bernoulli with respect to $T^{-1}$, and therefore with
respect to $T$.  
The proof in \cite{ChH} proceeds in two steps.

\smallskip
\noindent
{\em Step 1.  Construction of $\delta$-regular coverings.}  Given $\delta > 0$, 
the idea is to cover $M$, up to a set of $\mu_*$-measure at most $\delta$, by 
Cantor rectangles of stable and unstable manifolds such that $\mu_*$ restricted to each
rectangle is arbitrarily close to a product measure.  This is very similar to our covering
$\{ R_i \}_{i \in \mathbb{N}}$ from \eqref{fakecover}; however, some adjustments must be made
in order to guarantee uniform properties for the Jacobian of the relevant holonomy map.

On a Cantor rectangle $R$ with $\mu_*(R)>0$, we can define a product measure as 
follows.\footnote{We follow the definition in \cite[Section 5.1]{ChH}, exchanging the roles
of stable and unstable manifolds.}
Fix a point $z \in R$, and consider $R$ as the product of $R \cap W^s(z)$ with $R \cap W^u(z)$,
where $W^{s/u}(z)$ are the local stable and unstable manifolds of $z$, respectively.
As usual, we disintegrate $\mu_*$ on $R$ into conditional measures 
$\mu_{*,R}^W$, on $W \cap R$, where
$W \in \cW^s(R)$, and a factor measure $\hat\mu_*$ on the index set $\Xi_R$ of the curves
$\cW^s(R)$.

Define $\mu_{*,R}^p = \mu_{*,R}^{W^s(z)} \times \hat\mu_*$ and note that we can view
$\hat\mu_*$ as inducing a measure on $W^u(z)$.   Corollary~\ref{cor:abs cont} implies 
that $\mu_{*,R}^p$ is absolutely continuous with respect to $\mu_*$. 
The following definition is taken from \cite{ChH} (as mentioned
above,  a $\delta$-regular covering of $M$ is a collection of rectangles which covers
$M$ \emph{up to a set of measure $\delta$}).

\begin{definition}
\label{def:reg}
For $\delta > 0$, a $\delta$-regular covering of $M$ is a finite collection of disjoint Cantor rectangles
$\cR$ for which,\footnote{The corresponding definition in \cite{ChH} has a third condition, but
this is trivially satisfied in our setting since our stable and unstable manifolds are one-dimensional
and have uniformly bounded curvature.}
\begin{itemize}
  \item[a)]  $\mu_*(\cup_{R \in \cR}R) \ge 1- \delta$.
  \item[b)] Every $R \in \cR$ satisfies $\big| \frac{\mu_{*,R}^p(R)}{\mu_*(R)} - 1\big| < \delta$.
  Moreover, there exists $G \subset R$, with $\mu_*(G) > (1-\delta) \mu_*(R)$, such that
  $\bigl | \frac{d\mu_{*,R}^p}{d\mu_*}(x) - 1\bigr| < \delta$ for all $x \in G$.
\end{itemize}
\end{definition}

By \cite[Lemma~5.1]{ChH}, such coverings exist for any $\delta>0$.  The proof
essentially uses the covering from \eqref{fakecover}, and then subdivides the rectangles into
smaller ones on which the Jacobian of the holonomy between stable manifolds is nearly 1, in order to satisfy item (b) above.
This argument relies on Lusin's theorem and goes through in our setting with no changes.
Indeed, the proof in our case is simpler since the angles between stable and unstable subspaces
are uniformly bounded away from zero, and the hyperbolicity constants in \eqref{eq:hyp}
are uniform for all $x \in M$.

\smallskip
\noindent
{\em Step 2.  Proof that $\alpha = \cM_{-1}^1$ is vwB.}
Indeed, \cite{ChH} prove that any $\alpha$ with piecewise smooth boundary is vwB, but due to
Theorem~\ref{thm:vwB}, it suffices to prove it for a single partition which generates the $\sigma$-algebra
on $M$.  Moreover, using $\alpha = \cM_{-1}^1$ allows us to apply 
the bounds \eqref{nbhd} directly since $\partial \alpha = \cS_1 \cup \cS_{-1}$.

Fix $\ve > 0$, and define
\[
\delta = e^{- (\ve/C')^{2/(1-\gamma)}} \, ,
\]
where $C'>0$ is the constant from \eqref{eq:bad set}.

Let $\cR = \{ R_1, R_2, \ldots, R_k \}$ be a $\delta$-regular cover of $M$ such that
the diameters of the $R_i$ are less than $\delta$.
Define the partition $\pi = \{ R_0, R_1, \ldots, R_k \}$, where $R_0 = M \setminus \cup_{i=1}^k R_i$.
For each $i \ge 1$, let $G_i \subset R_i$ denote the set identified in 
Definition~\ref{def:reg}(b).

Since $T^{-1}$ is K-mixing, there exists an even integer $N = 2m$, such that for any
integers $N_0, N_1$ such that $N < N_0 < N_1$, $\delta$-almost every atom 
$A$ of $\bigvee_{N_0-m}^{N_1-m} T^{-i}\alpha$ satisfies,
\begin{equation}
\label{eq:K}
\left| \frac{\mu_*(R|A)}{\mu_*(R)} -1 \right| < \delta, \qquad \mbox{for all $R \in \pi$}\, .
\end{equation}

Now let $n, N_0, N_1$ be given as above, and define $\omega = \bigvee_{N_0-m}^{N_1-m} T^{-i}\alpha$.
\cite{ChH} proceeds to show that $c\ve$-almost every atom of $\omega$ satisfies \eqref{eq:vwB alt}
with $\ve$ replaced by $c\ve$ for some uniform constant $c>0$.  The first set of estimates in the
proof is to bound the measure of bad sets which must be thrown out, and to show that these
add up to at most $c\ve$.

The first set is $\hat F_1$, which is the union of all atoms in $\omega$, which do not
satisfy \eqref{eq:K}.  By choice of $N$, we have $\mu_*(\hat F_1) < \delta$.

The second set is $\hat F_2$.  Let $F_2 = \cup_{i=1}^k R_i \setminus G_i$, and
define $\hat F_2$ to be the union of all atoms $A \in \omega$, for which either
$\mu_*(F_2|A) > \delta^{1/2}$, or
\[
\sum_{i=1}^k \frac{\mu_{*, R_i}^p(A \cap F_2)}{\mu_*(A)} > \delta^{1/2} \, .
\]
It follows as in \cite[Page 38]{ChH}, with no changes, that $\mu_*(\hat F_2) < c\delta^{1/2}$, for 
some $c>0$ independent of $\delta$ and $k$.

Define $F_3$ to be the set of all points $x \in M \setminus R_0$ such that 
$W^s(x)$ intersects the boundary of the element $\omega(x)$ before it fully crosses
the rectangle $\pi(x)$.  Thus if $x \in F_3$, there exists a subcurve of $W^s(x)$
connecting $x$ to the boundary of $(T^{-i}\alpha)(x)$ for some $i \in [N_0 - m, N_1-m]$.
Then since $\pi(x)$ has diameter less than $\delta$, 
$T^i(x)$ lies within a distance $C_1\Lambda^{-i}\delta$ of the boundary of $\alpha$, where
$C_1$ is from \eqref{eq:hyp}.  Using \eqref{nbhd}, the total measure of such points must add
up to at most
\begin{equation}
\label{eq:bad set}
\sum_{i=N_0-m}^{N_1-m} \frac{C}{|\log(C_1\Lambda^{-i} \delta)|^{\gamma}}
\le C' |\log \delta|^{1-\gamma} \, ,
\end{equation}
for some $C'>0$.  Letting $\hat F_3$ denote the union of atoms $A \in \omega$ such that
$\mu_*(F_3 | A) > |\log \delta|^{\frac{1-\gamma}{2}}$, it follows that 
$\mu_*(\hat F_3) \le C' |\log \delta|^{\frac{1-\gamma}{2}}$.  This is at most $\ve$ by choice
of $\delta$.

Define $F_4$ (following \cite[Section~6.1]{ChH}, and not  \cite[Section 6.2]{ChH}) to be the set of all 
$x \in M\setminus R_0$ for which there exists
$y \in W^u(x) \cap \pi(x)$ such that $h(x,y) > 0$.  This implies that 
$W^u(x)$ intersects the boundary of the element $(T^i \alpha)(x)$ for some
$i \in [1+m, n+m]$, remembering \eqref{eq:vwB alt}, and the definition of $h$.
Using again the uniform hyperbolicity \eqref{eq:hyp}, this implies that $T^{-i}(x)$ lies in 
a $C_1\Lambda^{-i}\delta$-neighbourhood of the boundary of $\alpha$.  Thus the
same estimate as in \eqref{eq:bad set} implies $\mu_*(F_4) \le C'|\log \delta|^{1-\gamma}$.
Finally, letting $\hat F_4$ denote the union of all atoms $A \in \omega$ such that
$\mu_*(F_4 | A) > |\log \delta|^{\frac{1-\gamma}{2}}$, it follows as before that 
$\mu_*(\hat F_4) \le C' |\log \delta|^{\frac{1-\gamma}{2}}$.

Finally, the bad set to be avoided in the construction of the joining $\lambda$ is
$R_0 \cup (\cup_{i=1}^4 \hat F_i)$.  Its measure is less than $c\ve$ by choice of $\delta$.
From this point, once the measure of the bad set is controlled, 
the rest of the proof in Section 6.2 of \cite{ChH} can be repeated verbatim.  This proves that
\eqref{eq:vwB alt} holds for $c\ve$-almost every atom of $\omega$, and thus that $\alpha$
is vwB.
\end{proof}


\subsection{Uniqueness of the measure of maximal entropy}
\label{sec:un}

This subsection is devoted to the following proposition:

\begin{proposition}
\label{prop:unique}
The measure $\mu_*$ is the unique measure of maximal entropy.
\end{proposition}

The proof of uniqueness relies on exploiting the fact that while the lower bound on Bowen balls (or elements of $\cM_{-n}^0$)
cannot be improved for $\mu_*$-almost every $x$, yet if one fixes $n$, most elements of $\cM_{-n}^0$ should either
have unstable 
diameter of a fixed length, or have previously been contained in an element of $\cM_{-j}^0$ with this property for
some $j<n$.  Such elements collectively satisfy stronger lower bounds on their measure.  
Since we have established good control of the elements of $\cM_{-n}^0$ and $\cM_0^n$ 
in the fragmentation lemmas of
Section~\ref{sec:growth}, we will work with these partitions instead of Bowen balls.  

Recalling \eqref{eq:complex}, choose $m_1$ such that $(Km_1+1)^{1/m_1} < e^{h_*/4}$.
Now choose $\delta_2>0$ sufficiently small that for all $n,k \in \mathbb{N}$, if $A \in \cM_{-n}^k$ is
such that
$$\max \{ \diam^u(A), \diam^s(A) \} \le \delta_2\, ,
$$ then $A \setminus \cS_{\pm m_1}$ consists of no more than 
$Km_1+1$ connected components.  

For $n \ge1$, define
\begin{align*}
B_{-2n}^0 = \{ A \in \cM_{-2n}^0 &: \forall \, j, 0 \le j \le n/2, \,\\
&\quad  T^{-j}A \subset E \in \cM_{-n+j}^0 \mbox{ such that } \diam^u(E) < \delta_2 \} \, ,
\end{align*}
with the analogous definition for $B_0^{2n} \subset \cM_0^{2n}$ replacing unstable diameter by stable diameter.
Next, set $B_{2n} = \{ A \in \cM_{-2n}^0 : \mbox{ either $A \in B_{-2n}^0$ or $T^{-2n}A \in B_0^{2n}$ } \}$.
Define $G_{2n} = \cM_{-2n}^0 \setminus B_{2n}$.

Our first lemma shows that the set $B_{2n}$ is small relative to $\# \cM_{-2n}^0$ for large $n$. 
Let $n_1 >2 m_1$ be chosen so that for all $A \in \cM_{-n}^0$, $\diam^s(A) \le C \Lambda^{-n} \le \delta_2$ 
for all $n \ge n_1$.

\begin{lemma}
\label{lem:bad}
There exists $C>0$ such that for all $n \ge n_1$,
\[
\# B_{2n} \le C e^{3n h_*/2}  (Km_1+1)^{\frac{n}{m_1}+1} \le C e^{7n h_*/4} \, .
\]
\end{lemma}

\begin{proof}
Fix $n \ge n_1$ and
suppose $A \in B_{-2n}^0 \subset \cM_{-2n}^0$.  For $0 \le j \le \lfloor n/2 \rfloor$, define $A_j \in \cM_{-\lceil 3n/2 \rceil-j}^0$ 
to be the element containing $T^{-(\lfloor n/2 \rfloor -j)}A$ (note that $T^{-k}A \in \cM_{-2n + k}^k$ for each $k \le 2n$).

By definition of $B^0_{-2n}$ and choice of $n_1$, we have $\max \{ \diam^u(A_j), \diam^s(A_j) \} \le \delta_2$.  Thus
the number of connected components of $\cM_{-\lceil 3n/2 \rceil}^{m_1}$ in $A_0$ is at most $Km_1+1$.
Thus the number of connected components of $T^{m_1}A_0$ (one of which is $A_{m_1}$) is at most $Km_1+1$.
Since the stable and unstable diameters of $A_{m_1}$ are again both shorter than $\delta_2$ (since $A \in B_{-2n}^0$)
and $n >2 m_1$, we may apply this
estimate inductively.  Thus writing $\lfloor n/2 \rfloor = \ell m_1 + i$ for some $i < m_1$, we have that
$\# \{A' \in B_{-2n}^0 : T^{-\lfloor n/2 \rfloor}A' \subset A_0 \} \le (Km_1+1)^{\ell+1}$.  Summing over all possible 
$A_0 \in \cM_{-\lceil 3n/2 \rceil}^0$ yields by Proposition~\ref{cor:exp} and choice of $m_1$,
\[
\# B_{-2n}^0 \le \# \cM_{-\lceil 3n/2 \rceil}^0 (Km_1+1)^{n/m_1+1}
\le C e^{7nh_*/4} \, .
\]
A similar estimate holds for $\# B_0^{2n}$.   Given the one-to-one correspondence between elements of $\cM_{-2n}^0$ and
$\cM_0^{2n}$, it follows that $\# B_{2n} \le 2 \# B_{-2n}^0$, proving the required estimate.
\end{proof}

Next, the following lemma establishes the importance of long pieces in providing good lower bounds on the measure
of partition elements.

\begin{lemma}
\label{lem:lower bound}
There exists $C_{\delta_2} > 0$,
such that for all $j \ge 1$ and all $A \in \cM_{-j}^0$ such that $\diam^u(A) \ge \delta_2$ and $\diam^s(T^{-j}A) \ge \delta_2$,
we have\footnote{It also follows from the proof of Proposition~\ref{prop:max} that the upper bound
$\mu_*(A) \le C e^{-j h_*}$ holds for all $A \in \cM_{-j}^0$ for some constant $C>0$ independent of
$j$ and $\delta_2$, but we shall not need this here.}
\[
\mu_*(A) \ge C_{\delta_2} e^{-j h_*} \, .
\]
\end{lemma}

\begin{proof}
As in the proof of Proposition~\ref{prop:good growth}, by \cite[Lemma~7.87]{chernov}, we may choose finitely many
(maximal) Cantor rectangles, $R_1, R_2, \ldots R_k$, with $\mu_*(R_i) > 0$, and having the property that every unstable 
curve of length at least $\delta_2$ properly crosses at least one of them in the unstable direction, and every stable 
curve of length at least $\delta_2$ properly crosses at least one of them in the stable direction.
Let $\cR_{\delta_2} = \{ R_1, \ldots R_k \}$.

Now let $j \in \mathbb{N}$, and $A \in \cM_{-j}^0$ with $\diam^u(A) \ge \delta_2$ and $\diam^s(T^{-j}A) \ge \delta_2$.  
Notice that $T^{-j}A \in \cM_0^j$.
By construction, $A$ properly crosses one rectangle $R_i \in \cR_\delta$, and $T^{-j}A$ properly crosses another
rectangle $R_{i'} \in \cR_\delta$.  Let $\Xi_i$ denote the index set for the family of stable manifolds comprising $R_i$.
For $\xi \in \Xi_i$, let $W_{\xi, A} = W_\xi \cap A$.  Since $T^{-j}A$ properly crosses $R_{i'}$ in the stable direction and $T^{-j}$
is smooth on $A$, it follows that $T^{-j}(W_{\xi, A})$ is a single curve that contains a stable manifold in the family comprising $R_{i'}$.

Let $\ell_{\delta_2}$ denote the length of the shortest stable manifold in the finite set of rectangles comprising $\cR_{\delta_2}$.
Then using \eqref{eq:unwrap} and \eqref{eq:lower nu}, we have for all $\xi \in \Xi_i$,
\[
\int_{W_{\xi, A}} \nu = e^{-j h_*} \int_{T^{-j}(W_{\xi,A})} \nu \ge  e^{-j h_*}  \bar C \ell_{\delta_2}^{h_* \bar C_2} \, ,
\]
where $\bar C, \bar C_2>0$ are independent of $\delta$ and $j$.

Lastly, denoting by $D(R_i)$ the smallest solid rectangle containing $R_i$ (as in Definition~\ref{MCR}) and
using the fact from the proof of Corollary~\ref{cor:abs cont} that $\mu_*^W$ is equivalent to $\nu$ on $\mu_*$-a.e. 
$W \in \cW^s$, we estimate,
\[
\begin{split}
\mu_*(A) & \ge \mu_*(A \cap D(R_i)) \ge \int_{\Xi_i} \mu_*^{W_\xi}(A) \, d\hat \mu_*(\xi) \\
& \ge C \int_{\Xi_i} \nu(A \cap W_\xi) \, d\hat\mu_*(\xi) \ge C'_{\delta_2} e^{-j h_*} \hat \mu_*(\Xi_i) \, ,
\end{split}
\]
which proves the lemma since the family $\cR_{\delta_2}$ is finite.
\end{proof}

We may finally prove Proposition~\ref{prop:unique}:

\begin{proof} This follows from the previous two lemmas, adapting Bowen's proof  of uniqueness
of equilibrium states
(see the use of \cite[Lemma 20.3.4]{KH} in \cite[Thm 20.3.7]{KH}, as observed
in the proof of  \cite[Thm 6.4]{GL2}, noting that there is no need to check that
boundaries have zero measure).

Since $\mu_*$ is ergodic, it suffices by a standard argument (see e.g. the  beginning
of the proof of \cite[Thm 20.1.3]{KH}) to check that if $\mu$ is a
$T$-invariant probability measure so that there exists a Borel set $F \subset M$
with $T^{-1}(F)= F$ and $\mu_*(F)=0$ but $\mu(F)=1$ (that is, $\mu$ is singular with respect
to $\mu_*$) then
$h_\mu(T)< h_{\mu_*}(T)$.

Observe first that the billiard map $T$ (as well as its inverse
$T^{-1}$) is expansive, that is, there exists $\ve_0>0$ so that
if $d(T^j(x), T^j(y))<\ve_0$ for some $x, y \in M$ and all $j\in \integer$, then $x=y$.
(Indeed, if $x\ne y$ then there is $n\ge 1$ and an
element of either $\cS_n$ or $\cS_{-n}$ that  separates them.  So
$x$ and $y$ get mapped to different sides of a singularity line and by 
\eqref{eq:epsilon 0}
are separated
by a minimum distance $\ve_0$, depending on the table.)

For each $n \in \mathbb{N}$, we consider the partition $\cQ_n$ of maximal connected components of $M$ 
on which $T^{-n}$ is continuous.  By Lemmas~\ref{lem:card} and~\ref{prop:equiv}, $\cQ_n$ is $\cM_{-n}^0$ plus isolated points whose
cardinality grows
at most linearly with $n$.  Thus $G_n \subset \cQ_n$ for each $n$.  Define
$\tilde B_n = \cQ_n \setminus G_n$. The set $\tilde B_n$ contains $B_n$ plus isolated points, and so its cardinality is bounded
by the expression in Lemma~\ref{lem:bad}, by possibly adjusting the constant $C$.

By the uniform hyperbolicity of $T$, the diameters of the elements of $T^{-\lfloor n/2 \rfloor}(\cQ_n)$ 
tend to zero
as $n\to \infty$.   This implies the following fact.

\begin{sublem}
\label{sub:claim}
For each $n \ge n_1$ there exists a finite union $\cC_n$
of elements of $\cQ_n$ so that
\begin{equation*}
\lim_{n\to \infty} (\mu+\mu_*)((T^{- \lfloor n/2 \rfloor}\cC_n)\triangle F)=0\, .
\end{equation*}
\end{sublem}

\begin{proof}
See \cite[Lemma 2]{Bow}:
Let $\bar \mu=\mu + \mu_*$ and $\tilde \cQ_n = T^{- \lfloor n/2 \rfloor}(\cQ_n)$. For $\delta>0$ pick compact sets $K_1\subset F$ and $K_2 \subset M\setminus F$ so
that $\max\{ \bar \mu(F \setminus K_1),
\bar \mu((M \setminus F)\setminus K_2) \} <\delta$. We have $\eta=\eta_\delta:=d(K_1,K_2)>0$.
If $\diam(\tilde Q)<\eta/2$ then either $\tilde Q \cap K_1=\emptyset$ or $\tilde Q\cap K_2=\emptyset$.
Let $n=n_\delta$ be so that the diameter of $\tilde \cQ_n$ is $<\eta_\delta/2$.
Set $\tilde \cC_n=\cup \{ \tilde Q \in \tilde \cQ_n : Q \cap K_1 \ne \emptyset\}$. Then $K_1 \subset \tilde \cC_n$ and
$\tilde \cC_n \cap K_2=\emptyset$. Hence, $\bar \mu(\tilde \cC_n \triangle F)\le \delta + \bar\mu(\tilde \cC_n \triangle K_1)\le \delta +\bar \mu(M \setminus(K_1 \cup K_2))\le 3 \delta$. 
Defining $\cC_n = T^{\lfloor n/2 \rfloor} \tilde \cC_n$ completes the proof.
\end{proof}

Remark that, since
$T^{-1}(F)=F$, it  follows that 
$$(\mu+\mu_*)(\cC_n \triangle F)=(\mu+\mu_*)((T^{ \lfloor n/2 \rfloor}\cC_n)\triangle F)$$
also tends to zero as $n\to \infty$.

Since $\cQ_{2n}$ is generating for $T^{2n}$, we have
\begin{equation*}
h_\mu(T^{2n})=h_\mu(T^{2n}, \cQ_{2n})\le H_\mu(\cQ_{2n})=-
\sum_{Q \in \cQ_{2n}} \mu(Q) \log \mu(Q) \, .
\end{equation*}

By the proof of Sublemma~\ref{sub:claim}, for each $n$, there exists a compact set $K_1(n)$ that
defines $\tilde \cC_n=T^{-\lfloor n/2 \rfloor} \cC_n$, and satisfying $K_1(n) \nearrow F$ as $n \to \infty$.
Next, we group elements $Q \in \cQ_{2n}$ according to whether
$T^{-n}Q \subset \tilde \cC_n$ or $T^{-n} Q \cap \tilde \cC_n = \emptyset$.  Note that
if $Q$ is not an isolated point, and if $T^{-n}Q \cap \tilde \cC_n \neq \emptyset$, then 
$T^{-n}Q \in \cM_{-n}^n$ is contained in an element of $\cM_{- \lfloor n/2 \rfloor}^{\lfloor n/2 \rfloor}$
that intersects $K_1(n)$.  Thus $Q \subset T^n \tilde \cC_n = T^{\lceil n/2 \rceil} \cC_n$.
Therefore,
\begin{align*}
&2n h_\mu(T)= h_\mu(T^{2n})\le -
\sum_{Q \in \cQ_{2n}} \mu(Q) \log \mu(Q)\\
&\qquad \quad\le - \sum_{Q \subset  T^n \tilde \cC_n} \mu(Q) \log \mu(Q)
-  \sum_{Q \in \cQ_{2n} \setminus (T^n \tilde \cC_n)} \mu(Q) \log \mu(Q)\\
&\qquad \quad \le \frac 2 e + \mu(T^n \tilde \cC_n) \log \# (\cQ_{2n} \cap T^n \tilde \cC_n)
+ \mu(M \setminus (T^n \tilde \cC_n)) \log \# (\cQ_{2n} \setminus ( T^n \tilde \cC_n))  \, ,
\end{align*}
where we used in the last line that the convexity of $x \log x$ implies that, for
all $p_j>0$ with $\sum_{j=1}^N p_j \le 1$, we have (see e.g. \cite[(20.3.5)]{KH})
$$
-\sum_{j =1}^N p_j \log p_j \le  \frac 1 e + (\log N) \sum_{j=1}^N  p_j \, .
$$

Then, since 
$- h_{\mu_*}(T)=\big( \mu(T^n\tilde\cC_n)+\mu(M \setminus (T^n \tilde \cC_n)) \big) 
\log e^{-h_*}$,
for $n \ge n_1$, we write
\begin{equation}
\begin{split}
\label{eq:manipulate}
2n &(h_\mu(T)- h_{\mu_*}(T))- \frac 2 e\\
&\le \mu(T^n \tilde \cC_n)\log \sum_{Q \in \cQ_{2n}: Q \subset T^n \tilde\cC_n}  e^{-2nh_*}
+\mu(M \setminus (T^n \tilde \cC_n)) \log \sum_{Q \in \cQ_{2n} \setminus (T^n \tilde \cC_n)} 
e^{-2nh_*} \\
& \le  \mu(\cC_n)\log \left( \sum_{Q \in G_{2n}: Q \subset T^n \tilde \cC_n}  e^{-2nh_*}
+ \sum_{Q \in \tilde B_{2n}: Q \subset T^n \tilde \cC_n}  e^{-2nh_*} \right) \\
& \qquad + \mu(M \setminus \cC_n) \log \left( \sum_{Q \in G_{2n} \setminus (T^n \tilde\cC_n) } e^{-2nh_*}
+ \sum_{Q \in \tilde B_{2n} \setminus (T^n \tilde \cC_n) } e^{-2nh_*} \right) \, ,
\end{split}
\end{equation}
where we have used the invariance of $\mu$ in the last inequality.
By Lemma~\ref{lem:bad}, both sums over elements in $\tilde B_{2n}$ are bounded by $Ce^{-nh_*/4}$.   It remains to estimate
the sum over elements of $G_{2n}$.

First we provide the following characterization of elements of $G_{2n}$.  Let $Q \in G_{2n} \subset \cM_{-2n}^0$.  
Since $Q \notin B_{-2n}^0$, there exists 
$0 \le j \le \lfloor n/2 \rfloor$ such that $T^{-j}Q \subset E_j \in \cM_{-2n +j}^0$
and $\diam^u(E_j) \ge \delta_2$.  We claim that there exists $k \le \lfloor n/2 \rfloor$
and $\bar E \in \cM_{-2n+j+k}^0$ such that $E_j \subset \bar E$ and 
$\diam^s(T^{-2n+j+k} \bar E) \ge \delta_2$.

The claim follows from the fact that $T^{-2n}Q \notin B_0^{2n}$.  Thus there exists $k \le \lfloor n/2 \rfloor$ such that
$T^{-2n+k}Q \subset \tilde E_k \in \cM^{2n-k}_0$ with $\diam^s(\tilde E_k) \ge \delta_2$.  But notice that
$T^{-2n+j+k}E_j \in \cM_{-k}^{2n-j-k}$ contains $T^{-2n+k}Q$.  Thus letting $\tilde E$ denote the unique
element of $\cM^{2n-j-k}_0$ containing both $T^{-2n+j+k}E_j$ and $\tilde E_k$, we define
$\bar E = T^{2n-j-k}\tilde E \in \cM_{-2n+j+k}^0$, and $\bar E$ has the required property since
$T^{-2n+j+k}\bar E \supset \tilde E_k$.

By construction, $\bar E$ satisfies the assumptions of Lemma~\ref{lem:lower bound} since $\bar E \in \cM_{-2n+j+k}^0$ with
$\diam^u(\bar E) \ge \delta_2$, and $\diam^s(T^{-2n + j+k}\bar E) \ge \delta_2$.
Thus,
\begin{equation}
\label{goodb}
\mu_*(\bar E) \ge C_{\delta_2} e^{(-2n + j + k)h_*} \, .
\end{equation}

We call
$(\bar E, j ,k)$ an {\em admissible triple} for $Q \in G_{2n}$ if
$0 \le j, k \le \lfloor n/2 \rfloor$ and $\bar E \in \cM_{-2n + j + k}^0$, with
$T^{-j}Q \subset \bar E$ and $\min \{ \diam^u(\bar E), \diam^s(T^{-2n +j +k}\bar E) \} \ge \delta_2$.
Obviously, there may be many admissible triples associated to a given $Q \in G_{2n}$; however,
we define the unique {\em maximal triple} for $Q$ by taking first the maximum 
$j$, and then the maximum $k$ over all admissible triples for $Q$.  

Let $\cE_{2n}$  be the set of
maximal triples obtained in this way from elements of $G_{2n}$.  For $(\bar E, j, k) \in \cE_{2n}$,
let $\cA_M(\bar E, j,k)$ denote the set of $Q \in G_{2n}$ for which $(\bar E, j, k)$ is the
maximal triple.
The importance of the set $\cE_{2n}$
lies in the following property.

\begin{sublem}
\label{sub:disjoint}  
Suppose that $(\bar E_1, j_1, k_1)$, $(\bar E_2, j_2, k_2) \in \cE_{2n}$ with $j_2 \ge j_1$
and $(\bar E_1, j_1, k_1) \neq (\bar E_2, j_2, k_2)$.  
Then $T^{-(j_2-j_1)}\bar E_1 \cap \bar E_2 = \emptyset$.
\end{sublem}

\begin{proof}
Suppose, to the contrary, that there exist $(\bar E_1, j_1, k_1)$, $(\bar E_2, j_2, k_2) \in \cE_{2n}$ with 
$j_2 \ge j_1$ and $T^{-(j_2-j_1)}\bar E_1 \cap \bar E_2 \neq \emptyset$.  
Note that $T^{-(j_2-j_1)}\bar E_1 \in \cM_{-2n + j_2 + k_1}^{j_2-j_1}$ while
$\bar E_2 \in \cM_{-2n + j_2 + k_2}^0$.  

Thus if $k_1 \le k_2$, then $T^{-(j_2-j_1)}\bar E_1 \subset \bar E_2$, and so $(\bar E_1, j_1, k_1)$
is not a maximal triple for all $Q \in \cA_M(\bar E_1, j_1, k_1)$, a contradiction.

If, on the other hand, $k_1 > k_2$, then both $T^{-(j_2-j_1)}\bar E_1$ and $\bar E_2$ are contained
in a larger element $\bar E' \in \cM_{-2n + j_2 + k_1}^0$.  Since $\bar E' \supset \bar E_2$,
we have $\diam^u(\bar E') \ge \delta_2$, and since 
$T^{-2n+j_2+k_1}\bar E' \supset T^{-2n+j_1+k_1}\bar E_1$, we have
$\diam^s(T^{-2n+j_2+k_1}\bar E') \ge \delta_2$.  Thus neither $(\bar E_1, j_1, k_1)$ nor
$(\bar E_2, j_2, k_2)$ is a maximal triple, also a contradiction.
\end{proof}

Note that by definition, if $Q \in T^n\tilde \cC_n \cap \cA_M(\bar E, j, k)$, then 
$T^{-n+j} \bar E \in \cM_{-n+k}^{n-j}$ contains $T^{-n}Q$.  Also, since $j, k \le \lfloor n/2 \rfloor$,
$T^{-n+j} \bar E$ is contained in the same element of $\cM_{-\lfloor n/2 \rfloor}^{\lfloor n/2 \rfloor}$
that contains $T^{-n}Q$ and intersects $K_1(n)$.  Thus $T^{-n+j}\bar E \subset \tilde \cC_n$
whenever $T^n \tilde \cC_n \cap \cA_M(\bar E, j ,k) \neq \emptyset$.  This
also implies that $\cA_M(\bar E, j, k) \subset T^n\tilde \cC_n$ whenever
$T^n \tilde \cC_n \cap \cA_M(\bar E, j ,k) \neq \emptyset$.

Next, for a fixed $(\bar E, j, k) \in \cE_{2n}$, 
by submultiplicativity, since $\bar E \in \cM_{-2n + j +k}^0$ and $G_{2n} \subset \cM_{-2n}^0$, we
have $\# \cA_M(\bar E, j, k) \le  \# \cM_0^{j+k}$. 
Now using Proposition~\ref{cor:exp} and \eqref{goodb}, we estimate
\begin{align*}
&\sum_{Q \in G_{2n}: Q \subset T^n \tilde\cC_n}  e^{-2nh_*}
 \le \sum_{(\bar E, j, k) \in \cE_{2n} : \bar E \subset T^{n-j}\tilde\cC_n} \sum_{Q \in \cA_M(\bar E, j, k)} e^{-2nh_*} \\
&\qquad  \le \sum_{(\bar E, j, k) \in \cE_{2n} : \bar E \subset T^{n-j} \tilde\cC_n}  C e^{(-2n+j+k)h_*} 
\le \sum_{(\bar E, j, k) \in \cE_{2n} : \bar E \subset T^{n-j}\tilde \cC_n} C' \mu_*(\bar E) \\
& \qquad \le \sum_{(\bar E, j,k) \in \cE_{2n} : \bar E \subset T^{n-j}\tilde \cC_n} C' \mu_*(T^{-n+j}\bar E)
\le C' \mu_*(\tilde \cC_n) = C' \mu_*(\cC_n) \, ,
\end{align*}
where the constant $C'$ depends on $\delta_2$, but not on $n$.  We have also
used that $T^{-n+j_1}\bar E_1 \cap T^{-n+j_2} \bar E_2 = \emptyset$ for all distinct triples
$(\bar E_1, j_1, k_1), (\bar E_2, j_2, k_2) \in \cE_{2n}$, by Sublemma~\ref{sub:disjoint},
in order to sum over the elements of $\cE_{2n}$.
A similar bound holds for the sum over $Q \in G_{2n} \setminus (T^n\tilde \cC_n)$
since $T^{-n+j} \bar E \subset M \setminus \tilde \cC_n$ whenever
$T^n \tilde \cC_n \cap \cA(\bar E, j, k) = \emptyset$.  
Putting these bounds together allows us to complete our estimate of 
\eqref{eq:manipulate},
\[
\begin{split}
2n(h_\mu(T) - h_{\mu_*}(T)) - \frac 2e
& \le \mu(\cC_n) \log \left( C' \mu_*(\cC_n)  + C e^{-nh_*/4} \right) \\
& \qquad + \mu(M \setminus \cC_n) \log \left( C' \mu_*(M \setminus \cC_n)  + C e^{-nh_*/4}  \right) \, .
\end{split}
\]
Since $\mu(\cC_n)$ tends to $1$ as $n\to \infty$ while $\mu_*(\cC_n)$ tends to 
0 as $n\to \infty$ the limit of the right-hand side tends to $- \infty$.  
This yields a contradiction unless $h_\mu(T)< h_{\mu_*}(T)$.
\end{proof}


\bibliographystyle{amsplain}

\end{document}